\documentclass[11pt]{amsart}
\usepackage{amsmath}
\usepackage{mathtools}
\usepackage{amssymb}
\usepackage{amsthm}
\usepackage{mathrsfs}
\usepackage[pagebackref,colorlinks=true,linkcolor=blue,citecolor=blue]{hyperref}
\usepackage{tikz-cd}

\usepackage{etaremune}

\usepackage[english]{babel}
\usepackage{graphicx}
\usepackage{caption}
\usepackage{float}

\usepackage{bbding} 

\UseRawInputEncoding

\tikzset{
  symbol/.style={
    draw=none,
    every to/.append style={
      edge node={node [sloped, allow upside down, auto=false]{$#1$}}}
  }
}

\newcommand{\BC}{{\mathbb {C}}}

\newcommand{\BR}{{\mathbb {R}}}

\newcommand{\CC}{{\mathcal {C}}}

\newcommand{\FS}{{\mathfrak {S}}}

\newcommand{\Z}{\mathbb{Z}}

\newcommand{\R}{\mathbb{R}}

\newcommand{\Ind}{{\mathrm{Ind}}}

\newcommand{\Tau}{{\mathcal T}}

\newcommand{\ol}{\overline}

\newcommand{\SL}{\mathrm{SL}}
\newcommand{\GL}{\mathrm{GL}}
\newcommand{\SO}{\mathrm{SO}}
\newcommand{\Sp}{\mathrm{Sp}}
\newcommand{\St}{\mathrm{St}}
\newcommand{\RO}{\mathrm{O}}

\newcommand{\RG}{\mathrm{G}}

\newcommand{\Irr}{\mathrm{Irr}}

\newcommand{\temp}{\text{temp}}

\newcommand{\Rep}{\underline{\mathrm{Rep}}}

\newcommand{\half}[1]{\frac{#1}{2}}

\newcommand{\Temp}[3]{T_{#1,#3}^{#2}}

\newcommand{\comment}[1]{}
\newcommand{\EE}{\mathcal{E}}
\newcommand{\FF}{\mathcal{F}}

\newtheorem{thm}{Theorem}[section]
\newtheorem{cor}[thm]{Corollary}
\newtheorem{lemma}[thm]{Lemma}
\newtheorem{prop}[thm]{Proposition}
\newtheorem {conj}[thm]{Conjecture}

\newtheorem {ques/conj}[thm]{Question/Conjecture}
\newtheorem {ques}[thm]{Question}
\newtheorem{defn}[thm]{Definition}
\newtheorem{remark}[thm]{Remark}

\newtheorem{exmp}[thm]{Example}

\newtheorem{algo}[thm]{Algorithm}
\newtheorem{conv}[thm]{Convention}

\newtheorem*{globalcond*}{Global Condition}
\newtheorem*{localcond*}{Local Condition}

\newtheorem*{globalconj*}{Global Conjecture}
\newtheorem*{localconj*}{Local Conjecture}

\newtheorem*{nonzero*}{Conjecture on the non-vanishing of the normalized intertwining operators}
\newtheorem*{holo*}{Conjecture on the holomorphicity of the normalized intertwining operators}

\numberwithin{equation}{section}

\let\oldbullet\bullet
\renewcommand{\bullet}{{\vcenter{\hbox{\tiny$\oldbullet$}}}}

\begin{document}

\title[Arthur representation and the Unitary Dual]{On Arthur representations and the unitary dual}

\author[A. Hazeltine]{Alexander Hazeltine
}
\address{Department of Mathematics\\
University of Michigan\\
Ann Arbor, MI 48109, USA}
\email{ahazelti@umich.edu}

\author[D. Jiang]{Dihua Jiang}
\address{School of Mathematics, University of Minnesota, Minneapolis, MN 55455, USA}
\email{dhjiang@math.umn.edu}

\author[B. Liu]{Baiying Liu}
\address{Department of Mathematics\\
Purdue University\\
West Lafayette, IN, 47907, USA}
\email{liu2053@purdue.edu}

\author[C.-H. Lo]{Chi-Heng Lo}
\address{Department of Mathematics\\
National University of Singapore\\
119076, Singapore}
\email{{ch\_lo@nus.edu.sg}}

\author[Q. Zhang]{Qing Zhang}
\address{School of Mathematics and Statistics, Huazhong University of Science and Technology, Wuhan, 430074, China}
\email{qingzh@hust.edu.cn}


\date{\today}


\thanks{The research of the second-named author is partially supported by the Simons Grants: SFI-MPS-SFM-00005659 and SFI-MPS-TSM-00013449. The research of the third-named author is partially supported by the NSF Grant DMS-1848058 {and the Simons Foundation: Travel Support for Mathematicians}. The research of the fifth-named author is partially supported by NSFC Grant 12371010.}

\begin{abstract}
In this paper, we propose a new conjecture describing the structure of the unitary dual in terms of Arthur representations for connected reductive algebraic groups defined over any non-Archimedean local field of characteristic zero. This conjecture provides a candidate set for the unitary dual, constructed from Arthur representations. 
For classical groups, we develop an explicit algorithm to generate this candidate set. Evidence for its exhaustiveness includes compatibility with the known generic unitary dual, unramified unitary dual, and low-corank representations. As further support, we verify the conjecture for the unitary dual of the exceptional group of type $\RG_2$. 
\end{abstract}

\maketitle

\section{Introduction}

For any locally compact topological group $G$, one of the fundamental problems in abstract harmonic analysis is the description of the set of equivalence classes of its irreducible unitary representations.
This set is known as the {\it unitary dual} of $G$ and is denoted by $\widehat{G}$ {or $\Pi_{u}(G)$}.
When $G$ is abelian, the relation between $G$ and $\widehat{G}$ is given by the Pontryagin duality. 
When $G$ is a compact connected Lie group, the unitary dual is understood by the highest weight theory of H. Weyl.
In general, the structure of the unitary dual is closely related to the theory of Quantum 
Mechanics and remains one of the major mysterious problems in Mathematics. 
The study of the unitary dual of reductive algebraic groups over locally compact topological fields 
(local fields) is also motivated by the Langlands Program, which aims to understand the theory of automorphic forms and their connections to Algebraic Geometry and Number Theory. 

Our objective is to understand the unitary dual problem by the 
classification and characterization of unitary representations of connected reductive algebraic groups over local fields. More precisely, let $F$ be a local field of characteristic zero and $\RG$ be a reductive algebraic group defined over $F$. Denote by $G=\RG(F)$, the $F$-rational points. Let $\Pi(G)$ be the 
admissible dual of $G$ which is the set of equivalence classes of its irreducible admissible representations. The conjectured local Langlands correspondence for $G$ asserts the parametrization of $\Pi(G)$ 
by means of the enhanced local $L$-parameters. There are two general goals related to the unitary dual problem for $G$:
\begin{enumerate}
    \item develop algorithms to determine the unitarity of representations in $\Pi(G)$ based on their enhanced local $L$-parameters, via the local Langlands conjecture; and 
    \item describe the structures of the unitary dual from the known classification theory for $\Pi(G)$. 
\end{enumerate}

The celebrated Langlands classification determines the admissible dual $\Pi(G)$ by means of the tempered dual $\Pi_\temp(G)$ of $G$, which is the subset of $\Pi(G)$ consisting of tempered members. 
When $G$ is a real reductive Lie group, there is a well-known Atlas computer program that determines 
the unitarity of representations in $\Pi(G)$ via an algorithm (see \cite{Atl}). However, 
there is no general approach available to determine the unitary dual $\Pi_u(G)$ based on the Langlands classification of the admissible dual $\Pi(G)$ when $F$ is non-Archimedean. On the other hand, the far-reaching  
endoscopic classification conjecture of J. Arthur (\cite{Art89}) produces local Arthur packets (Definition \ref{A parameters}), which are expected to be finite subsets of the unitary dual $\Pi_u(G)$. Denote by $\Pi_A(G)$ the subset of $\Pi(G)$ consisting of Arthur representations, also referred to as representations of Arthur type, which is the union of all local Arthur packets. 
It is strongly desirable that the whole unitary dual $\Pi_u(G)$ can be exhausted via various expected constructions from the Arthur representations $\Pi_A(G)$. 

In this paper, we assume that $F$ is a non-Archimedean local field of characteristic zero and define a set $\Pi_{\overline{A}}^{\lim}(G)$ from Arthur representations $\Pi_A(G)$, which is constructed from $\Pi_A(G)$ by applying the standard constructions of unitary representations: unitary parabolic induction, complementary series, and limits of complementary series, as explained in Definition \ref{def closure A}. Then, we conjecture that this set is exactly the unitary dual. 
Loosely speaking, the idea is that \emph{complimentary series can be obtained via irreducible continuous Hermitian deformations from irreducible unitary inductions} (see  Remark \ref{rmk Tadic construction}(2) for more precise statements).
Similar idea has been clearly reflected in the construction of the unitary dual for general linear groups (see \cite{Tad86}). More precisely, 
$\Pi_A(GL_n)$ consists of parabolic induction of generalized Speh representations (see \eqref{generalized Speh representation} and \eqref{u rho a b}):
$$\bigtimes_i u_{\rho_i}(a_i,b_i), \text{ where } \rho_i \text{ are unitary supercuspidal; }$$
while $\Pi_u(GL_n)$ consists of parabolic induction of irreducible continuous Hermitian deformations of generalized Speh representations as follows:
$$\bigtimes_i u_{\rho_i}(a_i,b_i) \times \bigtimes_j (u_{\rho_j}(a_j,b_j)\lvert \cdot \rvert^{x_j} \times u_{\rho_j}(a_j,b_j)\lvert \cdot \rvert^{x_j}), \text{ where } \rho_i, \rho_j \text{ are unitary supercuspidal, } 0 \leq x_j < \frac{1}{2}.$$
For general connected reductive groups, we carry out the above idea and inductively define the set 
$\Pi_{\overline{A}}^{\lim}(G)$ from Arthur representations.

The following is our main conjecture on the unitary dual, which is a modification and simplification of our previously announced conjecture 
(\cite[Conjecture 1.1]{HJLLZ24}).  

\begin{conj}[Unitary Dual]\label{conj A+ intro}
Let $G$ be a general connected reductive group defined over $F$. Assume that the local Arthur conjecture as in \cite[Conjecture 6.1]{Art89} holds for every Levi subgroup $M$ of $G$. Then the following two sets are equal:
    $$\Pi_{\overline{A}}^{\lim}(G) = \Pi_{u}(G).$$
\end{conj}

The significance of Conjecture \ref{conj A+ intro} can be explained as follows. 
The unitary dual $\Pi_{u}(G)$ admits a simple definition within representation theory, yet its classification and internal structure remain largely mysterious. By contrast, according to the conjectures of J. Arthur (\cite{Art89, Art13}), the set of Arthur representations $\Pi_{A}(G)$ 
is defined through local stability and local twisted endoscopic transfer, providing a rich and highly structured framework; however, its intrinsic representation-theoretic nature is not well understood. This conjecture mediates between these mysteries by proposing a candidate set $\Pi_{\overline{A}}^{\lim}(G)$ for the unitary dual, constructed from $\Pi_{A}(G)$ and designed to be structurally informative and algorithmically accessible. Indeed, a primary motivation of this paper is that $\Pi_{\overline{A}}^{\lim}(G)$ should be computable, opening the possibility of an algorithmic determination of the unitary dual itself. In support of this, we present an explicit algorithm (Algorithm \ref{alg A bar}) that generates $\Pi_{\overline{A}}^{\lim}(G)$ for symplectic and split odd special orthogonal groups. 
We also verify Conjecture \ref{conj A+ intro} for the unitary dual of the exceptional group of type $\RG_2$ (Proposition \ref{proposition:G2}). We summarize these results as our main theorem in the following. 

\begin{thm}\label{thm main}\
\begin{enumerate}
    \item Let $G=G_n$, the symplectic or split odd special orthogonal group of rank $n$. Then Algorithm \ref{alg A bar} generates 
    the set $\Pi_{\overline{A}}^{\lim}(G)$, hence a candidate set for $\Pi_{u}(G)$.
    \item Conjecture \ref{conj A+ intro} holds for the unitary dual of the exceptional group of type $\RG_2$ (Proposition \ref{proposition:G2}), if the local Arthur conjecture as in \cite[Conjecture 6.1]{Art89} holds.
\end{enumerate}
\end{thm}

Algorithm \ref{alg A bar} has been implemented using Sagemath. For classical groups, Conjecture \ref{conj A+ intro} is straightforward for the known generic unitary dual (\cite{LMT04}) and unramified unitary dual (\cite{BM96, Bar10, MT11}), based on their classifications. Following Algorithm \ref{alg A bar}, for symplectic and split odd special orthogonal groups $G_n$, in \cite{HJLLZ24}, the authors verified Conjectures \ref{conj A+ intro} for unitary representations of corank up to 3 (see Definition \ref{def corank}), based on the classification by M. Tadi\'c  in \cite{Tad23}. Recently, in \cite{LLW}, the third and fourth-named authors and B. Wen verified Conjecture \ref{conj A+ intro} for unitary representations of corank 4, which completes the classification of the unitary dual for $\Sp_8$ and split $\SO_9$.  

\begin{thm}[{\cite{LLW}}]
    For $\mathrm{G}=\Sp_8$ or split $\SO_9$, Conjecture \ref{conj A+ intro} holds, i.e.,
    $\Pi_{\overline{A}}^{\lim}(G) = \Pi_{u}(G).$
\end{thm}

Conjecture \ref{conj A+ intro} has the following interesting consequences. 

\begin{cor}[{Corollary \ref{cor implication of 1.1} and Remark \ref{rmk-general-classical}}]\label{cor implication of 1.1 intro}
Assume that Conjecture \ref{conj A+ intro} holds for a quasi-split classical group $G$. Then the following hold.
    \begin{enumerate}
    \item If $G=G_n$, the symplectic or split odd special orthogonal group of rank $n$, then Tadi{\'c}'s two hypotheses (Conjectures \ref{conj Tadic preservation}, \ref{conj independence}) hold for $\Pi_u(G_n)$.
\item The unitary dual $\Pi_{u}(G)$ is preserved under Aubert-Zelevinsky involution.
\end{enumerate}
\end{cor}

We remark that indeed for any classical group $G$, assuming that the local Arthur conjecture as in \cite[Conjecture 6.1]{Art89} holds, then Corollary \ref{cor implication of 1.1 intro}(2) also holds, with similar arguments.

Notice that by the work of  
C. M{\oe}glin (see Theorem \ref{thm red from nu to gp}), for classical groups, $\Pi_A(G)$ can be constructed from a special family of representations that are of good parity (Definition \ref{def critical type}). Let $\Pi_{gp}(G)$ denote the set of good parity representations of $G$. We define
\[
\Pi_{A, \, gp}(G)=\Pi_{gp}(G)\cap\Pi_A(G), \quad
{\rm and}
\quad
\Pi_{u, \, gp}(G)=\Pi_{gp}(G)\cap\Pi_u(G).
\]
Towards a representation-theoretical description of Arthur representations and a refinement of Conjecture \ref{conj A+ intro}, the following conjecture for classical groups (symplectic, orthogonal, and unitary) was proposed in \cite{HJLLZ24}.

\begin{conj}[{\cite[Conjecture 1.2]{HJLLZ24}}]\label{conj main}
For any classical group $G$, assuming that the local Arthur conjecture as in \cite[Conjecture 6.1]{Art89} holds, then the following holds:
\begin{equation*}
\Pi_{A, \, gp}(G) =\Pi_{u, \, gp}(G).
\end{equation*}
In other words, for any $\pi \in \Pi_{gp}(G)$, $\pi$ is of Arthur type if and only if it is unitary.
\end{conj}

This is an extension of a conjecture of M. Tadi{\'c} (\cite[Conjecture 1.1]{Tad22}). In \cite{HJLLZ24}, we discussed several applications of this conjecture and verified it for unitary representations of $G_n$ with corank up to 3.  
Recently, H. Atobe and A. M{\'i}nguez proved Conjecture \ref{conj main} for symplectic and split odd orthogonal groups using a different idea (\cite{AM25}). The case for even (special) orthogonal groups is proved in \cite{HLL25} as an application of the local theta correspondence. 

\begin{thm}[\cite{AM25, HLL25}]\label{thm AM}
Conjecture \ref{conj main} holds for symplectic, split odd special orthogonal, and even (special) orthogonal groups. 
\end{thm}

\subsection{More evidence for Conjecture \ref{conj A+ intro}}

Recall that $G_n=\Sp_{2n}$ or $\SO_{2n+1}$ (a split symplectic or odd special orthogonal group). Let $\Psi_{1/2}^+(G_n)$ be the set of complementary local Arthur parameters of $G_n$ (see 
Definition \ref{A parameters}) which may not be bounded and can be twisted up to $\frac{1}{2}$ when restricted to the Weil group, and we let $\Psi(G_n)$ denote the subset consisting of parameters that are bounded when restrict to Weil group.
For each $\psi \in \Psi_{1/2}^+(G_n)$, the corresponding local packet $\Pi_\psi$ can be defined as in \cite[Theorem 2.2.1, \S 1.5 and \S 8.3]{Art13}, which will be explained in \S \ref{local A packets}. 
We define $\Pi_{A}(G_n):=\cup_{\psi\in\Psi(G_n)}\Pi_{\psi}$ and $\Pi_{A+}(G_n):=\cup_{\psi\in\Psi^{+}_{1/2}(G_n)}\Pi_{\psi}$.
Representations in $\Pi_{A}(G_n)$ and $\Pi_{A+}(G_n)$ are called Arthur and complementary Arthur representations, respectively. It is clear by Definition \ref{def closure A} that, as a subset of the unitary dual, $ \Pi_{A}(G_n) \subset \Pi_{\overline{A}}^{\lim}(G_n)$. 

To guarantee that the complementary Arthur representations in $\Pi_{A+}(G_n)$ are Hermitian, it is necessary to introduce a smaller subset $\Psi^{+}_{\text{unit}}(G_n)$ of $\Psi^+_{1/2}(G_n)$ as in \cite[Chapter 8]{Art13} (see \eqref{eq def Psi^+_unit}). According to \cite[Conjecture 8.3.1]{Art13}, the complementary local packet $\Pi_{\psi}$, with $\psi \in \Psi^{+}_{\text{unit}}(G_n)$, 
should consist of irreducible unitary representations. However, our recently found examples indicate that this assertion is not generally true. We refer to Example \ref{ex A+}(1) for more details.  Therefore, we define
\[ 
\Pi_{A+,\, u}(G_n):=\Pi_{A+}(G_n) \cap \Pi_u(G_n).
\]
In \cite{HJLLZ25}, we gave a characterization of $\Pi_{A+,\, u}(G_n)$ (independently given in \cite{AM25}). 
As further evidence for Conjecture \ref{conj A+ intro}, we show that 
indeed $ \Pi_{A+,\, u}(G_n) \subset \Pi_{\overline{A}}^{\lim}(G_n)$ (Theorem \ref{cor closure A+,u=A,u}).

\subsection{New algorithms on Arthur representations} \label{new algorithms}

In this paper, we also give a framework to inductively classify $\Pi_{A}(G_n)$ (see Algorithm \ref{alg Pi A}), which is used in the development of Algorithm \ref{alg A bar}. 
The key ingredients are two new algorithms (Algorithms \ref{algo Arthur type} and \ref{algo Arthur type 2}) that determine whether a representation of $G_n$ is of Arthur type from its enhanced Langlands parameter. 
While 
the classification by corank of tempered representations is based on the local Langlands correspondence in \cite{Art13} (see Theorem \ref{thm Arthur tempered}) and the computation of derivatives for tempered representations in  
\cite[Theorem 3.6]{Ato22a}, the classification of non-tempered representations is based on the above two new algorithms and the algorithm in \cite[\S 8]{HLL22} on exhausting all the local Arthur packets containing a given Arthur representation. 

Previously, there are two other different algorithms for $G_n$, 
independently given by H. Atobe (\cite[Algorithm 3.3]{Ato23}),
 and by the first, third, and fourth-named authors (\cite[Algorithm 7.9]{HLL22}). More precisely, the algorithm given by Atobe constructed several representations $\pi_i$ of smaller ranks, and used the information from the derivatives of $\pi$ and $\pi_i$ together with the construction of $\Psi(\pi_i)=\{\psi \ | \ \pi_i \in \Pi_{\psi}\}$. The other algorithm given by the first, third, and fourth-named authors only used the derivative information of $\pi$ to construct a local Arthur parameter $\psi$ such that $\pi$ is of Arthur type if and only if $\pi \in \Pi_{\psi}$. We remark that both algorithms end up with constructing a possibly non-tempered packet $\Pi_{\psi}$ and checking whether $\pi$ or $\pi_i$ is in this packet.

On the other hand, the new algorithms (Algorithms \ref{algo Arthur type} and \ref{algo Arthur type 2}) only use the construction of $\Psi(\pi_i)$, where $\pi_i$ are several representations of smaller ranks constructed from $\pi$ without using its highest derivatives.
Moreover, there is no construction of any non-tempered local Arthur packet involved. This difference allows us to extend the new algorithms for {quasi-split special even orthogonal groups} in \cite{HLL25}, and so are the results in this paper, as we expect. \\

The unitary dual problem was previously known for the following cases. For $\GL_n$, the unitary dual was classified by M. Tadi{\'c} in \cite{Tad85a, Tad86, Tad09b} over local fields, and by D. Vogan in \cite{Vog86} for Archimedean local fields. For $\GL_n(\mathcal{D})$ with $\mathcal{D}$ a central division algebra over non-Archimedean local fields of characteristic zero, the unitary dual of $\GL_n(\mathcal{D})$ was classified by the work of A. Badulescu, D. Renard \cite{BR04} and  V. S\'echerre \cite{Sec09}, based on Tadi{\'c}'s work \cite{Tad90}. For the case of non-Archimedean local fields of positive characteristic, the unitary dual of $\GL_n(\mathcal{D})$ was  completed by  A. Badulescu,  G. Henniart,  B. Lemaire, and  V. S\'echerre in \cite{BHLS10}. 
For recent uniform and simplified approaches on the unitary dual of the general linear group and its inner forms over non-Archimedean local 
fields, see \cite{Tad15, LM16}. 
For these known cases, assuming that there is a theory of local Arthur packets, we expect that Conjecture \ref{conj A+ intro} holds. 
For instance, when $G=\GL_n(F)$ over a non-Archimedean local field $F$ of characteristic zero, as commented above, it is an easy exercise to show that the classification of the unitary dual by M. Tadi\'c (\cite{Tad86}) can be reformulated as the statement that $\Pi_{A+,\,u}(\GL_n(F))=\Pi_u(\GL_n(F))=\Pi_{\ol{A}}^{\lim}(\GL_n(F))$ holds. There are special families of unitary representations known through the work of D. Barbasch and D. Ciubotaru (\cite{Bar10, BC05, BC09, Ciu05, Ciu08}). It is an interesting problem to see the connection of these families with Conjecture \ref{conj A+ intro}. \\

The following is the structure of this paper. In \S \ref{sec nota and prel}, we recall some necessary notation and preliminaries.
In \S \ref{sec proof of main thm}, we discuss several conjectures of Tadi{\'c} about the unitary dual. In \S \ref{sec A-packet}, we recall the notation and construction of local Arthur packets. 
In \S \ref{sec A+}, we define the set $\Pi_{\overline{A}}^{\lim}(G)$ in Conjecture \ref{conj A+ intro} and explore the implications of this conjecture.
In \S \ref{sec new algorithm}, for symplectic and split odd special orthogonal groups $G_n$, we introduce the new algorithms (Algorithms \ref{algo Arthur type} and \ref{algo Arthur type 2}) to determine whether a representation $\pi$ of $G_n$ is of Arthur type.
In \S \ref{sec inductive approach}, based on the algorithms given in \S \ref{sec new algorithm}, we describe an inductive approach to classify  representations of Arthur type (Algorithm \ref{alg Pi A}).
In \S\ref{algorithm for A bar}, based on the algorithms given in \S \ref{sec new algorithm} and \S \ref{sec inductive approach}, we develop an algorithm to determine the candidate set $\Pi_{\overline{A}}^{\lim}(G_n)$ of the unitary dual (Algorithm \ref{alg A bar}).
Finally, in \S \ref{sec:G2}, we verify the unitary dual conjecture (Conjecture~\ref{conj A+ intro}) for the exceptional group of type $\RG_2$ (Proposition \ref{proposition:G2}).

\subsection*{Acknowledgements} 
The authors would like to thank Freydoon Shahidi for the interest and constant support. The authors would like to thank James Arthur and Bin Xu for helpful communications on Example \ref{ex A+}. The authors also would like to thank
Jeffrey Adams, Dan Ciubotaru, Chris Jantzen, Lucas Mason-Brown, Marko Tadi{\'c}, and David Vogan for helpful comments, discussions, and suggestions. 
Part of this paper was written when the third and fourth-named authors were attending the Relative Langlands Program workshop and conference at the Institute for Mathematical Sciences, the National University of Singapore (January 2026), where we discussed the result of this paper with Peter Sarnak. The authors appreciate very much the warm hospitality of the Institute and the mathematics department of NUS, the deep interest and the very inspiring discussions and suggestions from Peter Sarnak.

\section{Notation and Preliminaries}
\label{sec nota and prel}

\subsection{Langlands classification}

The Langlands classification is one of the fundamental classifications in the representation theory of 
reductive algebraic groups over local fields (\cite{Sil78,BW00,Kon03}). We recall a more explicit version of the 
Langlands classification for the classical groups considered in this paper over  
a non-Archimedean local field $F$ of characteristic zero, which may be called a $p$-adic local field in the rest of this paper.  

Let $n$ be a positive integer. We consider the general linear group $\GL_n(F)$ and the symplectic group or split odd special orthogonal group, which are denoted by $G_n=\RG_n(F)$ with $\RG_n=\Sp_{2n}$ or $\SO_{2n+1}$, respectively. Fix a Borel subgroup of $\GL_n(F)$ and consider a standard parabolic subgroup $P$ with Levi subgroup $M\cong \GL_{n_1}(F)\times\cdots\times \GL_{n_r}(F).$ For $\tau_i\in\Pi(\GL_{n_i}(F)),$ the normalized parabolic induction is denoted by 
\[
\tau_1\times\cdots\times\tau_r := \mathrm{Ind}_{P}^{\GL_n(F)}(\tau_1\otimes\cdots\otimes\tau_r).
\]
Let $\rho$ be an irreducible unitary supercuspidal representation of $\GL_n(F).$ 
A segment $[x,y]_\rho$ is the set of (not necessarily unitary) supercuspidal representations of 
the form
$$
[x,y]_\rho=\{\rho\vert\cdot\rvert^x, \rho\vert\cdot\rvert^{x-1},\dots,\rho\vert\cdot\rvert^y\},
$$
where $x,y\in\mathbb{R}$ such that $x-y$ is a non-negative integer. Here $\vert\cdot\rvert$ denotes the composition of the normalized $p$-adic absolute value of $F$ with the determinant of $\GL_n(F)$. The unique irreducible subrepresentation of $\rho\vert\cdot\rvert^x\times\cdots\times\rho\vert\cdot\rvert^y$ is called the ({essentially}) Steinberg representation attached to $[x,y]_\rho$ and is denoted by $\Delta_\rho[x,y].$ 
For convenience, we set $\Delta_\rho[x,x+1]$, with $y=x+1,$ to be the trivial representation of $\GL_0(F).$ When the context is unambiguous, we refer to both $[x,y]_\rho$ and $\Delta_\rho[x,y]$ simply as segments.

The Langlands classification for $\GL_n(F)$ classifies the admissible dual $\Pi(\GL_n(F))$. Specifically, any irreducible admissible representation $\tau$ of $\GL_n(F)$ can be realized as the unique irreducible subrepresentation of a parabolic induction of the form
\[\Delta_{\rho_1}[x_1,y_1]\times\cdots\times\Delta_{\rho_r}[x_r,y_r],\]
where $\rho_i$ is an irreducible unitary supercuspidal representation of $\GL_{n_i}(F),$ $[x_i,y_i]_{\rho_i}$ is a segment, and $x_1+y_1\leq\cdots\leq x_r+y_r.$ In this situation, we write
$$
\tau=L(\Delta_{\rho_1}[x_1,y_1],\dots,\Delta_{\rho_r}[x_r,y_r]).
$$
Let $(x_{i,j})_{1\leq i\leq s, 1\leq j \leq t}$ be real numbers such that $x_{i,j}=x_{1,1}-i+j.$ A generalized Speh representation is an irreducible representation of  the form
\begin{equation}\label{generalized Speh representation}
\begin{pmatrix}
x_{1,1} & \cdots & x_{1,t} \\
\vdots & \ddots & \vdots \\
x_{s,1} & \cdots & x_{s,t}
\end{pmatrix}_{\rho}:=L(\Delta_{\rho}[x_{1,1},x_{s,1}],\dots,\Delta_{\rho}[x_{1,t}, x_{s,t}]).
\end{equation}
These representations appear in the construction of local Arthur packets of bad parity (see \S\ref{local A packets}). 

Next, we discuss the Langlands classification of $G_n.$ Fix a Borel subgroup of $G_n$ and let $P$ be a standard parabolic subgroup of $G_n$ with Levi subgroup $M\cong \GL_{n_1}(F)\times\cdots\times \GL_{n_r}(F)\times G_{m},$ where $G_m$ is a group of the same type as $G_n$, but of possibly lower rank. Given smooth representations $\tau_i$ of $\GL_{n_i}(F)$ for $i=1,2,\dots,r$ and a smooth representation $\sigma$ of $G_{m}$, the normalized parabolic induction is denoted by
$$
\tau_1\times\cdots\times\tau_r\rtimes\sigma := \mathrm{Ind}_{P}^{G_n}(\tau_1\otimes\cdots\otimes\tau_r\otimes\sigma).
$$
The Langlands classification for $G_n$ states that every irreducible admissible representation $\pi$ of $G_n$ is a unique irreducible subrepresentation of 
\[\Delta_{\rho_1}[x_1,y_1]\times\cdots\times\Delta_{\rho_r}[x_r,y_r]\rtimes\pi_{temp},\]
where $\rho_i$ is an irreducible unitary supercuspidal representation of $\GL_{n_i}(F),$ $x_1+y_1\leq\cdots\leq x_r+y_r<0,$ and $\pi_{temp}$ is an irreducible tempered representation of $G_{m}.$ In this situation, we write
$$
\pi=L(\Delta_{\rho_1}[x_1,y_1],\dots,\Delta_{\rho_r}[x_r,y_r];\pi_{temp}),
$$
and say that the tuple $(\Delta_{\rho_1}[x_1,y_1],\dots,\Delta_{\rho_r}[x_r,y_r];\pi_{temp})$ is the Langlands data, or $L$-data, of $\pi.$  We refer to the multi-set $\{\Delta_{\rho_1}[x_1,y_1],\dots,\Delta_{\rho_r}[x_r,y_r]\}$ as the non-tempered portion of the $L$-data of $\pi.$ 

We also compare the $L$-data of various representations as follows. Given an irreducible admissible representation 
\[
\pi=L(\Delta_{\rho_1}[x_1,y_1],\dots,\Delta_{\rho_r}[x_r,y_r];\pi_{temp}),
\]
we can obtain a new representation by ``inserting" another segment $\Delta_{\rho_{r+1}}[x_{r+1},y_{r+1}]$ with $x_{r+1}+y_{r+1}<0$. That is, we consider the irreducible admissible representation
\[
L(\Delta_{\rho_1}[x_1,y_1],\dots,\Delta_{\rho_r}[x_r,y_r],\Delta_{\rho_{r+1}}[x_{r+1},y_{r+1}];\pi_{temp}),
\]
where the data is possibly re-indexed appropriately, i.e.,  $x_1+y_1\leq\cdots\leq x_r+y_r\leq x_{r+1}+y_{r+1}<0.$ We adopt the following convention.

\begin{conv}\label{conv L-data}
    We do not distinguish an irreducible admissible representation from its $L$-data. If $\pi'$ is obtained from $\pi$ by inserting a collection of segments $\{\Delta_{\rho_1}[x_1,y_1],\ldots, \Delta_{\rho_f}[x_f,y_f]\}$, then we write    
    \[\pi'=\pi+\{\Delta_{\rho_1}[x_1,y_1],\ldots, \Delta_{\rho_f}[x_f,y_f]\}.\]
\end{conv}

\subsection{Arthur's parametrization of the tempered spectrum 
}\label{sec temp spectrum}

We recall from \cite[\S8.2]{Bor79} the local Langlands correspondence and from \cite[Theorem 1.5.1]{Art13} the Arthur's parametrization for $G_n$. 

For a split reductive algebraic group $\RG$, the Langlands dual group of $\RG$ is denoted by $\RG^\vee(\BC)$. 
A local $L$-parameter of $G_n$ is a 
${\RG}^\vee_n(\BC)$-conjugacy class of an admissible homomorphism $\phi:W_F\times\SL_2(\BC)\rightarrow{\RG}^\vee_n(\BC)$ (more generally, see \cite[\S8.2]{Bor79}). We do not distinguish $\phi$ from its conjugacy class. We say that $\phi$ is tempered if its restriction to $W_F$ has bounded image. The component group of $\phi$ is defined by
\[
\mathcal{S}_{\phi}:= \pi_0( \mathrm{Cent}_{{\RG}^\vee_n(\BC)}(\mathrm{Im}(\phi))/Z({\RG}^\vee_n(\BC))^{\Gamma}).
\]
The Pontryagin dual $\widehat{\mathcal{S}}_\phi$ of $\mathcal{S}_{\phi}$ is described explicitly below. 

By the local Langlands correspondence for $\GL_n(F)$ (\cite{Hen00, HT01, Sch13}), we identify an irreducible supercuspidal representation $\rho$ of $\GL_n(F)$ with its local $L$-parameter $\phi_\rho$, which is an irreducible 
$n$-dimensional representation of the Weil group $W_F$. Following this convention, we write 
any tempered local $L$-parameter $\phi$ in the following form:
\[
\phi=\bigoplus_{i=1}^r \rho_i\otimes S_{a_i},
\]
where $\rho_i$ are irreducible unitary supercuspidal representations of $\GL_{n_i}(F)$ 
and $S_k$ denotes the unique irreducible representation of $\SL_2(\BC)$ of dimension $k$. {In order that $\phi$ has image in the complex dual group $G^{\vee}$, if $\rho_i\otimes S_{a_i}$ is not selfdual or selfdual but preserves a bilinear form of opposite parity from $G^{\vee}$, then there must exist another $j \in \{1,\ldots,r\}\setminus \{i\}$ such that $\rho_j\otimes S_{a_j} \cong (\rho_i\otimes S_{a_i})^{\vee}$. Let $\mathrm{Jord}(\phi)$ 
be the multi-set consisting of all the irreducible summands occurring in $\phi$ (counting multiplicities) that are selfdual and preserves a bilinear form of the same parity as $G^{\vee}$.
} Then, $\widehat{\mathcal{S}}_\phi$ is a finite abelian group consisting of characters which may be identified with functions $\varepsilon:\mathrm{Jord}(\phi)\rightarrow\{\pm 1\}$ such that
$\varepsilon(\rho_i\otimes S_{a_i})=\varepsilon(\rho_j\otimes S_{a_j})$ whenever $\rho_i\otimes S_{a_i}\cong\rho_j\otimes S_{a_j}$, and 
\[
\prod_{\rho\otimes S_a\in \mathrm{Jord}(\phi)} \varepsilon(\rho\otimes S_a)^{m_{\rho,a}}=1,
\]
where $m_{\rho,a}$ denotes the multiplicity of $\rho\otimes S_a$ in $\phi.$ 

The conjectural local Langlands correspondence for a general reductive group $G$ associates each local $L$-parameter $\phi$ a finite set of irreducible admissible representations satisfying certain properties (see \cite{Bor79}), which is called the local $L$-packet attached to $\phi$ and is denoted by $\Pi_\phi.$ 
Moreover, $\Pi_\phi$ should be in bijection with $\widehat{\mathcal{S}}_\phi.$ For $G_n$, this follows from Arthur's parameterization of tempered representations.

\begin{thm}[{\cite[Theorem 1.5.1]{Art13}}]\label{thm Arthur tempered}
Fix a choice of Whittaker datum for $G_n$ and let $\phi$ be a tempered local $L$-parameter. Then there is a bijective map between the tempered local $L$-packet $\Pi_{\phi}$ and $\widehat{\mathcal{S}}_\phi.$
\end{thm}
Hereinafter, we implicitly fix a choice of Whittaker datum for $G_n.$ When $\phi$ is tempered, we write $\pi(\phi,\varepsilon)$ for the element of $\Pi_\phi$ corresponding to $\varepsilon\in\widehat{\mathcal{S}}_\phi$ via the bijection in Theorem \ref{thm Arthur tempered}.
Following the notation in \cite{AM20}, when 
\[
\phi=\bigoplus_{i=1}^m \rho_i\otimes S_{a_i},
\]
is tempered, we let $x_{i}=\frac{a_i-1}{2}$ for $i=1,\dots, m$, and rewrite $\pi(\phi,\varepsilon)$ 
as (we set $\varepsilon(\rho_i \otimes S_{a_i})=0$ if $ \rho_i \otimes S_{a_i} \not\in \textrm{Jord}(\phi)$)
\begin{equation*}
\pi((x_{1},\rho_1)^{\varepsilon(\rho_1 \otimes S_{a_1})},\dots, (x_m,\rho_m)^{\varepsilon(\rho_m \otimes S_{a_m})}).
\end{equation*}
When $\rho=\rho_1=\cdots=\rho_m$, we often ignore it in the notation and simply rewrite $\pi(\phi,\varepsilon)$ as
\begin{equation*}
\pi(x_{1}^{\varepsilon(\rho \otimes S_{a_1})},\dots,x_{m}^{\varepsilon(\rho \otimes S_{a_m})}).
\end{equation*}

\subsection{Supercuspidal representations and reducibility points}

In this subsection, we recall from M{\oe}glin's work (\cite{Moe11a}) the characterization of local $L$-parameters for irreducible supercuspidal representations of 
$G_n$, along with key results on the reducibility points of parabolic induction with supercuspidal data. Recall that $\mathcal{C}$ (resp. $\mathcal{C}_{cl}$) is the set of supercuspidal representations of general linear groups (resp. classical groups), and 
$\mathcal{C}^{u}:= \{\rho \in \mathcal{C}\ | \ \rho \text{ is unitary}\}$ and $
    \mathcal{C}^{sd}:= \{\rho \in \mathcal{C}^u\ | \ \rho \text{ is selfdual}\}.$
{
Let 
\[  \phi= \bigoplus_{i\in I} \rho_i \otimes S_{a_i}\]
be a tempered $L$-parameter. We write $\rho \otimes S_{a} \subset \phi$ if $\rho \otimes S_a= \rho_i\otimes S_{a_i}$ for some $i \in I$. We say $\phi$ is of \emph{good parity} if every summand of $\phi$ is selfdual and preserves a bilinear form of the same parity as $G^{\vee}$. We say $\phi$ is \emph{discrete} if the $\rho_i \otimes S_{a_i}$'s are selfdual and pairwise non-equivalent; hence automatically of good parity. We say $\phi$ is \emph{without gaps} if for $a \geq 1$, }
\[ \rho \otimes S_{a+2} \subset \phi \Rightarrow \rho \otimes S_{a} \subset \phi. \]
Applying Theorem \ref{thm Arthur tempered}, M{\oe}glin's parametrization of irreducible supercuspidal representations of $G_n$ is as follows.

\begin{thm}[{\cite[Theorem 1.5.1]{Moe11a}, \cite[Theorem 3.3]{Xu17a}}] \label{thm characterizatioin of supercuspidal}
An irreducible tempered representation $\pi(\phi,\varepsilon)$ of $G_n$ is supercuspidal if and only if the following hold.
\begin{enumerate}
    \item [$\oldbullet$] As a tempered local $L$-parameter, $\phi$ is discrete and without gaps.
    \item [$\oldbullet$] If both $\rho \otimes S_{a}\subset \phi$ and $\rho \otimes S_{a+2} \subset\phi$, then $\varepsilon(\rho \otimes S_{a}) \varepsilon(\rho \otimes S_{a+2})=-1.$
    \item [$\oldbullet$] If $\rho \otimes S_{2}\subset\phi$, then $ \varepsilon(\rho \otimes S_2)=-1$.
\end{enumerate}
\end{thm}

Let $\rho\in\mathcal{C}^u$ and $\sigma\in\mathcal{C}_{cl}$. If $\rho$ is not selfdual, then $\rho\vert\cdot\rvert^x\rtimes\sigma$ is irreducible for all $x\in\mathbb{R}.$ If $\rho$ is selfdual, then there exists a unique $\alpha_{\rho,\sigma}\in\mathbb{R}_{\geq 0}$, known as the reducibility point, such that $\rho\vert\cdot\rvert^{\alpha_{\rho,\sigma}}\rtimes\sigma$ is reducible (\cite{Sil80}). Moreover, $\alpha_{\rho,\sigma}\in\frac{1}{2}\mathbb{Z}_{\geq 0}$.
According to \cite[Remark 4.5.2]{MW06}, we have an explicit description of $\alpha_{\rho,\sigma}$ based on the local $L$-parameter $\phi_{\sigma}$ of $\sigma$, as follows. 

By Theorem \ref{thm characterizatioin of supercuspidal}, we may write
\begin{align}\label{eq decomp phi_sc}
    \phi_{\sigma}= \bigoplus_{\rho \in R}  \bigoplus_{i=0}^{a_{\rho}} \rho \otimes S_{2(i+\epsilon_{\rho})+1},
\end{align}
where $R$ is a finite set of $\mathcal{C}^{sd}$ and $a_{\rho} \in \Z_{\geq 0}$.
Here $\epsilon_{\rho}\in \{0,\half{1}\}$ based on the parity of $\rho$. More explicitly, if $G_n$ is an orthogonal (resp. symplectic) group, then $\epsilon_{\rho}= 0$ (resp. $\epsilon_{\rho}=\half{1}$) if the image of the homomorphism $\rho: W_{F} \to \GL(V)$ preserves a symplectic bilinear form on $V$, and $\epsilon_{\rho}= \half{1}$ (resp. $\epsilon_{\rho}=0$) if the image of $\rho$ preserves a symmetric bilinear form on $V$. Set $a_{\rho}=-1$ if $\rho$ is not in the finite set $R$. Then the reducibility point $\alpha_{\rho,\sigma}$ is given by $a_{\rho}+\epsilon_{\rho}+1$ and the decomposition \eqref{eq decomp phi_sc} can be rewritten as 
\begin{equation}\label{eq-scparameter}
\phi_{\sigma}= \bigoplus_{\rho \in R}  (\rho\otimes S_{2 \epsilon_{\rho}+1}+\rho\otimes S_{2 \epsilon_{\rho}+3}+\cdots + \rho\otimes S_{2 (\alpha_{\rho,\sigma}-1)+1}).
\end{equation}

Most of our analysis proceeds based on the notion of the coranks of representations, whose definition is recalled below.

\begin{defn}\label{def corank}
A representation $\pi\in \Pi(G_n)$ is said to be of corank $r$ if there exists an injection
$\pi \hookrightarrow \rho_1 \times \cdots \times \rho_r \times \pi_{sc},$ 
where $\rho_i \in \mathcal{C}$ and $\pi_{sc}\in \mathcal{C}_{cl}$.
\end{defn}

\subsection{Fell topology on the unitary dual and properties}\label{Fell top}
In this subsection, for the convenience of the readers, we briefly recall the notion of Fell topology on $ \Pi_u(G)$ and useful properties for any locally compact topological group $G$ (see \cite{Fe60}). Our main references are \cite{Tad88, Fo95, Vog07}.

For $f\in L^1(G)$, define a norm (see \cite[Corollary 7.2]{Fo95} for example)
$$ \lVert f\rVert_*:=\sup_{\pi\in \Pi_u(G)} \lVert \pi(f)\rVert.$$
Here $\lVert \pi(f)\rVert $ is the norm of $\pi(f)$ as an operator from $V_\pi$ to $V_\pi$, where $V_\pi$ denotes the (Hilbert) space of $\pi$. It is known that $\lVert \pi(f)\rVert\le \lVert f\rVert_1.$ Thus, $\lVert f\rVert_*\le \lVert f\rVert_1.$ Let $C^*(G)$ be the completion of $L^1(G)$ with respect to the norm $\lVert \cdot\rVert_*$. The algebra operations on $L^1(G)$ can be extended to $C^*(G)$. This makes $C^*(G)$ a $C^*$-algebra, called the group $C^*$-algebra of $G$. Moreover, there is a one-to-one correspondence between unitary representations of $G$ and nondegenerate $*$-representations of $C^*(G)$. For $\pi\in \Pi_u(G)$, we can view it as a representation of $C^*(G)$ and define 
$$\ker(\pi)=\{f\in C^*(G)\mid \pi(f)=0\}.$$
Then $\ker(\pi)$ is a two-sided ideal of $C^*(G)$. Ideals of this form are called primitive ideals of $C^*(G)$ and set
$$\mathrm{Prim}(G)=\{\ker(\pi)\mid \pi\in \Pi_u(G)\}.$$

For a non-empty subset $U\subset \mathrm{Prim}(G)$, we set 
$$ \overline{U}=\left\{\mathcal{I}\in \mathrm{Prim}(G) \mid \bigcap_{\mathcal{J}\in U}\mathcal{J}\subset \mathcal{I} \right\}.$$
Moreover, set $\overline{\emptyset}=\emptyset$. Then there is a unique topology on $\mathrm{Prim}(G)$ such that the closure of $U$ with respect to this topology is $\overline{U}$ defined above. This is called the Jacobson topology or hull-kernel topology.  The Fell topology on $\Pi_u(G)$ is defined to be the topology on $\Pi_u(G)$ obtained by pulling back the Jacobson topology on $\mathrm{Prim}(G)$ through the natural surjective map $\Pi_u(G)\to \mathrm{Prim}(G)$. 

In the following, we assume that $\RG$ is a connected reductive group over a $p$-adic field $F$ and let $G=\RG(F)$. Then the Fell topology on $\Pi_u(G)$ is equivalent to one of the following (see \cite[Sections 2 and 4]{Tad88}):
\begin{enumerate}
\item A sequence of representations $(\pi_n)_n \subset \Pi_u(G)$ is convergent to $\pi\in \Pi_u(G)$ if and only if for every vector $\xi$ in the Hilbert space of $\pi$, there exists a sequence of vectors $(\xi_n)_n$, where $\xi_n$ is in the Hilbert space of $\pi_n$ such that $c_{\xi_n,\xi_n}$ converges to $c_{\xi,\xi}$ uniformly on any compact subset of $G$. Here $c_{\xi,\xi}$ (respectively, $c_{\xi_n,\xi_n}$) is the matrix coefficient $c_{\xi,\xi}(g)=(\pi(g)\xi,\xi)$ (respectively, $c_{\xi_n,\xi_n}(g)=(\pi_n(g)\xi_n,\xi_n)$). Here, $(\cdot,\cdot)$ denotes the inner product on the Hilbert space of $\pi$.
\item A sequence of representations $(\pi_n)_n \subset \Pi_u(G)$ is convergent to $\pi\in \Pi_u(G)$ if and only if for any two vectors $\xi, \eta\in \pi$, there  exist sequences $\xi_n,\eta_n\in \pi_n$ such that $c_{\xi_n,\eta_n}\to c_{\xi,\eta}$ uniformly on any compact subset of $G$. Here $c_{\xi,\eta}$ is the matrix coefficient $c_{\xi,\eta}(g)=(\pi(g)\xi,\eta)$.
\end{enumerate}

Let $\Pi(G)$ be the admissible dual of $G$. Then $\Pi_u(G)$ can be viewed as a subset of $\Pi(G)$. On $ \Pi(G)$ one can also define a topology using matrix coefficients as in (1) or (2) above. Then the Fell topology on $\Pi_u(G)$ can be regarded as a subspace topology on $\Pi(G)$ (see \cite[P.158]{Tad88}) and $\Pi_u(G)$ is a closed subset of $ \Pi(G)$ by \cite[Theorem 5.13]{Tad88}.
 Let $P=MN$ be a parabolic subgroup of $G$ with Levi $M$ and unipotent $N$. Let $\mathrm{Unr}(M)$ be the group of all unramified characters of $M$. It is naturally isomorphic to a direct product of finite copies of $\BC^\times$ as topological groups.  Then we have the following facts:

\begin{prop}\label{facts on unitary reps}\
Let $\tau$ be an irreducible admissible representation of $M$, and for any $\chi\in \mathrm{Unr}(M)$ let $I(\chi,\tau):=\Ind_P^G(\chi\otimes \tau)$.
    \begin{enumerate}

\item $($\cite[p.16]{MT11}$)$ If $\tau$ is unitary, then $I(1\otimes \tau)$ is unitary, where $1$ is the trivial character of $M$. Since $I(1\otimes \tau)$ is of finite length, it is semisimple and each of its direct summand is also unitary.
    
    \item $($\cite[Theorem 2.2]{Tad88}$)$ Let $\Omega(G)$ be the variety of cuspidal data of $G$. Let $\nu: \Pi_u(G)\to \Omega(G)$ be the map that sends $\pi$ to its cuspidal data. Then $\nu$ is continuous. 
\item $($\cite[Lemma 4.4]{Tad88}$)$ Let $I_\tau$ be the subset of $\mathrm{Unr}(M)$ consisting of $\chi$ such that $ I(\chi,\tau)$ is irreducible. Then the map $I_\tau\to \Pi(G), \chi\mapsto I(\chi,\tau)$ is continuous.

\item $($\cite[p.16]{MT11}$)$ Let $X\subset \mathrm{Unr}(M)$ be a connected set such that for any $\chi\in X $, the induced representation $I(\chi,\tau)$ is Hermitian and irreducible. If there exists $\chi_0\in M$ such that $I(\chi_0, \tau)$ is unitary, then $I(\chi,\tau)$ is unitary for any $\chi\in X.$

\item $($\cite[Theorem 2.5]{Tad88}$)$ The set $\chi\in\mathrm{Unr}(M)$ such that $I(\chi,\tau)$ has a unitarizable subquotient is compact.
\item $($\cite[Theorem 2.7]{Tad86}$)$ Let $S\subset \mathrm{Unr}(M)$ be a connected subset such that for all $\chi\in S$, $ I(\chi,\tau)$ is irreducible and unitarizable. Then for any $\chi\in \overline{S}$, any irreducible subquotient of $I(\chi,\tau)$ is unitarizable.
\item $($\cite[Theorem 3]{Vog07}$)$ Let $\sigma\in \Pi_u(G)$ be a non-isolated representation, and let $\{\pi_n\}\subset \Pi_u(G)$ be a sequence  distinct from $\sigma$ such that $\pi_n$ converges to $ \sigma$ in the Fell topology. Then there is a subsequence $\{\pi_{n_j}\}$ in $\Pi_u(G)$, a parabolic subgroup $P=MN$ of $G$, an irreducible admissible representation $\rho$ of $M$, and a sequence of unramified characters $\chi_j$ of $M$ with the following properties. 
\begin{enumerate}
\item The sequence $\chi_j$ converges to the trivial character.
\item The induced representation $\Ind_{P}^G(\rho\otimes \chi_j)$ is infinitesimally equivalent to $\pi_{n_j}$.
\item The representation $\sigma$ is a composition factor of $\Ind_P^G(\rho)$.
\end{enumerate}
In this case, the limit points of $\{\pi_{n_j}\} $ are all the composition series of $\Ind_P^G(\rho)$.
\end{enumerate}
\end{prop}

\section{Conjectures on the Unitary Dual} \label{sec proof of main thm} 

In this section, we recall a collection  of conjectures of Tadi{\'c} on the unitary dual problem.

\subsection{Representations of good parity or critical type}
 
First, we recall the definitions of representations being \emph{null-parity}, \emph{good parity}, \emph{bad parity}, and \emph{critical type}. 

\begin{defn}\label{def critical type}
    Let $\pi$ be an irreducible admissible representation of $G_n$, which is a subquotient of 
    $\rho_1\vert\cdot\rvert^{x_1} \times \cdots \times \rho_r\vert\cdot\rvert^{x_r} \rtimes \pi_{sc},$ 
    where $\rho_i \in \mathcal{C}^u$, $x_i\in \R_{\geq 0}$ and $\pi_{sc} \in \mathcal{C}_{cl}$.
    \begin{enumerate}
        \item  Define $\pi$ to be of {\emph{null parity}}
        $($named as ``ugly" in \cite{AM20}$)$ if there exists $1 \leq i \leq r$, such that either $\rho_i \not\in \mathcal{C}^{sd}$ or $x_i \not\in \half{1}\Z$.
        \item Define $\pi$ to be of \emph{good parity} if for any $1 \leq i \leq r$, $\rho_i \in \mathcal{C}^{sd}$ and $x_i \in \alpha_{\rho_i,\pi_{sc}}+ \Z$, where $\alpha_{\rho_i,\pi_{sc}} \in\mathbb{R}_{\geq 0}$ is the reducibility point of
    the pair $(\rho_i, \pi_{sc})$.
        \item Define $\pi$ to be of \emph{bad parity} if $\pi$ is of neither null-parity nor good parity.
        \item Define $\pi$ to be of \emph{critical type} if it is of good parity and for each $\rho \in \mathcal{C}^{sd}$, the set (not multi-set) 
        $\{ x_i \ | \ \rho_i \cong \rho \}$ 
        is either empty or
        forms a segment containing $\alpha_{\rho,\pi_{sc}}$. 
    \end{enumerate}
\end{defn}

Note that $\alpha_{\rho_i,\pi_{sc}}+\Z =\epsilon_{\rho_i}+\Z $ only depends on the type of $\rho_i$ and the group $G_n$, not on $\pi_{sc}$. 
Recall that $\epsilon_{\rho_i}$ is defined in \eqref{eq decomp phi_sc}. A local $L$-parameter $\phi$ of $G_n$ is defined to be of good parity if the local $L$-packet $\Pi_{\phi}$ contains one representation of good parity. {It is known that $\phi$ is of good parity if and only of all representations in $\Pi_{\phi}$ are of good parity.}

\subsection{Preservation and independence of unitarizability}

In this subsection, we recall the Jantzen decomposition from \cite{Jan97} and two problems on unitary representations proposed in \cite{Tad18}.

We begin by recalling the following notation for certain subsets of the set \( \mathcal{C} \) of irreducible supercuspidal representations of general linear groups (see, for example, \cite{Jan97, Tad18}).

\begin{defn}\label{def Jantzen decomposition}
Let $X$ be a subset of $\mathcal{C}$.
\begin{enumerate}
    \item [(1)] $X$ is self-contragredient if for any $\rho \in X$, the contragredient of $\rho$ is also in $X$.
    \item [(2)] For an irreducible admissible representation $\beta$ of $\GL_d(F)$, we say that $\beta$ is supported on $X$ if the supercuspidal support of $\beta$ is contained in $X$.
    \item [(3)] Let $\pi$ be an irreducible admissible representation of $G_n$ that appears as an irreducible subquotient of 
    $\theta_1 \times \cdots \times \theta_f \rtimes \sigma,$
    where $\theta_i \in \mathcal{C}$ and $\sigma \in \mathcal{C}_{cl}$.
     We say that $\pi$ is supported on $X\cup \{\sigma\}$ for some self-contragredient $X\subseteq \mathcal{C}$ if $\theta_i \in X$ for $i=1,\dots, r$. 
    \item [(4)] Fix a supercuspidal representation $\sigma$ of $G_m$ and a self-contragredient subset $X\subseteq \mathcal{C}$. We denote by $\Irr(X;\sigma)$ the set of irreducible admissible representations $\pi$ of $G_n$ with $n\geq m$ such that $\pi$ is supported on $X \cup \{\sigma\}$. 
    \item [(5)] Suppose that $X$ is self-contragredient. Let $X= X_1 \sqcup X_2$ be a partition of $X$. 
    Such a partition is called regular if $X_1$ is self-contragredient and 
    \[ \theta \in X_1 \Longrightarrow \theta \vert\cdot\rvert^{1} \not\in X_2.  \]
    That is, there is no reducibility among $X_1$ and $X_2$. 
\end{enumerate}
\end{defn}

Next, we recall certain results on the Jantzen decomposition.

\begin{thm}[{\cite[Theorem 9.3]{Jan97}}]\label{thm Jantzen} 
Let $X$ be a subset of $\mathcal{C}$ that is self-contragredient and $X=X_1 \sqcup X_2$ is a regular partition. Let $\sigma \in \mathcal{C}_{cl}$.
\begin{enumerate}
    \item [(i)] For any $\pi \in  \Irr(X;\sigma)$, there exist irreducible admissible representations $\beta_1, \beta_2$ of general linear groups and $ X_1(\pi),X_2(\pi)$ of lower rank classical groups such that 
    \begin{enumerate}
        \item [$\oldbullet$] $\beta_i$ is supported on $X_i$ and $X_i(\pi)$ is supported on $X_i \cup \{\sigma\}$; and
        \item [$\oldbullet$] there are injections
        \[ \pi \hookrightarrow \beta_1 \rtimes X_2(\pi)\quad {\rm and}\quad  \pi \hookrightarrow \beta_2 \rtimes X_1(\pi). \]
    \end{enumerate}
    The representations $X_i(\pi)$ are uniquely determined by $\pi$.
    \item [(ii)] The map 
    \begin{align*}
        \Irr(X; \sigma )\longrightarrow \Irr(X_1;\sigma) \times \Irr(X_2; \sigma ),\quad {\rm with}\ 
         \pi \mapsto  (X_1(\pi), X_2(\pi))
    \end{align*}
    is a bijection. We denote the inverse map by $\Psi_{X_1,X_2}$.
    \item [(iii)] For $i=1,2$, let $\beta_i$ be an irreducible admissible representation of a general linear group supported on $X_i$ and $\gamma_i\in \Irr(X_i; \sigma)$. Here we allow $\beta_i$ to be the trivial representation of $\GL_0(F)$. Then 
    \[ (\beta_1 \times \beta_2)\rtimes \Psi_{X_1,X_2}(\gamma_1,\gamma_2) \text{ is irreducible} \Longleftrightarrow \text{ both }\beta_i \rtimes \gamma_i \text{ are irreducible.} \]
    \item [(iv)] For any  $\pi \in \Irr(X;\sigma)$, $\pi$ is tempered (resp. square-integrable) if and only if $X_1(\pi), X_{2}(\pi)$ are both tempered (resp. square-integrable). 
\end{enumerate}
\end{thm}

Recall that $\mathcal{C}^{sd}$ is the subset of $\mathcal{C}$ consisting of selfdual supercuspidal representations.
We consider the cuspidal lines $X_{\rho} \subseteq \mathcal{C}$ defined as follows. 

\begin{defn}\label{def X_rho}
 For each $\rho\in \mathcal{C}^{sd}$, define
    $X_{\rho}= \{ \rho\vert\cdot\rvert^{x} \ | \ x \in \R \},$ called \emph{the cuspidal line of $\rho$}. 
    Then $\mathcal{C}= X_{\rho} \sqcup (\mathcal{C}- X_{\rho})$ is a regular partition. For each 
    $\pi\in\Pi(G_n)$, define $X_{\rho}(\pi)$ as in Theorem \ref{thm Jantzen}(i).    
\end{defn}

Before stating the reduction of the unitary dual considered in \cite{Tad09a, Tad18}, we first recall the definition of weakly real.

\begin{defn}
    We say that $\pi \in \Pi(G_n)$ is \emph{weakly real} if there exists a finite subset $\{\rho_1,\dots, \rho_r\}$ of $\mathcal{C}^{sd}$ and a supercuspidal representation $\sigma$ of $G_m$ such that $\pi$ is supported on $(\sqcup_{i=1}^r X_{\rho_i}) \cup \{\sigma\}$.
\end{defn}

Let $\pi$ be a general irreducible admissible representation of $G_n$. Let $X^{sd}:= \sqcup_{\rho \in \mathcal{C}^{sd}} X_{\rho}$. We may apply Theorem \ref{thm Jantzen} to the regular decomposition $\mathcal{C}= X^{sd} \sqcup (\mathcal{C}- X^{sd})$ to obtain an injection
\[ \pi \hookrightarrow \theta \rtimes X^{sd}(\pi), \]
where $\theta $ is supported on $\mathcal{C}- X^{sd}$ and $X^{sd}(\pi)$ is weakly real. Indeed, the right hand side of the injection is irreducible (see \cite[Proposition 3.2(i)]{Tad09a}), i.e., $\pi= \theta \rtimes X^{sd}(\pi)$. Tadi{\'c} showed that the unitarity of $\pi$ is reduced to the unitarity of $X^{sd}(\pi)$ and a choice of $\theta$, as follows.

\begin{thm}[{\cite[Theorem 1.1]{Tad09a}}]\label{thm weakly real}
In the setting above, $\pi$ is unitary if and only if there is a choice of $\pi=\theta \rtimes X^{sd}(\pi)$ such that both $\theta$ and $X^{sd}(\pi)$ are unitary.    
\end{thm}

\begin{remark}
    Suppose that $\pi$ is unitary and $\pi=\theta \rtimes X^{sd}(\pi)$. We remark that not any choice of $\theta$ is unitary. Here is an example, take any $ \rho \in \mathcal{C}^{sd}$ and $\sigma \in \mathcal{C}_{cl}$, then 
    \[ \pi:= \rho\vert\cdot\rvert^{1/3} \times \rho\vert\cdot\rvert^{-1/3} \rtimes \sigma= \rho\vert\cdot\rvert^{1/3} \times \rho \vert\cdot\rvert^{1/3} \rtimes \sigma \]
    is irreducible and unitary with $X^{sd}(\pi)=\sigma$. Set 
    \[\theta_1= \rho\vert\cdot\rvert^{1/3} \times \rho\vert\cdot\rvert^{-1/3},\  \theta_2=\rho\vert\cdot\rvert^{1/3} \times \rho \vert\cdot\rvert^{1/3}.\] 
    Then $\theta_1$ is unitary, while $\theta_2$ is not unitary since it is not Hermitian. Indeed, applying the unitary parabolic reduction (see \cite[\S 3]{Tad93}), one can see that if $\pi=\theta \rtimes X^{sd}(\pi)$ with $\theta$ being Hermitian, then $\pi$ is unitary if and only if both $\theta$ and $X^{sd}(\pi)$ are unitary.
\end{remark}

Now suppose that $\pi$ is weakly real and $\pi \in \Irr( X_{\rho_1}\sqcup\cdots \sqcup X_{\rho_r}; \sigma )$. The next natural question is whether the unitarity of $\pi$ can also be reduced to $X_{\rho_i}(\pi)$. This is the preservation of unitarizability problem proposed by Tadi{\'c}, which can be easily generalized to arbitrary $\pi \in \Pi_u(G_n)$ by Theorem~\ref{thm weakly real}.

\begin{conj}[{\cite[\S1]{Tad18}}]\label{conj Tadic preservation}
 Assume that $\pi\in \Pi(G_n)$ is weakly real. Then $\pi$ is unitary if and only if $X_{\rho_i}(\pi)$ are unitary for all $i=1,\dots ,r$.
\end{conj}

We remark that the Jantzen decomposition for an arbitrary regular partition does not necessarily preserve the unitarity, as shown in the following example.

\begin{exmp}\label{ex Jantzen decomp not preserving unitarity}
Take $\rho \in \mathcal{C}^{sd}$ and $\sigma \in \mathcal{C}_{cl}$ such that $\alpha_{\rho,\sigma}=\half{1}$. According to Tadi{\'c}'s list for corank 3 unitary dual (see \cite[Theorem 8.13]{Tad23}), the following parabolic induction is irreducible and unitary for $0 \leq x <1$.
\[\Pi_x:=L(\Delta_{\rho}[x-1/2,x-1/2], \Delta_{\rho}[x+1/2,x+1/2]) \rtimes L(\Delta_{\rho}[-1/2,-1/2];\sigma).\]
On the other hand, $\Pi_x \in \Irr(X;\sigma)$ where $X=\{\pm (x+\half{1}),\pm (x-\half{1}),\pm \half{1}\}$ $($here we simply write $z$ for $\rho\lvert\cdot\rvert^{z}$ $)$. For $0<x <\half{1}$, we have a regular partition $X=X_1 \sqcup X_2$, where
$X_1=\{\pm 1/2\}$ and $ X_2= \{\pm (x+1/2),\pm (x-1/2)\}.$
In this case, it follows from the definition that
\begin{align*}
X_1(\Pi_x)&= L(\Delta_{\rho}[-1/2,-1/2];\sigma),\\
    X_{2}(\Pi_x)&= L(\Delta_{\rho}[x-1/2,x-1/2], \Delta_{\rho}[x+1/2,x+1/2]) \rtimes \sigma.     
\end{align*}
    Then $X_1(\Pi_x)$ is unitary but $X_2(\Pi_x)$ is not according to Tadi{\'c}'s list for corank 2 unitary dual (see \cite[Proposition 7.2]{Tad23}).
\end{exmp}

Next, suppose that $\pi_1 \in \Irr(X_{\rho_1},\sigma_1)$. Take any $\rho_2 \in \mathcal{C}^{sd}$ and $\sigma_2 \in \mathcal{C}_{cl}$ such that $\alpha_{\rho_1,\sigma_1}=\alpha_{\rho_2,\sigma_2}$. In \cite[\S 12]{Tad18}, Tadi{\'c} constructed a corresponding representation $E(\pi_1) \in \Irr(X_{\rho_2},\sigma_2)$, which we briefly recall below. Note that if $\alpha_{\rho_1,\sigma_1}=0$, we fix one of the choices that Tadi{\'c} constructed by Arthur's parametrization of tempered spectrum with respect to a fixed Whittaker datum. 

For $k\in\{1,2\}$, write $\sigma_k=\pi(\phi_{k,sc},\varepsilon_{k,sc})$. Let $\epsilon \in \{0,\half{1}\}$ such that $\epsilon+\alpha_{\rho_{k},\sigma_k} \in \Z$. By the classification of supercuspidal representations, we have $\phi_{k,sc}=(\phi_{k,sc})_{\rho_k} + (\phi_{k,sc})^{\rho_k}$ where
\[ (\phi_{k,sc})_{\rho_k}= \bigoplus_{i=0}^{r_k} \rho_k \otimes S_{2 (\epsilon+i)+1 },\  (\phi_{k,sc})^{\rho_k}= \bigoplus_{\rho' \not \cong \rho_k} \bigoplus_{i \in I_{\rho',k}} \rho' \otimes S_{a_i},  \]
with $r_k \in \Z_{\geq -1}$. Note that $\alpha_{\rho_k,\sigma_k}= \epsilon+r_k+1$; hence, $r_1=r_2$.

Suppose that $\pi_1$ is tempered and write $\pi_1=\pi(\phi_1,\varepsilon_1)$. The assumption that $\pi_1 \in \Irr(X_{\rho_1};\sigma)$ implies that we have
\[ \phi_1= (\phi_{1,sc})^{\rho_1}+ \bigoplus_{i \in I_{\rho_1}} \rho_1\otimes S_{a_i}. \]
Then define $E(\pi_1)= \pi(\phi_2,\varepsilon_2)$, where 
$\phi_2=(\phi_{2,sc})^{\rho_2}+ \bigoplus_{i \in I_{\rho_1}} \rho_2\otimes S_{a_i}$ 
and 
\begin{align*}
   \varepsilon_2( \rho\otimes S_{a_i} ):= \begin{cases}
       \varepsilon_{2,sc}( \rho\otimes S_{a_i} ) & \text{ if }\rho\not\cong \rho_2,\\
       \varepsilon_{1}(\rho_1\otimes S_{a_i}) \varepsilon_{1,sc}(\rho_1\otimes S_{2\alpha_{\rho_1,\sigma_1}-1})\varepsilon_{2,sc}(\rho_2\otimes S_{2\alpha_{\rho_1,\sigma_1}-1}) &\text{ if }\rho \cong \rho_2.
   \end{cases} 
\end{align*}
Here if $\alpha_{\rho_k,\sigma_k}=0$, then set $\varepsilon_{1,sc}(\rho_1\otimes S_{-1})\varepsilon_{2,sc}(\rho_2\otimes S_{-1})=1$. For non-tempered $\pi_1$, write
\[ \pi_1= L( \Delta_{\rho_1}[x_1,y_1],\ldots, \Delta_{\rho_1}[x_f,y_f]; \pi_{1,temp} ).\]
Then define
\[ E(\pi_1):=L( \Delta_{\rho_2}[x_1,y_1],\ldots, \Delta_{\rho_2}[x_f,y_f]; E(\pi_{1,temp}) ). \]

Now we state the independence of unitarizability question proposed by Tadi{\'c}.

\begin{conj}[{\cite[\S1]{Tad18}}]\label{conj independence}
In the setting above, $\pi_1$ is unitary if and only if $E(\pi_1)$ is unitary.
\end{conj}

A natural question is whether Conjectures \ref{conj Tadic preservation} and \ref{conj independence} hold if we replace the unitarity condition by Arthur type, which may serve as evidence or an initial step towards proving these conjectures. We confirm this question positively in \S \ref{sec proof of preservation} and refer to Theorem \ref{thm preservation and independence} for precise statements.

\section{Local Arthur Packets and Extended Multi-Segments}\label{sec A-packet}

To parameterize local components of discrete automorphic representations, Arthur (\cite{Art13}) established the theory of local Arthur packets using character identities. 
In a series of papers (\cite{Moe06a, Moe06b, Moe09a, Moe10, Moe11a}), M{\oe}glin explicitly constructed each local Arthur packet $\Pi_{\psi}$ and showed that it is multiplicity free. In recent work (\cite{Ato20b}), 
Atobe gave a reformulation of M{\oe}glin's construction. 
In this section, we recall the parameterization of local Arthur packets using extended multi-segments (following \cite{Ato20b}) and the theory of their intersections (following \cite{HLL22}).

\subsection{Arthur representations and reduction to good parity}\label{local A packets}

Recall that $F$ is a non-Archimedean local field of characteristic zero, $\RG_n=\Sp_{2n}$ or split $\SO_{2n+1}$, and $G_n = \RG_n(F)$. Let $W_F$ be the Weil group and ${\RG}^\vee_n(\mathbb{C})$ be the Langlands dual group.

\begin{defn}\label{A parameters}

A \emph{complementary local  Arthur parameter} is a homomorphism
$$\psi: W_F \times \SL_2(\mathbb{C}) \times \SL_2(\mathbb{C}) \longrightarrow {{\RG}^\vee_n(\mathbb{C})},$$
\begin{equation}\label{lap}
  \psi = \bigoplus_{i=1}^r \phi_i\vert\cdot\rvert^{x_i} \otimes S_{a_i} \otimes S_{b_i},  
\end{equation}
satisfying the following conditions: 
\begin{enumerate}
    \item [(1)]$\phi_i(W_F)$ is bounded and consists of semi-simple elements, and $\dim(\phi_i)=d_i$;
    \item [(2)] $x_i \in \R$ and $|x_i|<\half{1}$;
    \item [(3)]the restrictions of $\psi$ to the two copies of $\SL_2(\mathbb{C})$ are algebraic, $S_k$ is the $k$-dimensional irreducible representation of $\SL_2(\mathbb{C})$, and 
    $$\sum_{i=1}^r d_ia_ib_i = N:= 
\begin{cases}
2n+1 & \text{ when } G_n=Sp_{2n},\\
2n & \text{ when } G_n=SO_{2n+1}.
\end{cases}
$$ 
\end{enumerate}
The first copy of $\SL_2(\mathbb{C})$ is called the Deligne-$\SL_2(\mathbb{C})$ and is denoted by $\SL_2^D(\mathbb{C})$. The second copy of $\SL_2(\mathbb{C})$ is called the Arthur-$\SL_2(\mathbb{C})$ and is denoted by $\SL_2^A(\mathbb{C})$. 
A complementary local  Arthur parameter $\psi$ given in \eqref{lap} is called a \emph{\textbf{local Arthur parameter}} if $x_i=0$ for each $i$, namely, $\psi|_{W_F}$ is bounded; it is called {\it generic} if $b_i=1$ for $i=1, \ldots, r$; and is called {\it tempered} if it is a local Arthur parameter and generic. 

Given a complementary local  Arthur parameter as in \eqref{lap}, Arthur defined a packet $\Pi_{\psi}$ in \cite[Theorem 2.2.1, \S 1.5 and \S 8.3]{Art13}, called a \emph{complementary local  Arthur packet}, which is a finite subset of $\Pi(G_n)$, satisfying certain twisted endoscopic character identities. 
We let $\Psi_{1/2}^+(G_n)$ be the set of all complementary local Arthur parameters. Let 
$$\Pi_{A+}(G_n)=\{\pi \in \Pi_{\psi} \ | \ \psi \in \Psi^+_{1/2}(G_n)\}.$$
Representations in $\Pi_{A+}(G_n)$ are called \emph{{complementary Arthur representations}}. For $\pi\in\Pi_{A+}(G_n),$ we let
\[ \Psi(\pi):= \{ \psi \in \Psi^+_{1/2}(G_n) \ | \ \pi \in \Pi_{\psi}\}.\]

A complementary local  Arthur packet $\Pi_{\psi}$ is called a \emph{\textbf{local Arthur packet}} (resp. generic packet, tempered packet) if the corresponding complementary local  Arthur parameter $\psi$ is so. Let $\Psi(G_n)$ be the subset of local Arthur parameters.
We say that a representation $\pi$ is of \emph{\textbf{Arthur type}} if $\pi\in\Pi_\psi$ for some local Arthur parameter $\psi \in \Psi(G_n)$.
Let 
$$\Pi_{A}(G_n)=\{\pi \in \Pi_{\psi} \ | \ \psi \in \Psi(G_n)\}.$$
Representations in $\Pi_{A}(G_n)$ are called \emph{\textbf{Arthur representations}}.
For $\pi\in\Pi_{A}(G_n),$ we let
\[ \Psi(\pi):= \{ \psi \in \Psi(G_n) \ | \ \pi \in \Pi_{\psi}\}.\]
\end{defn}
        
For each local parameter $\psi \in \Psi_{1/2}^+(G_n)$, Arthur associated a local $L$-parameter $\phi_{\psi}$ as follows
\begin{equation*}
\phi_{\psi}(w, x) = \psi\left(w, x, \begin{pmatrix}
        |w|^{\frac{1}{2}} & 0 \\
        0 & |w|^{-\frac{1}{2}}\\
\end{pmatrix}\right).
\end{equation*}
Arthur also showed that the map $\psi \mapsto \phi_{\psi}$ is injective and the local $L$-packet $\Pi_{\phi_{\psi}}$ is a subset of the local packet $\Pi_{\psi}$ (\cite[Proposition 7.4.1]{Art13}). The corresponding local $L$-packet $\Pi_{\phi_{\psi}}$ is called a local $L$-packet of Arthur type. A local $L$-parameter $\phi$ is said to be of Arthur type if $\phi=\phi_\psi$ for some local parameter $\psi \in \Psi_{1/2}^+(G_n).$ The {\it diagonal restriction} of $\psi$ is another local $L$-parameter associated to $\psi$ which is defined as 
\begin{align}\label{def diag rest}
\begin{split}
    \psi^{\Delta}:  W_F \times \SL_2(\BC)  \longrightarrow {\RG}^\vee_n(\mathbb{C}),\quad {\rm with}\ 
    (w,x) \mapsto \psi(w,x,x).
\end{split}
\end{align}
For each supercuspidal representation $\pi_{sc}$, the set $\Psi(\pi_{sc})$ is completely determined by the local $L$-parameter of $\pi$ via the diagonal restriction, as follows.

\begin{thm}[{\cite[Theorem 7.5]{Moe09a}}]\label{thm GGP20}
Let $\pi_{sc}=\pi(\phi_{sc},\varepsilon_{sc})$ be a supercuspidal representation of $G_n$. Then $\psi \in \Psi(\pi_{sc})$ if and only if 
$\psi^{\Delta}= \phi_{sc} \otimes S_{1}.$
\end{thm}

Now we recall the decomposition of local parameters and the reduction of the construction of local Arthur packets to the good parity case. By the Local Langlands Correspondence for $\GL_{d_i}(F)$, an irreducible   bounded representation $\phi$ of $W_F$ can be identified with an irreducible unitary supercuspidal representation $\rho$ of $\GL_{d_i}(F)$ (\cite{Hen00, HT01, Sch13}). Consequently, we may write \eqref{lap} as
\begin{equation}\label{A-param decomp}
  \psi = \bigoplus_{i \in I } \rho_i\vert\cdot\rvert^{x_i} \otimes S_{a_i} \otimes S_{b_i},  
\end{equation}
where $\rho_i$'s are irreducible unitary supercuspidal representations of $\GL_{d_i}(F)$. With this decomposition, we say that $\psi$ is of \emph{good parity} if the following holds: $\rho_i$ is selfdual, $x_i=0$, and $\half{a_i+b_i}\in \alpha_{\rho_i,\sigma} +\Z$ for any $i \in I$ and any irreducible supercuspidal representation $\sigma$ of $G_n$.
Equivalently, $\psi$ is of good parity if and only if any representation in the local $L$-packet $\Pi_{\phi_{\psi}}$ is of good parity (Definition \ref{def critical type}(2)). We remark that by Theorem \ref{thm red from nu to gp} below and the construction of good parity local Arthur packets (see the next subsection), $\psi$ is of good parity if and only if any representation in the local packet $\Pi_{\psi}$ is of good parity.

Consider a complementary local Arthur parameter $\psi \in \Psi_{1/2}^+(G_n).$ Since $\psi$ factors through $\RG_n^\vee(\BC)$, we may rewrite the decomposition \eqref{A-param decomp} as 
\begin{align*}
    \psi= & \bigoplus_{i \in I_{>0}} (\rho_i \vert\cdot\rvert^{x_i} \otimes S_{a_i}\otimes S_{b_i} + \rho_i^{\vee} \vert\cdot\rvert^{-x_i} \otimes S_{a_i}\otimes S_{b_i})\\
    &\qquad\qquad\qquad\oplus\bigoplus_{i \in I_{ngp}} (\rho_i \otimes S_{a_i} \otimes S_{b_i} +\rho_i^{\vee} \otimes S_{a_i} \otimes S_{b_i}) \oplus\bigoplus_{i \in I_{gp}} \rho_i \otimes S_{a_i} \otimes S_{b_i},
\end{align*}
where 
\begin{enumerate}
    \item [$\oldbullet$] for any $i \in I_{>0}$, $0<x_i<\half{1}$; 
    \item [$\oldbullet$] for any $i \in I_{ngp}$, either $\rho_i\not\cong \rho_i^{\vee}$, or, $\rho_i\cong \rho_i^{\vee}$ and $\half{a_i+b_i}\not\in \alpha_{\rho_i,\sigma} +\Z$;
    \item [$\oldbullet$] for any $i \in I_{gp}$, $\rho_i \cong \rho_i^\vee$ and $\half{a_i+b_i}\in \alpha_{\rho_i,\sigma} +\Z$.
\end{enumerate}
Set $x_i:=0$ for $i \in I_{ngp}\sqcup I_{gp}$. For $\ast \in  \{ >0, \ ngp,\ gp  \}$, we define subrepresentations $\psi_{\ast}$ of $\psi$ by 
\[ \psi_{\ast}:= \bigoplus_{i \in I_{\ast}} \rho_i\vert\cdot\rvert^{x_i} \otimes S_{a_i} \otimes S_{b_i}.\]
Thus $\psi_{gp}$ is of good parity and 
\begin{align}\label{eq decomp red to gp}
    \psi=( \psi_{>0}+ \psi_{>0}^{\vee}) +(\psi_{ngp} + \psi_{ngp}^{\vee})+ \psi_{gp}.
\end{align}
Also set $\psi_{u}:= (\psi_{ngp}+\psi_{ngp}^{\vee}) + \psi_{gp}$. Then $\psi_{u}\in \Psi(G_m)$ for some $m \leq n$.

For a unitary supercuspidal representation $\rho$ of $\GL_d(F)$ and $a, b \in \Z_{>0}$, we let 
\begin{equation}\label{u rho a b}
u_{\rho}(a,b):= \begin{pmatrix}
\frac{a-b}{2} & \cdots & \frac{a+b}{2}-1 \\
\vdots & \ddots & \vdots \\
\frac{-a-b}{2}+1 & \cdots & \frac{b-a}{2}
\end{pmatrix}_{\rho}, 
\end{equation}
be the corresponding unitary generalized Speh representation (see \eqref{generalized Speh representation}). This is the unique member of the local Arthur packet $\Pi_{\rho \otimes S_{a} \otimes S_b} (\GL_{abd}(F))$.

For each $i \in I_{>0} \sqcup I_{ngp}$, define $\tau_i$ to be the generalized Speh representation
$u_{\rho_i}(a_i,b_i)\lvert \cdot \rvert^{x_i}$. 
Then we set 
\[ \tau_{\psi_{>0}}:=\bigtimes_{i \in I_{>0}} \tau_i\quad {\rm and}\quad \tau_{\psi_{ngp}}:= \bigtimes_{i \in I_{ngp}} \tau_i,\]
which are irreducible. We take this chance to recall a smaller subset $ \Psi_{\text{unit}}^{+}(G_n) \subset \Psi^+_{1/2}(G_n)$ defined by
\begin{align}\label{eq def Psi^+_unit}
    \Psi_{\text{unit}}^{+}(G_n):= \{ \psi\in \Psi^+_{1/2}(G_n)  \ | \ \bigtimes_{i \in I_{>0}}  \tau_i \times \tau_i^{\vee} \text{ is unitary} \}.
\end{align}
Based on the classification of unitary dual of general linear groups in \cite{Tad86}, the set $\Psi_{\text{unit}}^{+}(G_n)$ is equal to the collection of $\psi \in \Psi^+_{1/2}(G_n)$ such that the index set $I_{>0}$ can be equipped with a duality map
\begin{align*}
    (\cdot)^{\vee}: I_{>0} \longrightarrow I_{>0},\quad \   i  \mapsto i^{\vee},
\end{align*}
such that $\rho_{i^{\vee}} \cong (\rho_i)^{\vee}$.
The disjoint union $\Psi(G_n) \cup \Psi_{\text{unit}}^{+}(G_n)$ is equal to the set $\widetilde{\Psi}^+_{\text{unit}}(G_n)$ defined in \cite[\S 8.3]{Art13},

For $\psi \in \Psi^{+}_{1/2}(G_n)$ (i.e., $I_{>0}$ is non-empty), 
the local Arthur packet $\Pi_{\psi}$ is defined to be the set of representations (\cite[(1.5.1)]{Art13})
\begin{align}\label{eq def packet A+}
    \Pi_{\psi}:=\{ \tau_{\psi_{>0}} \rtimes \pi_{u}\ | \ \pi_u \in \Pi_{\psi_u}   \}.
\end{align}
 Indeed, Arthur only defined $\Pi_{\psi}$ for $\psi \in \Psi^{+}_{\text{unit}}(G_n)$, but the definition can be extended naturally for $\psi \in \Psi^{+}_{1/2}(G_n)$. In \cite{Moe11b}, M{\oe}glin showed that for $\psi \in \Psi^+_{1/2}(G_n)$, the parabolic inductions on the right hand side of \eqref{eq def packet A+} are all irreducible. Moreover, $\pi_u$ can be further written as a parabolic induction according to the decomposition of $\psi_{u}$ as follows.

\begin{thm}[{\cite[Proposition 5.1]{Moe11b}}]\label{thm red from nu to gp}
Let $\psi\in\Psi_{1/2}^+(G_n)$ with decomposition \eqref{eq decomp red to gp}. Then, for any $\pi_{gp}\in\Pi_{\psi_{gp}},$ the parabolic induction $\tau_{\psi_{>0}}\times\tau_{\psi_{ngp}}\rtimes\pi_{gp}$ is irreducible. As a consequence, we obtain
\begin{equation}\label{non-unitary A-packet}
    \Pi_\psi=\{\tau_{\psi_{>0}}\times\tau_{\psi_{ngp}}\rtimes\pi_{gp}  | ~ \pi_{gp}\in\Pi_{\psi_{gp}}\}.
\end{equation}
\end{thm}

\begin{remark}\label{rmk projection}
We may rephrase the above theorem as a \emph{reduction to good parity}.
First, there is a projection
     \begin{align*}
        p_{+}: \Pi_{A+}(G_n) \longrightarrow\bigsqcup_{0 \leq d \leq n} \Pi_{A}(G_{n-d}).
    \end{align*}
    For any $\sigma \in \Pi_{A}(G_{n-d})$, the fiber $p_{+}^{-1}(\sigma)$ consists of $\tau \rtimes \sigma$ such that {$\tau \times {\tau}^{\vee}$} is a product of unitary complementary series. 
Secondly, there is a projection
    \begin{align*}
        p_{gp}: \Pi_{A}(G_n) \longrightarrow \bigsqcup_{0 \leq d \leq n} \Pi_{A, \, gp}(G_{n-d}),
    \end{align*}
    such that for any $\sigma \in \Pi_{A, \, gp}(G_{n-d})$, the fiber $p_{gp}^{-1}(\sigma)$ can be described explicitly as
    \[ p_{gp}^{-1}(\sigma)= \{ \tau \rtimes \sigma \ | \ \tau \in \Pi_{A}(\GL_d(F)) \text{ is of opposite parity from }G_n \}.\]    
\end{remark}

\subsection{Arthur representations and critical type}\label{sec critical type and Arthur type}
In this subsection, we propose several questions on the relation between $\Pi_{A}(G_n)$ and $\Pi_{A, \, crit}(G_n)$.

Let $\pi \in \Pi(G_n)$ and realize it as an irreducible subquotient of
$\rho_1\vert\cdot\rvert^{x_1} \times \cdots \times \rho_f\vert\cdot\rvert^{x_f} \rtimes \pi_{sc},$
where $\rho_i \in \mathcal{C}^u $, $x_i \geq 0$ and $\pi_{sc} \in \mathcal{C}_{cl}$. For simplicity, we assume that $\rho_1=\cdots=\rho_f=\rho \in \mathcal{C}^{sd}$ and let $\alpha=\alpha_{\rho,\pi_{sc}}$ be the reducibility point. Let $X=\{\rho\vert\cdot\rvert^{x_i}, \rho\vert\cdot\rvert^{-x_i}\}_{i=1,\ldots,f}$, then $\pi \in \Irr(X; \pi_{sc})$. There is a unique regular decomposition $X=X_{crit}\sqcup X_{non-crit}$ such that $X_{crit}$ is a union of a segment $\Delta_{crit}$ with its contragredient such that $\alpha \in \Delta_{crit}$. Then according to Jantzen decomposition (Theorem \ref{thm Jantzen}(ii)), there is an injection
$\pi \hookrightarrow \beta \rtimes X_{crit}(\pi),$
where $X_{crit}(\pi)$ is of critical type, supported on $X_{crit} \cup \{\pi_{sc}\}$, and is uniquely determined by $\pi$, while there may be multiple choices for $\beta$. This defines a projection
\begin{align*}
     p_{crit}: \Pi(G_n) \longrightarrow \bigsqcup_{0 \leq d \leq n} \Pi_{crit}(G_{n-d}), \quad     \pi\mapsto X_{crit}(\pi).
\end{align*}
Here are some natural questions about whether the behavior of $p_{crit}$ is similar to those of $p_{+}$ and $p_{gp}$ defined in Remark \ref{rmk projection}.

\begin{ques}\label{ques crit} \
    \begin{enumerate}
        \item Does the projection $p_{crit}$ preserve Arthur type? Namely, is the following map well-defined?
        \begin{align*}
            p_{crit}|_{\Pi_A(G_n)}\colon \Pi_A(G_n) \longrightarrow\sqcup_{0\leq d \leq n} \Pi_{A}(G_{n-d}).
        \end{align*}
        \item Assume that the previous question is affirmative. For any $\pi \in \Pi_{A}(G_n)$, how to characterize the set
        $\{ \beta \ | \ \pi= \beta \rtimes p_{crit}(\pi)\}$? 
        In particular, is this set always non-empty? In other words, can we always realize a non-critical Arthur representation as an irreducible parabolic induction?
        \end{enumerate}
\end{ques}

Note that if Part (2) is true, then for any $\pi \in \Pi_{A}(G_n)$, which is not of critical type,
we may form a complementary series 
$\Pi_{s}:= \beta\vert\cdot\rvert^{s} \rtimes p_{crit}(\pi),$ for $s \in \mathbb{R}$. Note that $\Pi_{s}$ is irreducible when 
 $|s|$ small, Hermitian when it is irreducible since it is weakly real, and unitary when $s=0$ since $\Pi_0=\pi$. It follows that any such $\pi$ is not isolated in the unitary dual. In other words, Part (2) above implies that
\[ \Pi_{A}(G_n)\cap \Pi_{iso}(G_n) \subseteq \Pi_{A, \, crit}(G_n), \]
which may be a crucial step towards Conjecture \cite[Theorem 1.1]{Tad22}. 
Here $\Pi_{iso}(G_n)$ is the set of isolated representations in the unitary dual $\Pi_u(G_n)$.

On the other hand, even if the above questions have affirmative answers, it is still hard to construct $\Pi_{A}(G_n)$ from $\bigsqcup_{0 \leq d \leq n}\Pi_{A, \, crit}(G_{n-d})$ and the projection $p_{crit}$, compared to the cases of $p_{+}$ and $p_{gp}$. The main reasons are that $\pi= \beta \rtimes p_{crit}(\pi) $ is usually not a unitary parabolic induction and that it seems hard to list 
$\Psi(\pi)= \{ \psi \ | \ \pi \in \Pi_{\psi}\}$ 
from $\Psi(p_{crit}(\pi))$ and $\beta.$ The following example provides an illustration.

\begin{exmp}\label{ex non-critical type}
Let $\rho$ be the trivial representation of $\GL_1(F)$. Consider the following representation of $Sp_{16}$:
\[ \pi= L (  \Delta_{\rho}[-1,-1], \Delta_{\rho}[0,-2]; \pi(0^-,1^+,2^-)). \]
Applying \cite[Algorithm 7.9]{HLL22}, one can check that it is of Arthur type and lies in three local Arthur packets $\{\Pi_{\psi_i}\}_{i=1,2,3}$, where
\begin{align*}
    \psi_1&= \rho\otimes S_1\otimes S_3+\rho\otimes S_3\otimes S_3+\rho\otimes S_5\otimes S_1, \\
    \psi_2&= \rho\otimes S_2\otimes S_4+\rho\otimes S_2\otimes S_2+\rho\otimes S_5\otimes S_1, \\
    \psi_3&= \rho\otimes S_1\otimes S_5+\rho\otimes S_1\otimes S_3+\rho\otimes S_2\otimes S_2+\rho\otimes S_5\otimes S_1.
\end{align*}
The representation $\pi$ is an irreducible subquotient of
$\rho\vert\cdot\rvert^{2} \times \rho\vert\cdot\rvert^{1}\times \rho\vert\cdot\rvert^{1} \times \rho\vert\cdot\rvert^{0} \rtimes \pi_{sc},$ 
where $\pi_{sc}=\pi(0^-,1^+,2^-) $ and $\alpha=\alpha_{\rho,\pi_{sc}}=3$. Thus, we have $X_{crit}(\pi)= \pi_{sc}$. We list below two choices of $\beta$ such that
$\pi \hookrightarrow \beta \rtimes X_{crit}(\pi).$

A naive choice is $\beta_1= \Delta_{\rho}[0,-2] \times \Delta_{\rho}[-1,-1] $, which realizes $\pi$ as the unique irreducible subrepresentation of its standard module. However, we have the following claim on $\beta_1 \rtimes \pi_{sc}$:

\textbf{Claim}: $\beta_1 \rtimes \pi_{sc}$ is not irreducible. More precisely, $\beta_1 \rtimes \pi_{sc}= \pi+ \pi'$ in the Grothendieck group, where 
$\pi'=L(\Delta_{\rho}[1,-2];\pi_{sc}).$ 

First, we use the following fact to list possible candidates of the irreducible subquotients: If $\pi''\not \cong \pi$ is a subquotient of $\beta_1 \rtimes \pi_{sc}$, which is the standard module of $\pi$, then the local $L$-parameter of $\pi''$ must be greater than the local $L$-parameter of $\pi$ in terms of closure ordering (for the definition of closure ordering, see \cite[\S4]{HLLZ25}). Thus, there are only three choices for the local $L$-parameter of $\pi''$:
    \begin{align*}
        \phi_1&= \rho\otimes S_{3} + \rho \otimes S_5 +\phi_{sc},\\
        \phi_2&= (\rho\vert\cdot\rvert^{-1} \otimes S_1+\rho\vert\cdot\rvert^{1} \otimes S_1)+ \rho\otimes S_{1} + \rho \otimes S_5 + \phi_{sc},\\
        \phi_3&=(\rho\vert\cdot\rvert^{-1/2}\otimes S_4+\rho\vert\cdot\rvert^{1/2}\otimes S_4)+ \phi_{sc}.
    \end{align*}
    It can be checked that $\pi'$ is the only representation in $\Pi_{\phi_1} \sqcup \Pi_{\phi_2} \sqcup \Pi_{\phi_3}$ that shares the same supercuspidal support as $\pi$. Therefore, we have 
    $\beta_{1} \rtimes \pi_{sc}= \pi+ m \pi'$ 
    for some $m\geq 0$.
    
Next, it is not hard to see that $\pi'$ appears in the parabolic induction, i.e., $m \geq 1$. Indeed, in the Grothendieck group, we have
\begin{align*}
    \beta_1 \rtimes \pi_{sc}&= \Delta_{\rho}[0,-2] \times \Delta_{\rho}[-1,-1] \rtimes \pi_{sc}\\
    &= \Delta_{\rho}[0,-2] \times \Delta_{\rho}[1,1] \rtimes \pi_{sc}\\
    &= \Delta_{\rho}[1,-2] \rtimes \pi_{sc}+ L( \Delta_{\rho}[0,-2] \times \Delta_{\rho}[1,1]) \rtimes \pi_{sc}.
\end{align*}
    Here $\beta_2:= L( \Delta_{\rho}[0,-2] \times \Delta_{\rho}[1,1])$ is the unique irreducible subrepresentation of 
    $\Delta_{\rho}[0,-2] \times \Delta_{\rho}[1,1].$ 
    Thus, $\pi'$, the unique irreducible subrepresentation of $\Delta_{\rho}[1,-2] \rtimes \pi_{sc}$, appears in $\beta_1\rtimes \pi_{sc}.$ Moreover, one can apply the results in \cite[\S 3]{Ato22d} to check that $ \Delta_{\rho}[1,-2] \rtimes \pi_{sc}$ is irreducible; hence, it is equal to $\pi'$.
    
    Finally, it remains to show that $m=1$. Indeed, this follows from the results in \cite[\S 3]{Ato22d}. More precisely, we have 
    \begin{enumerate}
        \item [$\oldbullet$] The parabolic induction $\sigma= \Delta_{\rho}[0,-2] \rtimes \pi_{sc}$ is irreducible.
        \item [$\oldbullet$] The representation $\pi$ (resp. $\pi'$) is the unique irreducible subrepresentation of the parabolic induction
        $\Delta_{\rho}[-1,-1] \rtimes \sigma$ (resp $\Delta_{\rho}[1,1] \rtimes \sigma$). Hence, they appear in the semisimplification of $\beta_1 \rtimes \pi_{sc}$ of multiplicity one (\cite[Corollary 5.1]{Ato22d}).
    \end{enumerate}
    This completes the verification of the claim. Also, we realize $\pi$ as an irreducible parabolic induction
    \[ \pi= \beta_2 \rtimes \pi_{sc}=L( \Delta_{\rho}[0,-2] \times \Delta_{\rho}[1,1]) \rtimes \pi_{sc}. \]

    Note that the supercuspidal representation $\pi_{sc}$ lies in 9 local Arthur packets, which are listed in \cite[Example 7.6]{HLL22}. It is not straightforward to relate the three local Arthur packets containing $\pi$ to these 9 local Arthur packets and $\beta_2$.
\end{exmp}

\subsection{Extended multi-segments}\label{ssec-EMSKO}

In this subsection, we revisit the concept of extended multi-segments, which plays an important role in the construction of local Arthur packets, along with the associated theorems and operators. 

\begin{defn} [{\cite[Definition 3.1]{Ato20b}}]
(Extended multi-segments)\label{def multi-segment}

\begin{enumerate}
\item
An \emph{extended segment} is a triple $([A,B]_\rho, l, \eta)$,
where
\begin{itemize}
\item
$[A,B]_\rho = \{\rho\vert\cdot\rvert^A, \rho\vert\cdot\rvert^{A-1}, \dots, \rho\vert\cdot\rvert^B \}$ is a segment 
for an irreducible unitary supercuspidal representation $\rho$ of some $\GL_d(F)$; 
\item
$l \in \Z$ with $0 \leq l \leq \frac{b}{2}$, where $b = \#[A,B]_\rho = A-B+1$; 
\item
$\eta \in \{\pm1\}/E$, where $E=\{\pm 1\}$ if $b=2l$ and $E=\{+1\}$ if $b>2l$. 
\end{itemize}
In the statements and formulas in this section, we regard $\eta \in \{\pm 1\}$ by fixing any of its preimages except in Definition \ref{dual segment}, where the choice of the preimage is specified.
\item Consider a multi-set of extended segments of the form $\{([A_i,B_i]_{\rho},l_i,\eta_i)\}_{i \in I_{\rho}}.$
We say that a total order $>$ on $I_{\rho}$ is \emph{admissible} (or \emph{satisfies $(P)$}) if
\[ A_i< A_j, B_i< B_j\Longrightarrow i<j. \]
We say that an admissible order $>$ \emph{satisfies ($P'$)} if
$B_i< B_j\Longrightarrow i<j. $

\item
An \emph{extended multi-segment} for $G_n$ is a union of multi-sets of extended segments indexed by a collection of total ordered sets $(I_{\rho},>)$:
$\EE = \cup_{\rho}\{ ([A_i,B_i]_{\rho}, l_i, \eta_i) \}_{i \in (I_\rho,>)}$ 
such that 
\begin{enumerate}
\item
$I_\rho$ is a totally ordered finite set with a fixed total order $>$ satisfying (P);

\item
$A_i + B_i \geq 0$ for all $\rho$ and $i \in I_\rho$; 

\item
as a representation of $W_F \times \SL_2(\BC) \times \SL_2(\BC)$, 
$\psi_{\EE} = \bigoplus_\rho \bigoplus_{i \in I_\rho} \rho \otimes S_{a_i} \otimes S_{b_i},$
where $(a_i, b_i) = (A_i+B_i+1, A_i-B_i+1)$,
is a local Arthur parameter for $G_n$ of good parity. We denote by $\psi_{\EE}$ the local Arthur parameter associated with $\EE$. 
\item The sign condition
\begin{align*}
\prod_{\rho} \prod_{i \in I_\rho} (-1)^{[\frac{b_i}{2}]+l_i} \eta_i^{b_i} = 1
\end{align*}
holds.
\end{enumerate}
\item For each extended multi-segment $\EE$, we denote by $\pi(\EE)$ the representation associated with $\EE$ as in \cite[\S 3.2]{Ato20b}, which is either irreducible or zero. We denote by $\Rep$ the set of extended multi-segments that give nonzero representations, and by $\Rep^{(P')}$ the subset of $\Rep$ that consists of extended multi-segments whose total order on any $I_{\rho}$ satisfies $(P')$.
For a given 
$\EE = \cup_{\rho}\{ ([A_i,B_i]_{\rho}, l_i, \eta_i) \}_{i \in (I_\rho,>)} \in \Rep,$
define
$$\EE_{\rho}=\{([A_i,B_i]_{\rho},l_i,\eta_i)\}_{i \in (I_{\rho},>)}\quad {\rm and}\quad 
\EE^{\rho}=\cup_{\rho' \ncong \rho}
\{([A_i,B_i]_{\rho'},l_i,\eta_i)\}_{i \in (I_{\rho'},>)}.$$
Then $\EE=\EE^{\rho} \cup \EE_{\rho}$.
\item Sometimes we write $\EE_{\rho}=\{([A_1,B_1]_{\rho},l_1,\eta_1),\ldots,([A_k,B_k]_{\rho},l_k,\eta_k) \}$, implying that the elements in $\EE_{\rho}$ are listed increasingly with respect to the admissible order of $\EE_{\rho}$. Assume that 
\[ \FF=\{([A_i,B_i]_{\rho},l_i,\eta_i)\}_{i \in (I_{\rho},>)}\quad {\rm and}\quad \FF'=\{([A_i,B_i]_{\rho},l_i,\eta_i)\}_{i \in (I_{\rho}',>)}.\]
Then we let $\FF+\FF'=\{ ([A_i,B_i]_{\rho},l_i,\eta_i)\}_{i \in (I_{\rho}\sqcup I_{\rho}',\gg})$ be the extended multi-segment, 
where the admissible order $\gg$ is defined by $i\gg j$ if and only if $(i,j)\in I_{\rho}\times I_{\rho} \sqcup I_{\rho}' \times I_{\rho}' $ and $i> j$, or, $(i,j) \in I_{\rho}' \times I_{\rho}$.
\item Suppose that $\EE \in \Rep^{(P')}$ and denote $$\FF=\EE_{\rho}= \{([A_i,B_i]_{\rho},l_i,\eta_i)\}_{ i \in (I_{\rho},>)}.$$ 
Let $ x \in \R$. We define
\begin{align*}
\FF_{ >x}&:=\{ ([A_i,B_i]_{\rho},l_i,\eta_i)\}_{i \in I_{\rho}, B_i>x},\\
    \FF_{=x}&:=\{ ([A_i,B_i]_{\rho},l_i,\eta_i)\}_{i \in I_{\rho}, B_i=x},\\
    \FF_{ <x}&:=\{ ([A_i,B_i]_{\rho},l_i,\eta_i)\}_{i \in I_{\rho}, B_i<x},
\end{align*}
with the admissible order inherited from $(I_{\rho}, >)$. Note that $\FF= \FF_{<x}+\FF_{=x}+ \FF_{>x}$. We also write $\FF_{\leq x}= \FF_{<x} + \FF_{=x}$ and $\FF_{\geq x}= \FF_{=x} + \FF_{>x}$. 
\end{enumerate}
\end{defn}

Atobe attached a symbol to an extended multi-segment $\EE$ (\cite[\S3]{Ato20b}), to understand examples. When $\EE=\{([A,B]_\rho,l,\eta)\}$ is a singleton, we denote
\[
\EE= 
\left(
\begin{array}{rcl}
\underbrace{\overset{B}{\lhd} \lhd \cdots \overset{B+l-1}{\lhd}}_l 
&
\overset{B+l}{\odot} \odot \cdots \odot \overset{A-l} \odot 
&
\underbrace{\overset{A-l+1}{\rhd} \cdots \rhd \overset{A}{\rhd}}_l
\end{array}
\right)_\rho.
\] 
Here, $\odot\cdots\odot$ represents an alternating sequence of $\oplus$ and $\ominus$, 
starting with $\oplus$ if $\eta = 1$ (resp. $\ominus$ if $\eta = -1$). When $\EE$ is not a singleton, we stack each row vertically.
We give an example as follows.

\begin{exmp}
 Let $\rho$ be the trivial representation of $\GL_1(F)$. The symbol
\[\EE= \scalebox{0.8}{\bordermatrix{
  & 0 &1 & 2 &3 \cr
  &\lhd & \oplus & \ominus & \rhd  \cr
  & &  & \lhd &\rhd  \cr
  & &  &  &  \ominus \cr
}}_{\rho}\]
corresponds to the extended multi-segment
$\EE= \{ ([ 3,0]_{\rho},1,1),([3,2]_{\rho},1,1),([3,3]_{\rho},0,-1)\}$
of $Sp_{34}$. The associated local Arthur parameter is
    $\psi_{\EE}= \rho \otimes S_{4}\otimes S_{4} + \rho \otimes S_{6}\otimes S_{2} + \rho \otimes S_7 \otimes S_1.$
\end{exmp}

In this paper, we do not use Atobe's definition to compute $\pi(\EE)$. Instead, we give new algorithms (Algorithms \ref{algo rep(E)} and \ref{algo rep(E) 2}) to compute $\pi(\EE)$ based on the previous result of the first, third, fourth, and fifth-named authors in \cite{HLLZ25}. See \S \ref{sec new algorithm} for more details. The algorithms reduce the computation to the following setting.

\begin{defn}[{\cite[Definition 9.1]{HLL22}}]\label{def (L)}
An extended multi-segment 
\[\EE=\cup_{\rho} \{([A_i,B_i]_{\rho},l_i,\eta_i)\}_{i \in (I_{\rho},>)}\]
is said to \emph{satisfy the condition} $(L)$ if it satisfies the following conditions.
For all $\rho$ and $i<j \in I_{\rho}$,
\begin{enumerate}
    \item [(i)] $ A_i +B_i \leq A_{j}+B_{j}.$
    \item [(ii)] $(A_i-B_i+1)-2l_i \leq 1$.
    \item [(iii)] If $A_i+B_i=A_{j}+B_{j}$ and $A_i-B_i+1$ is odd, then $\eta_i=\eta_{j}$.
\end{enumerate}
\end{defn}

The representations attached to the extended multi-segments that satisfy the condition $(L)$ are always nonzero. Moreover, their $L$-data have simple descriptions. 

\begin{prop}[{\cite[Theorem 9.5]{HLL22}}]\label{prop $L$-data L-packet}
If $\EE=\cup_{\rho} \{([A_i,B_i]_{\rho},l_i,\eta_i)\}_{i \in (I_{\rho},>)}$ is an extended multi-segment satisfying the condition $(L)$, then $0 \neq \pi(\EE) \in \Pi_{\phi_{\psi_{\EE}}}$. Moreover, its $L$-data can be described as follows. Write
\[ \pi(\EE)= L\left( \Delta_{\rho_1}[x_1,-y_1],\dots,\Delta_{\rho_t}[x_t,-y_t]; \pi\left(\sum_{i=t+1}^{m} \rho_i \otimes  S_{2z_i+1},\varepsilon \right) \right). \]
Then the following hold.
\begin{enumerate}
    \item [(a)] We have an equality of multi-sets
    \begin{align*}
        \{ [x_i,-y_i]_{\rho_i} \}_{i=1}^t&= \bigcup_{\rho}\sum_{i\in I_{\rho}} \{ [B_i,-A_i]_{\rho},\dots,[B_i+l_i-1,-A_i+l_i-1]_{\rho} \}.
    \end{align*}
    \item [(b)] Let $I_{\rho}^{odd}:= \{j \in I_{\rho}\ | \ A_j-B_j+1 \text{ is odd}\}$. Then we have an equality of multi-sets
    \[\\
        \{(\rho_i, z_i)\}_{i=t+1}^m =\bigcup_{\rho} \left\{ \left(\rho,\frac{A_j+B_j}{2}\right)\ \middle| \ j \in I_{\rho}^{odd} \right\}. \]
    If $z_i= \half{A_j+B_j}$, the character $\varepsilon$ takes value $\eta_j$ on the summand $\rho\otimes S_{2z_i+1}$.
\end{enumerate}
\end{prop}

Atobe showed that extended multi-segments parameterize local Arthur packets of good parity.
\begin{thm}[{\cite[Theorem 3.4]{Ato20b}}]\label{thm Atobe's reformulation}
Let $\psi= \bigoplus_{\rho} \bigoplus_{i \in I_{\rho}} \rho \otimes S_{a_i} \otimes S_{b_i}$ be a local Arthur parameter of good parity of $G_n$. Choose an admissible order $>_{\psi}$ on $I_{\rho}$ for each $\rho$ that satisfies ($P'$) if $\half{a_i-b_i}<0$ for some $i \in I_{\rho}$. Then
\[ \bigoplus_{\pi \in \Pi_{\psi}} \pi= \bigoplus_{\EE} \pi(\EE),\]
where $\EE$ runs over all extended multi-segments with $\psi_{\EE}=\psi$ with admissible order $>_{\psi}$ such that $\pi(\EE) \neq 0$. 
\end{thm}

\subsection{Operators on extended multi-segments}

Since local Arthur packets may have intersections, it is natural to ask how to use their parameterization to describe the intersections. In other words, given any $\pi \in \Pi_{gp}(G_n)$, we are interested in the set 
$\Psi(\pi):= \{ \psi \ | \ \pi \in \Pi_{\psi} \}.$ 
In terms of extended multi-segments, given any $\EE \in \Rep$, we would like to understand the set
\[ \mathscr{E}(\EE):= \{ \EE' \  | \ \pi(\EE')=\pi(\EE) \}.\]
This problem was independently answered by Atobe (\cite[Theorem 1.3]{Ato23}), and by the first, third, and fourth-named authors (\cite[Theorem 1.4]{HLL22}). The set $\mathscr{E}(\EE)$ can be constructed from $\EE$ by applying a sequence of operators, which we recall now.

The first operator is called the \emph{row exchange}. It is used to change the admissible order of an extended multi-segment, but does not affect the corresponding local Arthur parameters. We say that $k<k+1$ are \emph{adjacent} in a total order set $(I_{\rho},>)$ if there does not exist $i \in I_{\rho}$ such that $k< i <k+1$.

\begin{defn}[{\cite[Section 4.2]{Ato20b}}, Row exchange]\label{def row exchange} 
Let $\EE\in\Rep^{(P')}$ with
$$\EE_{\rho}=\{([A_i,B_i]_{\rho},l_i,\eta_i)\}_{i \in (I_{\rho},>)}.$$
Assume that $k<k+1$ are adjacent in $(I_{\rho},>)$. Let $\gg$ be the total order on $I_\rho$ defined by $k\gg k+1$ and if $(i,j)\neq (k,k+1)$, then $ i \gg j$ if and only if $
i >j .$ 

If $\gg$ is not an admissible order on $I_{\rho}$, then we define $R_k(\EE)=\EE$. Otherwise, we define 
\[R_{k}(\EE_{\rho})=\{([A_i,B_i]_{\rho},l_i',\eta_i')\}_{i \in (I_{\rho},\gg)},\]
where $( l_i',\eta_i')=(l_i,\eta_i)$ for $i \neq k,k+1$, and $(l_k',\eta_k')$ and $(l_{k+1}', \eta_{k+1}')$ are given as follows.  Denote $\epsilon=(-1)^{A_k-B_k}\eta_k\eta_{k+1}$.
\begin{enumerate}
    \item [(1)] Assume that $[A_k,B_k]_{\rho} \supset [A_{k+1},B_{k+1}]_{\rho}$.    
    We set $(l_{k+1}',\eta_{k+1}')=(l_{k+1}, (-1)^{A_k-B_k}\eta_{k+1})$ in this case. 
    \begin{enumerate}
    \item [(a)] If $\epsilon=1$ and $b_k- 2l_k < 2(b_{k+1}-2l_{k+1})$, then
    \[ (l_k', \eta_{k}')= (b_k-(l_k+ (b_{k+1}-2l_{k+1})), (-1)^{A_{k+1}-B_{k+1}} \eta_k);  \]
    \item [(b)] If $\epsilon=1$ and $b_k- 2l_k \geq  2(b_{k+1}-2l_{k+1})$, then
    \[ (l_{k}', \eta_{k}')= (l_k+ (b_{k+1}-2l_{k+1}), (-1)^{A_{k+1}-B_{k+1}+1} \eta_k);  \]
    \item [(c)] If $\epsilon=-1$, then
    $(l_{k}', \eta_{k}')= (l_k- (b_{k+1}-2l_{k+1}), (-1)^{A_{k+1}-B_{k+1}+1} \eta_k).$
\end{enumerate}
    \item [(2)] Assume that $ [A_k,B_k]_{\rho} \subseteq [A_{k+1},B_{k+1}]_{\rho}$.     
    We set $(l_{k}',\eta_{k}')=(l_{k}, (-1)^{A_{k+1}-B_{k+1}}\eta_{k})$ in this case. 
    \begin{enumerate}
   \item [(a)] If $\epsilon=1$ and $b_{k+1}- 2l_{k+1} < 2(b_{k}-2l_{k})$, then
    \[ (l_{k+1}', \eta_{k+1}')= (b_{k+1}-(l_{k+1}+ (b_{k}-2l_{k})), (-1)^{A_{k}-B_{k}} \eta_{k+1});  \]
    \item [(b)] If $\epsilon=1$ and $b_{k+1}- 2l_{k+1} \geq  2(b_{k}-2l_{k})$,
    then
    \[ (l_{k+1}', \eta_{k+1}')= (l_{k+1}+ (b_{k}-2l_{k}), (-1)^{A_{k}-B_{k}+1} \eta_{k+1});  \]
    \item [(c)] If $\epsilon=-1$, then
    $(l_{k+1}', \eta_{k+1}')= (l_{k+1}- (b_{k}-2l_{k}), (-1)^{A_{k}-B_{k}+1} \eta_{k+1}).$
\end{enumerate}
\end{enumerate}
Finally, we define $R_{k}(\EE)= \EE^{\rho} \cup R_{k}(\EE_{\rho})$.

If $\gg$ is another admissible order on $I_{\rho}$, then we deform  $\EE_{\rho}$ into
$\{([A_i,B_i]_{\rho},(l_i)_{\gg}, (\eta_i)_{\gg})\}_{i \in (I_{\rho}, \gg)}$ 
by applying a sequence of row exchanges on $\EE_{\rho}$. We shall denote the resulting extended multi-segment by $\EE_{\rho, \gg}$.

\end{defn}

Next, we recall the definition of the operators $sh_j^{d}, add_j^{d}$ on extended multi-segments. These operators do not preserve representations, but are useful in constructing new extended multi-segments.

\begin{defn}[Shift, Add]
Let $\EE = \cup_{\rho}\{ ([A_i,B_i]_{\rho}, l_i, \eta_i) \}_{i \in (I_\rho,>)}$ be an extended multi-segment. For $j \in I_{\rho'}$ and $d \in \Z$, we define the following operators. 

\begin{enumerate}
    \item [1.] Define $sh_j^{d}(\EE)= \cup_{\rho}\{ ([A_i',B_i']_{\rho}, l_i, \eta_i) \}_{i \in (I_\rho,>)}$ with 
    \[ [A_i',B_i']_{\rho}= \begin{cases}
    [A_i+d,B_i+d]_{\rho} & \text{ if }\rho=\rho' \text{ and } i = j,\\
     [A_i,B_i]_{\rho} & \text{ otherwise. }\end{cases} \]
    Set $sh^d_{\rho'}=\sum_{j\in I_{\rho'}} sh_j^{d}$ and $sh^d:=\sum_{\rho'} sh^d_{\rho'}$.
     \item [2.] Define $add_j^{d}(\EE)= \cup_{\rho}\{ ([A_i',B_i']_{\rho}, l_i', \eta_i) \}_{i \in (I_\rho,>)}$ with 
    \[ ([A_i',B_i']_{\rho},l_i')= \begin{cases}
    ([A_i+d,B_i-d]_{\rho},l_i+d) & \text{ if }\rho=\rho' \text{ and } i = j,\\
     ([A_i,B_i]_{\rho},l_i) & \text{ otherwise. }\end{cases} 
    \]
    Set $add^d_{\rho'}=\sum_{j\in I_{\rho'}} add_j^{d}$ and $add^d:=\sum_{\rho'} add^d_{\rho'}$.
\end{enumerate}
It is immediate that the operators commute with each other, so, we denote the composition by summation. 
We only use this notation in the case that the resulting object is still an extended multi-segment.
\end{defn}

With the row exchange, shift, and add operators, we recall the following explicit criterion on an extended multi-segment $\EE$ such that $\pi(\EE)\neq 0$, which is a reformulation of \cite[\S8]{Xu21b}.

\begin{thm}[{\cite[Theorems 3.6, 4.4]{Ato20b}}]\label{thm non-vanishing}
Let $\EE=\cup_{\rho} \{ ([A_i,B_i]_{\rho},l_i,\eta_i)\}_{i \in (I_{\rho},>)}$ be an extended multi-segment such that for any $\rho$, if there exists $i \in I_{\rho}$ with $B_{i}<0$, then the admissible order on $I_{\rho}$ satisfies ($P'$).
\begin{enumerate}
    \item [(i)] We have $\pi(\EE) \neq 0$ if and only if both of the following conditions hold.
    \begin{enumerate}
        \item [(a)] For any integer $d $ such that $ B_i+d \geq 0$ for any $i \in \sqcup_{\rho} I_{\rho}$, we have $ \pi(sh^d (\EE))\neq 0$.
        \item [(b)] For any $i \in \sqcup_{\rho} I_{\rho}$, we have
        \[ B_i+l_i \geq \begin{cases}  0 & \text{ if }B_i \in \Z, \\
    \frac{1}{2} & \text{ if } B_i \not\in \Z \text{ and }\eta_i= (-1)^{\alpha_i+1},\\
        -\frac{1}{2} & \text{ if } B_i \not\in \Z \text{ and }\eta_i= (-1)^{\alpha_i},
        \end{cases} \]
        where 
        $\alpha_i:= \sum_{j < i }A_j+B_j+1.$
    \end{enumerate}
    \item [(ii)] Assume that $B_i\geq 0$ for any $i \in \sqcup_{\rho} I_{\rho}$. Then $\pi(\EE)\neq 0$ if and only if the following hold for any admissible order $\gg$ on $I_{\rho}$. Write
    $\EE_{\rho,\gg}= \{ ([A_i,B_i]_{\rho},l_i,\eta_i)\}_{i \in (I_{\rho},>)}$.
    For any adjacent $k \ll k+1$, let $\epsilon= (-1)^{A_k-B_k}\eta_k \eta_{k+1}.$ 
    \begin{enumerate}
        \item [(1)] If $ A_k\leq  A_{k+1}$, $B_k \leq B_{k+1}$, then
        $$ \begin{cases}
        \epsilon=1 &\Rightarrow \,\,\,\, B_k+l_k \leq B_{k+1}+l_{k+1}, A_k-l_k \leq A_{k+1}-l_{k+1},\\
        \epsilon=-1 &\Rightarrow \,\,\,\,  A_k-l_k < B_{k+1}+l_{k+1}.
        \end{cases} $$
        \item [(2)] If $[A_k,B_k]_{\rho}\subset [A_{k+1},B_{k+1}]_{\rho}$, then
        $$ \begin{cases}
        \epsilon=1 &\Rightarrow \,\,\,\, 0 \leq l_{k+1} -l_{k} \leq b_{k+1}-b_k,\\
        \epsilon=-1 &\Rightarrow \,\,\,\,  l_k+l_{k+1} \geq b_k.    \end{cases} $$
        \item [(3)] If $[A_k,B_k]_{\rho}\supset [A_{k+1},B_{k+1}]_{\rho}$, then
        $$ \begin{cases}
        \epsilon=1 &\Rightarrow \,\,\,\, 0 \leq l_{k} -l_{k+1} \leq b_{k}-b_{k+1},\\
        \epsilon=-1 &\Rightarrow  \,\,\,\, l_k+l_{k+1} \geq b_{k+1}.    \end{cases} $$
    \end{enumerate} 
\end{enumerate}
\end{thm}

The next operator is called the \emph{union-intersection}, which would change the corresponding local Arthur parameters if acting non-trivially. 

\begin{defn}[{\cite[Section 5.2]{Ato20b}}, Union-intersection]\label{ui def}
 Let $\EE$ be an extended multi-segment with
$\EE_{\rho}=\{([A_i,B_i]_{\rho},l_i,\eta_i)\}_{i \in (I_{\rho},>)}.$
For $k< k+1$ adjacent in $(I_{\rho},>)$, we define an operator $ui_k$, called the \emph{union-intersection}, on $\EE$ as follows. 
  Denote $\epsilon=(-1)^{A_k-B_k}\eta_k \eta_{k+1}.$ If $A_{k+1}>A_k$, $B_{k+1}>B_k$ and any of the following cases holds:
\begin{enumerate}
    \item [{Case 1}.] $ \epsilon=1$ and $A_{k+1}-l_{k+1}=A_k-l_k,$
    \item [{Case 2}.] $ \epsilon=1$ and $B_{k+1}+l_{k+1}=B_k+l_k,$
    \item [{Case 3}.] $ \epsilon=-1$ and $B_{k+1}+l_{k+1}=A_k-l_k+1,$
\end{enumerate}
we define that 
$ui_{k}(\EE_{\rho})=\{ ([A_i',B_i']_{\rho},l_i',\eta_i')\}_{i \in (I_{\rho}, >)},$
where $ ([A_i',B_i']_{\rho},l_i',\eta_i')=([A_i,B_i]_{\rho},l_i,\eta_i)$ for $i \neq k,k+1$, and $[A_k',B_k']_{\rho}=[A_{k+1},B_k]_{\rho}$, $[A_{k+1}',B_{k+1}']_{\rho}=[A_k,B_{k+1}]_{\rho}$, and $( l_k', \eta_k', l_{k+1}',\eta_{k+1}' )$ are given case by case as follows:
\begin{enumerate}
    \item[$(1)$] in Case 1, $( l_k', \eta_k', l_{k+1}',\eta_{k+1}' )= (l_k,\eta_k, l_{k+1}-(A_{k+1}-A_k), (-1)^{A_{k+1}-A_k}\eta_{k+1})$;
    \item [$(2)$] in Case 2, if $b_k-2l_k \geq A_{k+1}-A_k$, then
    \[( l_k', \eta_k', l_{k+1}',\eta_{k+1}' )= (l_k+(A_{k+1}-A_k),\eta_k, l_{k+1}, (-1)^{A_{k+1}-A_k}\eta_{k+1}),\]
    and if $b_k-2l_k < A_{k+1}-A_k$, then
    $( l_k', \eta_k', l_{k+1}',\eta_{k+1}' )= (b_k-l_k,-\eta_k, l_{k+1}, (-1)^{A_{k+1}-A_k}\eta_{k+1});$
    \item [$(3)$] in Case 3, if $l_{k+1} \leq  l_k$, then
    $( l_k', \eta_k', l_{k+1}',\eta_{k+1}' )= (l_k,\eta_k, l_{k+1}, (-1)^{A_{k+1}-A_k}\eta_{k+1})$, and 
    if $l_{k+1}> l_{k}$, then
    $( l_k', \eta_k', l_{k+1}',\eta_{k+1}' )= (l_k,\eta_k, l_{k}, (-1)^{A_{k+1}-A_k+1}\eta_{k+1});$
    \item [$(3')$] if we are in Case 3 and $l_k=l_{k+1}=0$, then we delete $ ([A_{k+1}',B_{k+1}']_{\rho},l_{k+1}',\eta_{k+1}')$ from $ui_k(\EE_{\rho})$.
\end{enumerate}
\end{defn}

The union-intersection operator can be extended to non-adjacent extended segments as follows.

\begin{defn} \label{def ui}
Let $\EE$ be an extended multi-segment with
$\EE_{\rho}=\{ ([A_i,B_i]_{\rho},l_i,\eta_i)\}_{i\in (I_{\rho,>})}.$
Given $i,j \in I_{\rho}$, we define that $ui_{i,j}(\EE_{\rho})=\EE_{\rho}$ unless the following conditions hold:
\begin{enumerate}
    \item [1.]$ A_i< A_j$, $B_i <B_j$, and $j\gg i$ are adjacent under some admissible order $\gg$ on $I_{\rho}$. 
    \item [2.] The operator $ui_i$ is applicable on $\EE_{\rho,\gg}$.
\end{enumerate}
In this case, we define that $ui_{i,j}(\EE_{\rho}):=(ui_{i}(\EE_{\rho,\gg}))_{>}$, so that the admissible order of $ui_{i,j}(\EE_{\rho})$ and $\EE_{\rho}$ are the same (if the $ui_i$ is of type $3'$, then we delete the $j$-th row). Finally, we define $ui_{i,j}(\EE):= \EE^{\rho} \cup ui_{i,j}(\EE_{\rho})$.

We say that $ui_{i,j}$ is applicable on $\EE$ if $ui_{i,j}(\EE) \neq \EE$. Furthermore, we say that $ui_{i,j}$ is of type 1, 2, 3, or $3'$ if the corresponding operator $ui_i$ is of type 1, 2, 3, or $3'$, respectively, in Definition \ref{ui def}.
\end{defn}

Let $\pi$ be an irreducible admissible representation of $G_n.$ Aubert showed that there exists $\varepsilon\in\{\pm 1\}$ such that the virtual representation defined by
$$
\widehat{\pi}:=\varepsilon\sum_P (-1)^{\mathrm{dim}(A_P)}[\mathrm{Ind}_{P}^{G_n}(Jac_P(\pi))]
$$
is an irreducible representation (\cite{Aub95}). The above sum is taken over all standard parabolic subgroups $P$ of $G_n$, $A_P$ denotes the maximal split torus in the center of the Levi subgroup of $P$, and $Jac_P$ denotes the Jacquet module along $P.$ We say that $\widehat{\pi}$ is the Aubert-Zelevinsky dual or Aubert-Zelevinsky involution of $\pi.$ 

The next operator is known as the dual, which sends a representation corresponding to an extended multi-segment to its Aubert-Zelevinsky dual.

\begin{defn}[{\cite[Definition 6.1]{Ato20b}}, Dual]\label{dual segment}
Let $\EE= \cup_\rho \{([A_i,B_i]_{\rho},l_i,\eta_i)\}_{i\in (I_\rho, >)}$ be an extended multi-segment such that the admissible order $>$ on $I_{\rho}$ satisfies $(P')$ for all $\rho$. We define 
$$dual(\EE):=\cup_{\rho}\{([A_i,-B_i]_{\rho},l_i',\eta_i')\}_{i\in (I_\rho, >')},$$ as follows:
\begin{enumerate}
    \item The order $>'$ is defined by $i>'j$ if and only if $j>i.$ 
    \item Set \begin{align*}
l_i'=\begin{cases}
l_i+B_i  & \mathrm{if} \, B_i\in\mathbb{Z},\\
 l_i+B_i+\frac{1}{2}(-1)^{\alpha_{i}}\eta_i  & \mathrm{if} \, B_i\not\in\mathbb{Z},
\end{cases}\quad {\rm and}\quad 
\eta_i'=\begin{cases}
(-1)^{\alpha_i+\beta_i}\eta_i  & \mathrm{if} \, B_i\in\mathbb{Z},\\
 (-1)^{\alpha_i+\beta_i+1}\eta_i  & \mathrm{if} \, B_i\not\in\mathbb{Z},
\end{cases}
\end{align*}
where $\alpha_{i}=\sum_{j\in I_\rho, j<i}a_j,$ and $\beta_{i}=\sum_{j\in I_\rho, j>i}b_j,$ $a_j=A_j+B_j+1$, $b_j=A_j-B_j+1$.
\item When $B_i\not\in\mathbb{Z}$ and $l_i=\frac{b_i}{2}$, set $\eta_i=(-1)^{\alpha_i+1}.$
\end{enumerate}
Finally, we define $dual(\EE_{\rho}):= (dual(\EE))_{\rho}$.
\end{defn}


\begin{thm}[{\cite[Theorem 6.2]{Ato20b}}]\label{thm dual}
    Let $\EE \in \Rep^{(P')}$. Then $\pi(dual(\EE))= \widehat{\pi(\EE)}$ holds.
\end{thm}

The last operator we need to exhaust $\Psi(\EE)$ is known as the partial dual. 

\begin{defn}[{\cite[Definition 6.5]{HLL22}}, Partial dual]\label{def partial dual}
Let $\EE\in \Rep^{(P')}$, and let $$\FF:=\EE_{\rho}= \{([A_i,B_i]_{\rho},l_i,\eta_i)\}_{i \in (I_{\rho},>)}.$$
For $i \in I_{\rho}$, denote 
\[\alpha_i= \sum_{ j<i} (A_j+B_j+1)\quad {\rm and}\quad {\beta_i= \sum_{j> i} (A_j-B_j+1)}.\]
Assume that there exists $k \in I_{\rho}$ such that
\begin{enumerate}
    \item [(1)] $B_k=\half{1},l_k=0$;
    \item [(2)] $(-1)^{\alpha_k}\eta_k= -1$;
    \item [(3)] for any $i < k$, $B_i < \half{1}$.
\end{enumerate}
Then we define $dual_k^{+}(\FF)$ as follows. We write the decomposition
\[ \FF= \FF_1 + \{([A_k,1/2]_{\rho},0,\eta_k)\} + \FF_2,\]
where $\FF_1=\FF_{<1/2}$ and $\FF_2=\FF_{>1/2}$, and then write
\[dual(\FF)= \widetilde{\FF_2} + \{([A_k,-1/2]_{\rho},0,(-1)^{\beta_k}) \}+ \widetilde{\FF_1},\]
where $\widetilde{\FF_2}=(dual(\FF))_{<-1/2}$ and $\widetilde{\FF_1}=(dual(\FF))_{>-1/2}$. Next, write
\[ dual(\FF'):=dual(\widetilde{\FF_2} + \{([A_k,1/2]_{\rho},0,(-1)^{\beta_k+1})\} + \widetilde{\FF_1})= \widetilde{\widetilde{\FF_1}} + \{([A_k,-1/2]_{\rho},0,-\eta_k)\}+ \widetilde{\widetilde{\FF_2}},\]
where $\widetilde{\widetilde{\FF_1}}=(dual(\FF'))_{<-1/2}$ and $\widetilde{\widetilde{\FF_2}}=(dual(\FF'))_{>-1/2}$. Then we define
\[ dual_k^{+}( \FF):= \widetilde{\widetilde{\FF_1}} + \{([A_k,-1/2]_{\rho},0,-\eta_k)\} + \FF_2, \]
and say that $dual_k^{+}$ is applicable on $\FF$.

Assume that $dual(\FF)$ satisfies above conditions (1) -- (3). Then we define
\[ dual_k^{-}(\FF):= dual \circ dual_{k}^{+} \circ dual (\FF), \]
and say that $dual_k^{-}$ is applicable on $\FF$.
We call this operators $dual_k^{+}$, $dual_k^{-}$ the \emph{partial dual}, and denote by $dual_k$ if there is no ambiguity. 

Finally, we define $dual_k(\EE):= \EE^{\rho} \cup dual_k(\EE_{\rho}).$ 
\end{defn}

Next, we introduce certain collections of the operators defined above.

\begin{defn}\label{def raising operator}\ 
\begin{enumerate}
    \item We say that an operator $T$ defined above is a \emph{raising operator} if it is of the form  
$ ui_{i,j}^{-1}$, $dual \circ ui_{j,i} \circ dual,$ or $dual_k^{-}$.
\item We say that an extended multi-segment $\EE \in \Rep$ is \emph{absolutely maximal} if there is no raising operator applicable on $\EE$.
\end{enumerate}
\end{defn}

 The following is part of the main results in \cite{HLL22}.

\begin{thm}[{\cite[Theorems 1.4, 1.7]{HLL22}}]\label{thm operator}
Assume that $\EE,\EE' \in \Rep$.
\begin{enumerate}
    \item That $\pi(\EE)=\pi(\EE')$ holds if and only if there exists a sequence of raising operators $\{T_i\}_{i=1}^r$, such that
\[ \EE= T_1^{-1} \circ \cdots \circ T_m^{-1} \circ T_{m+1}\circ \cdots \circ T_{r}(\EE'),\]
up to row exchanges.
\item There exists a unique (up to row exchanges) absolutely maximal extended multi-segment $\EE^{|max|}$ such that $\pi(\EE^{|max|})= \pi(\EE)$.
\end{enumerate}
\end{thm}

By Part (2) of the above theorem, for any $\pi \in \Pi_{A, \, gp}(G_n)$, we define ``the" local Arthur parameter $\psi^{max}(\pi)$ of $\pi$ by
$\psi^{max}(\pi):= \psi_{\EE^{|max|}},$ 
where $\EE^{|max|}$ is the unique (up to row exchanges) absolutely maximal extended multi-segment such that $ \pi= \pi(\EE^{|max|})$. This definition extends to any $\pi \in \Pi_{A}(G_n)$ by Theorem \ref{thm red from nu to gp}. The parameter $\psi^{max}(\pi)$ is the ``most tempered" member inside $\Psi(\pi)$. See \cite[\S 11]{HLL22} for more details. This distinguished member plays an important role in \cite{HLLZ25} and \S \ref{sec new algorithm} below. 

We remark that given $\EE$, the set 
\[ \mathscr{E}(\EE)/\text{(row exchanges)}:= \{ \EE' \ | \ \pi(\EE') \cong \pi(\EE)  \}/\text{(row exchanges)} \]
can be computed by applying \cite[Theorem 7.4]{HLL22}, and the distinguished member $\EE^{|max|}$ can be computed by applying \cite[Corollary 10.7(ii)]{HLL22}. 
The corresponding sage codes have been implemented and are available here: 

\quad

\url{https://github.com/ChiHengLo/Intersection-of-local-Arthur-packets}

\quad

The following are some identities between these operators.

\begin{lemma}[{\cite[Lemma 3.31(i)(ii), and Corollary 5.5]{HLL22}}]\label{lem equalities of operators}
Let 
$\EE=\cup_{\rho} \{([A_i,B_i]_{\rho},l_i,\eta_i)\}_{i \in (I_\rho, >)}$
be an extended multi-segment. For $t,s \in \Z$ and $i,j,k \in I_{\rho}$, the following hold.
\begin{enumerate}
  \item Let $T \in \{sh_i^{t}, add_i^{s}\}$. Suppose $R_{k}(\EE) \neq \EE$ and $R_{k}(T(\EE))\neq T(\EE)$. Then $T \circ R_k (\EE)= R_k \circ T(\EE).$
    In other words, $sh_i^{t}$, $add_i^{s}$ commute with $R_{k}$.
    \item We have that  
        $dual \circ sh_{k}^t(\EE)= add_{k}^t \circ dual(\EE)$ 
        if any side of the equation is still an extended multi-segment.
        \item Assume that $\EE \in \Rep^{(P')}$ and $ ui_{i,j}$ is applicable on $\EE$ of type not $3'$. Then we have 
        \[ dual \circ ui_{j,i} \circ dual \circ ui_{i,j}(\EE)=\EE. \]
\end{enumerate}
\end{lemma}

Lemma \ref{lem equalities of operators}(3) shows that $dual\circ ui_{j,i}\circ dual$ is the inverse operator to $ui_{i,j}$ when it is of type not $3'.$ We recall the inverse when it is of type $3'$ below. 

\begin{lemma}[{\cite[Corollary 5.8]{HLL22}}]\label{cor ui inverse type 3'}
Assume that $\EE=\EE^{\rho} \cup \EE_{\rho} \in \Rep$. Suppose there exists an admissible order $\gg$ of $\EE_{\rho}$ such that 
$\EE_{\rho,\gg}= \{( [A_i,B_i]_{\rho},l_i', \eta_i')\}_{i \in (I_{\rho}, \gg)}$
with $l_j'=0$ for some $j \in I_{\rho}$. Let 
    \[\mathcal{F}_1=\{ ([A_i,B_i]_{\rho},l_i',\eta_i')\}_{i \ll j }\quad {\rm and}\quad  \mathcal{F}_2=\{ ([A_i,B_i]_{\rho},l_i',\eta_i')\}_{i \gg j}.\]
    For $0 \leq r \leq A_j-B_j -1$ such that $2 B_j+r \geq 0$, set
    \begin{align*}
         \EE_{\rho,r}:= \mathcal{F}_1 + \{([B_j+r, B_j]_{\rho},0,\eta_j')\} + \{([A_j,B_j+r+1]_{\rho},0, (-1)^{r+1}\eta_j')\} +\mathcal{F}_2.
    \end{align*}
    If the total order of $\EE_{\rho,r}$ satisfies ($P$), then for any admissible order $\gg'$ of $\EE_{\rho,r}$ satisfying $(P')$, we have  
    \[ \pi(\EE^{\rho} \cup (\EE_{\rho, r})_{\gg'}) \cong \pi(\EE). \]
\end{lemma}

\subsection{A technical lemma}
At the end of this section, we state an observation from the definitions of the operators above which will be used in \S\ref{sec new algorithm}.


\begin{lemma}\label{lem add^1 raising}
    Let $\EE\in \Rep^{(P')}$. Suppose that $\EE_{\rho}= \EE_{\rho,1}+ \EE_{\rho,2}+ \EE_{\rho,3}$ where $\EE_{\rho,i}= \{([A_j,B_j]_{\rho}, l_j, \eta_j)\}_{j \in (I_{\rho,i},>)}$ satisfying the following conditions:
    \begin{enumerate}
        \item There exists a $B$ such that $B_{j_1}< B$ (resp. $B_{j_2}=B$, $B_{j_3}>B$) for any $j_1 \in I_{\rho,1}$ (resp. $j_2 \in I_{\rho,2}$, $j_3 \in I_{\rho,3}$).
        \item For any $j_1 \in I_{j_1}$ and $j_2 \in I_{j_2}$, we have $[A_{j_1}, B_{j_1}]_{\rho} \not\supseteq [A_{j_2}, B_{j_2}]_{\rho}$.
    \end{enumerate}   
    Let $\EE^+:=\EE^{\rho} \cup ( \EE_{\rho,1}+ add^1(\EE_{\rho,2})+ \EE_{\rho,3} )$. Suppose that $\pi(\EE^+)\neq 0$. Then any raising operator that is applicable on $\EE^+$ is also applicable on $\EE$.
\end{lemma}
\begin{proof}
    We verify it case by case for $ui^{-1}$, $dual \circ ui_{j,k} \circ dual$ and $dual_k^{-}$.  Let $I_{\rho}:= I_{\rho,1} \sqcup I_{\rho,2} \sqcup I_{\rho,3}$ in the rest of the proof. We let $I_{\rho}^+= I_{\rho,1}^+ \sqcup I_{\rho,2}^+ \sqcup I_{\rho,3}^+$ to denote the index set of $\EE^{+}_{\rho}$. It is identical to $I_{\rho}$, but the admissible orders on them (Definition \ref{def multi-segment}(2)) are different.
   
    First, suppose there is a $ui^{-1}$ of type 3$'$ applicable on $\EE^+$, which is applied to some $k \in I_{\rho}^+$ such that $(l_k)_\gg=0$, where $(\EE^+)_{\rho, \gg}= \{([A_j, B_j]_{\rho}, (l_j)_{\gg}, (\eta_j)_{\gg})\}_{j \in (I_{\rho}^+,\gg)}$ for some admissible order $\gg$ on $I^+_{\rho}$, as described in Lemma \ref{cor ui inverse type 3'}. Then since $add$ commutes with row exchanges (Lemma \ref{lem equalities of operators}(1)), we must have $(l_j)_\gg \geq 1$ for any $j \in I_{\rho,2}^+$. We conclude that $k \not\in I_{\rho,2}^+$. Hence, this $ui^{-1}$ of type 3$'$ can be applied to $\EE$ as well.

    Next, we deal with $dual \circ ui_{j,k} \circ dual$, which also covers $ui^{-1}$ not of type 3$'$ by Lemma \ref{lem equalities of operators}(3).
    Write
     \begin{align}\label{eq raising operator add^1}
         dual(\EE)_{\rho}= \widetilde{\EE_{\rho,3}}+ \widetilde{\EE_{\rho,2}} + \widetilde{\EE_{\rho,1}},\ \ dual(\EE^+)_{\rho}= \widetilde{\EE_{\rho,3}}+ sh^1(\widetilde{\EE_{\rho,2}}) + \widetilde{\EE_{\rho,1}},
     \end{align}
        by Lemma \ref{lem equalities of operators}(2), where $\widetilde{\EE_{\rho,i}}= \{([A_j,-B_j]_{\rho}, l_j', \eta_j')\}_{j \in (I_{\rho,i}, >')}$ as in Definition \ref{dual segment}. Suppose that $ ui_{j,k}$ is applicable on $dual(\EE^+)$. We claim that $j,k \not\in I_{\rho,2}^+$, or equivalently, $(j, k) \not\in I_{\rho,2}^+ \times I_{\rho, 1}^+  \sqcup I_{\rho,3}^+ \times I_{\rho, 2}^+$. Indeed, if $(j, k) \in I_{\rho,2}^+ \times I_{\rho, 1}^+ $, then Condition (2) implies that $[A_{j},-B_j]_{\rho} \supseteq [A_{k}, -B_k]_{\rho}$; hence $ui_{j,k}$ is not applicable on $dual(\EE^+)$. If $(j, k) \in I_{\rho,3}^+ \times I_{\rho, 2}^+ $. Let $\gg$  be an admissible order on $I_{\rho}^+$ such that $j$ and $k$ are adjacent under $\gg$. Then $\gg$ is also an admissible order for $I_{\rho}$ by Condition (1). Then under $\gg$, the adjacent extended segments of $dual(\EE^+)$ (resp. $dual(\EE)$) are of the form
    \[ \{([A_j, -B_j]_{\rho}, l_j'', \eta_j''),([A_k+1, -B_k+1]_{\rho}, l_k'', \eta_{k}'')\}\ \ \text{(resp. }\{([A_j, -B_j]_{\rho}, l_j'', \eta_j''),([A_k, -B_k]_{\rho}, l_k'', \eta_{k}'')\} \text{)}.\]
        According to Definition \ref{ui def}, if $ui$ is applicable on the first pair above,
        then the second pair violates Theorem \ref{thm non-vanishing}(ii)(1), a contradiction. We conclude that $ui_{j,k}$ is not applicable on $dual(\EE^+)$ for any $(j, k) \in I_{\rho,3}^+ \times I_{\rho, 2}^+ $ as well. This completes the verification for the operator  $dual \circ ui_{j,k} \circ dual$.

    Finally, suppose that $dual_k^{-}$ is applicable on $\EE^+$, or equivalently, $dual_k^+$ is applicable on $dual(\EE^+)$. We claim that $k \not\in I_{\rho,2}^+$. Suppose the contrary. We may assume that $k$ is the largest element in $(I_{\rho,2}^+,>)$ after row exchanges if necessary, which becomes the smallest element in $(I_{\rho,2}^+,>')$ after taking dual. The corresponding extended segments in  $sh^1(\widetilde{\EE_{\rho,2}})$ and $\widetilde{\EE_{\rho,2}}$ in \eqref{eq raising operator add^1} are 
    \[\{([A_{k}+1, -B_k+1]_{\rho}, l_{k}', \eta_k'),([A_{k}, -B_k]_{\rho}, l_{k}', \eta_k')\}.\]
    If the conditions (1) and (2) in Definition \ref{def partial dual} hold for the first extended segment, then Theorem \ref{thm non-vanishing}(i)(b) fails for the second extended segment, a contradiction. We conclude that $k \not\in I_{\rho,2}^+$. This completes the verification of this case and the proof of the lemma.
\end{proof}

\section{\texorpdfstring{A candidate set for the unitary dual}{}}\label{sec A+}

Let $G$ be any connected reductive group defined over $F$ and $M$ be any Levi subgroup. 
An unramified character of $M$ is a character $\chi:M\rightarrow\mathbb{C}^\times$ that is trivial on all compact subgroups.
An unramified character $\chi$ of $M$ is called \emph{real} if it is positively real-valued. Let $X_{un,\R}(M)$ be the set of all real unramified characters of $M.$
In this section, we define a candidate set for the unitary dual of $G$ as follows. 

\begin{defn}\label{def closure A}
Let $G$ be any connected reductive group defined over $F$. Assume that the local Arthur conjecture as in {\cite[Conjectures 6.1 and 6.2]{Art89}} holds for any Levi subgroup $M$ of $G$. Let $\Pi_{{A}}(G)$ (resp. $\Pi_{{A}}(M)$) be the set of all Arthur representations of $G$ (resp. $M$).

\begin{enumerate}
    \item [(a)] For $k \in \Z_{\geq 0}$, we define $ \Pi_{\overline{A}}^{(k)}(G)$ (and similarly $ \Pi_{\overline{A}}^{(k)}(M)$ for any Levi subgroup $M$ of $G$) inductively as follows. Set $ \Pi_{\overline{A}}^{(0)}(G):= \Pi_{A}(G)$. For $k \geq 1$, and any $\pi \in \Pi(G)$, we define $\pi \in \Pi_{\overline{A}}^{(k)}(G)$ if there exists a quadruple
    \begin{equation}\label{quadruple}
        (M, \tau\in\Pi_A(M), \chi_0 \in X_{un,\R}(M), \chi_1 \in X_{un,\R}(M)),
    \end{equation}
    satisfying the following conditions.
    \begin{enumerate}
        \item [(1)] The character $\chi_{0}$ lies in the set 
        \[{U:= \{\chi\in X_{un,\R}(M) \ | \ \mathrm{Ind}_P^G\tau \otimes \chi \ \mathrm{is \ irreducible\ and \ Hermitian}\},}\]
        and there exists a Levi subgroup $M'$ of $G$ and $\tau_{M'} \in \Pi_{\overline{A}}^{(k-1)}(M')$ such that       
        $\Ind_{P}^G \tau \otimes \chi_0 = \Ind_{P'}^G \tau_{M'}$, where $P$ and $P'$ are parabolic subgroups with Levi subgroups $M$ and $M',$ respectively. 
        \item [(2)] The character $\chi_1$ lies in the connected component of $U$ containing $\chi_0$ and $\pi= \Ind_{P}^{G} \tau \otimes \chi_1$.
    \end{enumerate}
Finally, we let $ \Pi_{\overline{A}}(G):= \bigcup_{k \in \Z_{\geq 0}} \Pi_{\overline{A}}^{(k)}(G)$. 

\item [(b)] We say that $\pi \in \Pi_{\overline{A}}^{\lim}(G)$ if there exists a quadruple as in \eqref{quadruple} satisfying the following conditions.
\begin{itemize}
    \item [($1'$)] The character $\chi_{0}$ lies in the open set $U$ and $\Ind_{P}^G \tau \otimes \chi_0 \in \Pi_{\ol{A}}(G)$.
    
    \item [($2'$)] The character $\chi_1$ lies in closure of the connected component of $U$ containing $\chi_0$ and $\pi$ is a subquotient of $\Ind_{P}^{G} \tau \otimes \chi_1$.
\end{itemize}

\item [(c)] We define another set $\Pi_{\overline{A}}^{\lim'}(G)$ following the procedure defining 
$\Pi_{\overline{A}}(G)$ except replacing $(2)$ by $(2)'$ in each step. 
\end{enumerate}
\end{defn}

First we obtain the following. 

\begin{thm}\label{AbarAlim}
The following inclusion relations hold.
\[
    \Pi_{{A}}(G)\subset\Pi_{\overline{A}}(G)\subset\Pi_{\overline{A}}^{\lim}(G)\subset \Pi_{\overline{A}}^{\lim'}(G)\subset \Pi_u(G)
    \]
\end{thm}
\begin{proof}
 {First, we show that $\Pi_{\overline{A}}^{(k)}(M)\subset\Pi_u(M)$ for all Levi subgroups $M$ of $G$, by applying induction on $k$. This implies the inclusion $\Pi_{\overline{A}}(M)\subset\Pi_u(M)$.

  When $k=0$, the theory of local Arthur packets we assumed implies that $\Pi_{\overline{A}}^{(0)}(M)= \Pi_A(M) \subset \Pi_u(M)$. For $k >0$, let $(M, \tau, \chi_0, \chi_1)$ be a quadruple as in \eqref{quadruple} corresponding to an arbitrary representation $\pi \in \Pi_{\overline{A}}^{(k)}(G)$, and follow the notation in Definition \ref{def closure A}(a). By induction hypothesis, $ \tau_{M'}$ is unitary, and so is the irreducible unitary induction $  \Ind_{P'}^G \tau_{M'}=\Ind_P^G \tau \otimes \chi_0$ by Proposition \ref{facts on unitary reps}(1). Thus, by Proposition \ref{facts on unitary reps}(4), every representation of the form $\Ind_{P}^G \tau \otimes \chi$ where $\chi$ lies in the same connected component of $U$ containing $\chi_1$ is unitary. In particular, the representation $\pi = \Ind_{P}^G \tau \otimes \chi_1$ is unitary. This completes the induction process that $\Pi_{\overline{A}}^{(k)}(G)\subset\Pi_u(G)$ for all $k$.

The inclusion  $\Pi_{\overline{A}}^{\lim}(G)\subset \Pi_{\overline{A}}^{\lim'}(G)$ is straightforward by definition. 
  The proof that $\Pi_{\overline{A}}^{\lim'}(G)\subset\Pi_u(G)$ is similar as above except that we consider subquotients of the induction $\Ind_{P}^G \tau \otimes \chi$, where $\chi$ lies in the boundary of the connected component of $U$ containing $\chi_0$, which are unitary by Proposition \ref{facts on unitary reps}(5). This completes the proof of the theorem.
  }
\end{proof}

\begin{remark}\label{rmk Tadic construction} \

(1) We recall four types of constructions of unitary representations (see \cite[\S 3]{Tad93} and also Proposition \ref{facts on unitary reps}):
    \begin{enumerate}
        \item [(a)] Unitary parabolic induction.
        \item [(b)] Complementary series.
        \item [(c)] Irreducible subquotients of ends of complementary series.
        \item [(d)] Unitary parabolic reduction.
    \end{enumerate}
    Note that constructions of type (a), (b) and (c) result in non-isolated representations.    Clearly, Definition \ref{def closure A} already includes constructions of type (a), (b) and (c) (restricting to complementary series from Arthur representations). 
    It also covers construction of type (d) in the following sense. Let $\pi= \theta \rtimes \sigma$ be an irreducible parabolic induction where both $\theta$ and $\sigma$ are Hermitian. Assume that $\pi \in \Pi_{\overline{A}}^{(k)}(G_n)$.  Then one can check that $\sigma \in \Pi_{\overline{A}}^{(k')}(G_m)$ for some $k'\leq k$ and $m<n$ by applying induction on $k$. {For $k=0$, see Theorem \ref{thm preservation and independence}(a) below for example.} 

    (2) As a rephrase of 
    Definition \ref{def closure A}, 
    $\Pi_{\overline{A}}^{(k)}(G_n)$ can be inductively obtained as follows: for any parabolic subgroup $P'=M'N'$, consider all irreducible unitary induction $\Ind_{P'}^G \tau_{M'}$ with inducing data $\tau_{M'} \in \Pi_{\overline{A}}^{(k-1)}(M')$. Then $\Pi_{\overline{A}}^{(k)}(G_n)$ consists of all irreducible continuous Hermitian deformations of $\Ind_{P'}^G \tau_{M'}=\Ind_P^G\tau\otimes\chi_0$ via varying $\chi_0$. 
This specifies the form of the complementary series representations we need to consider. 

    (3) By Theorem \ref{AbarAlim}, one sees that $\Pi_{\overline{A}}^{\lim}(G) $ is indeed a subset of $\Pi_u(G)$. We expect that it contains the closure of $\Pi_{\ol{A}}(G)$ under the Fell topology of the unitary dual. We will discuss some examples in the following subsections (see Examples \ref{ex A+}, \ref{ex A+ unit ind}, and \ref{ex corank 3 limit}), which inspire the definition of $\Pi_{\ol{A}}^{\lim}(G)$ and motivate   
    us to conjecture that $\Pi_{\ol{A}}^{\lim}(G)$ is a candidate set of $\Pi_u(G)$ (see Conjecture \ref{conj A+} below).
    
(4) Instead of taking limits at the final step only, in the definition of  $\Pi_{\ol{A}}^{\lim'}(G)$, we take limits at every step. 
As shown in Theorem \ref{AbarAlim}, by definition, $\Pi_{\ol{A}}^{\lim}(G) \subseteq \Pi_{\ol{A}}^{\lim'}(G)$. However, we expect from examples that indeed $\Pi_{\ol{A}}^{\lim}(G) = \Pi_{\ol{A}}^{\lim'}(G)$ (see Conjecture \ref{conj A+} below).
\end{remark}

By Tadi{\'c}'s classification of the unitary dual for $\GL_n$ (\cite{Tad86}), one directly obtains the following theorem.

\begin{thm}[Tadi{\'c}]\label{thm GL Pi A bar = Pi u}
    For $\RG=\GL_n$, we have that $\Pi_{\ol{A}}(G)=\Pi_u(G)=\Pi_{\ol{A}}^{\lim}(G)=\Pi_{\ol{A}}^{\lim'}(G)$.
\end{thm}

\subsection{Definition \ref{def closure A} for classical groups} 

In this section, we show that Definition \ref{def closure A} for split symplectic and odd special orthogonal groups $G_n$ has the following more explicit form. We remark that the analogues for other classical groups (e.g. unitary groups) would follow similarly.

\begin{defn}\label{def closure A+}\

\begin{enumerate}
    \item [(a)] For $k \in \Z_{\geq 0}$, we define subsets $ \Pi_{\overline{A}}^{(k)}(G_n)$ of $\Pi_u(G_n)$ inductively as follows. Set $ \Pi_{\overline{A}}^{(0)}(G_n):= \Pi_{A}(G_n)$. For $k \geq 1$, and any $\pi \in \Pi(G_n)$, we say that $\pi \in \Pi_{\overline{A}}^{(k)}(G_n)$ if there exists a triple 
    \begin{align}\label{eq triple A bar Gn}
        ( \Pi_{\underline{x}}= \Pi_{x_1,\ldots, x_s}= u_{\rho_1}(a_1,b_1)\vert\cdot\rvert^{x_1}\times \cdots \times u_{\rho_s}(a_s,b_s)\vert\cdot\rvert^{x_s} \rtimes \pi_A, \underline{y} \in \R^s, \underline{z} \in \R^s),
    \end{align}
satisfying the following conditions.
\begin{enumerate}
    \item  [(1)]
    $\pi_A \in\Pi_{A, \, gp}(G_{m})$ for some $m <n$, and $\rho_i$'s are irreducible unitary supercuspidal representations of general linear groups (not necessarily selfdual).
    \item [(2)] The point $\underline{y}$ lies in the set 
    \[U:=\{\underline{w} \in \R^s\ | \ \Pi_{\underline{w}} \text{ is irreducible and Hermitian}\}\]
    and { $\Pi_{\underline{y}}= \tau \rtimes \pi^{(k-1)}$  for some unitary representation $\tau$ of $\GL_d(F)$ and $\pi^{(k-1)}\in \Pi_{\overline{A}}^{(k-1)}(G_{n-d})$.} 
        \item [(3)] The point $\underline{z}$ lies in the unique connected component of $U$ containing $\underline{y}$ and $\pi \cong \Pi_{\underline{z}}$.
\end{enumerate}
Finally, we let $ \Pi_{\overline{A}}(G_n):= \bigcup_{k \in \Z_{\geq 0}} \Pi_{\overline{A}}^{(k)}(G_n)$.

\item [(b)] We say that $\pi \in \Pi_{\overline{A}}^{\lim}(G_n)$ if there exists a triple as in \eqref{eq triple A bar Gn} satisfying Condition (1) in Part (a) and
\begin{itemize}
    \item [($2'$)] The point $\underline{y}$ lies in the open set $U$ and  $\Pi_{\underline{y}} \in \Pi_{\overline{A}}(G_n)$.
    
    \item [($3'$)] The point $\underline{z}$ lies in the \textbf{closure} of the unique connected component of $U$ containing $\underline{y}$ and $\pi $ is a subquotient of $ \Pi_{\underline{z}}$.
\end{itemize}

    \item [(c)] We say that $\pi \in \Pi_{\overline{A}}^{\lim'}(G_n)$ if there exists a quadruple as in \eqref{quadruple} satisfying the conditions $(1), (2)$, and $(3')$ above.
\end{enumerate}
\end{defn}

For the groups $G_n$, we explain in the following that $ \Pi_{\overline A}^{\lim} (G_n)$ in Definition \ref{def closure A} and $ \Pi_{\overline A}^{\lim}(G_n)$ in  Definition \ref{def closure A+} are indeed the same sets. Similarly, the sets $ \Pi_{\overline A}^{\lim'} (G_n)$ in Definition \ref{def closure A} and $ \Pi_{\overline A}^{\lim'}(G_n)$ in  Definition \ref{def closure A+} are also the same sets. In the following paragraphs, we temporarily write $\Pi_{\overline{A},1}^{(k)}(G_n)$, $\Pi_{\overline{A},1}(G_n)$ and $\Pi_{\overline{A},1}^{\lim}(G_n)$ for the sets defined in Definition \ref{def closure A}, and write $\Pi_{\overline{A},2}^{(k)}(G_n)$, $\Pi_{\overline{A},2}(G_n)$ and $\Pi_{\overline{A},2}^{\lim}(G_n)$ for the sets defined in Definition \ref{def closure A+}.

For a general quadruple as in \eqref{quadruple}, the Levi subgroup $M$ and the representation $\tau$ take the form
\begin{align}
 \label{general M}  M&=\bigtimes_{i=1}^s \GL_{d_i}(F) \times G_m,\\
   \label{general tau} \tau& = \bigotimes_{i=1}^s\left( u_{\rho_{(i,1)}}(a_{(i,1)},b_{(i,1)})\times \cdots \times u_{\rho_{(i, j_i)}}(a_{(i,j_i)},b_{(i,j_i)})\right) \otimes \pi_A.
\end{align}
Without loss of generality, we may assume that $\pi_A\in\Pi_{A,\,gp}(G_m)$ by Theorem \ref{thm red from nu to gp}.
On the other hand, the triple in \eqref{eq triple A bar Gn} only includes the cases that $j_i=1$ for $i=1,\ldots, s$. Therefore, we have $\Pi_{\overline{A},2}(G_n) \subseteq \Pi_{\overline{A},1}(G_n)$. 

Now we illustrate the opposite inclusion. Take any 
\[\pi= \bigtimes_{i=1}^s\left( u_{\rho_{(i,1)}}(a_{(i,1)},b_{(i,1)})\times \cdots \times u_{\rho_{(i, j_i)}}(a_{(i,j_i)},b_{(i,j_i)})\right)\lvert\cdot\rvert^{x_i} \rtimes \pi_A \in \Pi_{\overline{A},1}^{(k)}(G_n).\]
If $k=0$, then $\pi \in \Pi_A(G)= \Pi_{\overline{A},2}^{(0)}(G_n)$. Thus, we assume that $k\geq 1$. By definition, there exists a path $\{\gamma(t)=( \gamma_1(t),\ldots,\gamma_s(t)) \ | \ t \in [0,1]\} \subset \BR^s$ such that $\gamma(0)= (x_1,\ldots, x_s)$, and 
\[ \pi_t:=   \bigtimes_{i=1}^s\left( u_{\rho_{(i,1)}}(a_{(i,1)},b_{(i,1)})\times \cdots \times u_{\rho_{(i, j_i)}}(a_{(i,j_i)},b_{(i,j_i)})\right)\lvert\cdot\rvert^{\gamma_i(t)} \rtimes \pi_A  \]
is irreducible for all $t \in [0,1]$, and $\pi_{1} = \tau \rtimes \pi_1'$ such that $\tau$ is unitary (hence lies in $\Pi_{\overline{A}}(\GL_{n-m}(F))$ by Theorem \ref{thm GL Pi A bar = Pi u}) and $\pi_1' \in \Pi_{\overline{A},1}^{(k-1)}(G_m)$, for some $m \leq n$. Applying induction on both $k$ and $n$, we assume that $\pi_1' \in \Pi_{\overline{A},2}^{(k')}(G_m)$ for some $k'$.

Let $I:= \bigsqcup_{i=1}^s \{ (i,1),\ldots, (i,j_{i})\}$ and set $x_{(i,l)}:= x_i$. We may rewrite
\begin{align*}
    \pi& = \bigtimes_{i=1}^s\left( u_{\rho_{(i,1)}}(a_{(i,1)},b_{(i,1)})\lvert\cdot\rvert^{x_i}\times \cdots \times u_{\rho_{(i, j_i)}}(a_{(i,j_i)},b_{(i,j_i)})\lvert\cdot\rvert^{x_i}\right) \rtimes \pi_A\\
    &= \bigtimes_{l \in I} u_{\rho_{l}}(a_l, b_l)\lvert\cdot\rvert^{x_l} \rtimes \pi_A.
\end{align*}
For $\epsilon= (\epsilon_l)_{l \in I}\in \BR^{|I|}$, we consider 
\[\Pi_{\epsilon}:=\bigtimes_{l \in I} u_{\rho_{l}}(a_l, b_l)\lvert\cdot\rvert^{x_l+\epsilon_l} \rtimes \pi_A.\]
Then by the irreducibility criterion for $u_{\rho_l}(a_l,b_l) \lvert\cdot\rvert^{y_l} \times u_{\rho_{l'}}(a_{l'},b_{l'}) \lvert\cdot\rvert^{y_{l'}}$  (\cite[Theorem 1.1]{Tad14}) and $ u_{\rho_l}(a_l,b_l) \lvert\cdot\rvert^{y_l} \rtimes \pi_A$ (\cite{Ato22d}) and \cite[Theorem 1.2]{BS25}, it follows that $\Pi_{\epsilon}$ is irreducible for any $\epsilon$ in a small neighborhood of $(0,\ldots, 0)$, since $\pi= \Pi_{(0,\ldots,0)}$ is irreducible. Moreover, we see that $\pi_1= \Pi_{\epsilon}$ for some $\epsilon$ in the connected component of $\{\epsilon' \ | \ \Pi_{\epsilon'} \text{ is irreducible}\}$ that contains $(0,\ldots, 0)$. Therefore, we conclude by definition that $\pi \in \Pi_{\overline{A},2}^{(k'+1)}(G_n) \subseteq \Pi_{\overline{A},2}(G_n)$, since $\pi_1' \in \Pi_{\overline{A},2}^{(k')}(G_m)$. Hence, $\Pi_{\overline{A},1}(G_n) = \Pi_{\overline{A},2}(G_n)$. 
Using similar arguments, we can show that $\Pi_{\overline{A},1}^{\lim}(G_n) = \Pi_{\overline{A},2}^{\lim}(G_n)$.

\subsection{Motivating examples}
In this subsection, we give several examples on various subsets of $\Pi_u(G)$ for $G=G_n$, which motivate our definitions of $\Pi_{\overline{A}}(G)$ and $\Pi_{\overline{A}}^{\lim}(G).$

 We begin by presenting some examples of representations of corank $1$ of $G_n$ that demonstrate the \emph{proper containment} of both: 
$$
\Pi_{A+, \, u}(G_n):=\Pi_{A+}(G_n)\cap\Pi_u(G_n) \subsetneq \Pi_{A+}(G_n) \quad \text{and} \quad \Pi_{A+, \, u}(G_n) \subsetneq \Pi_u(G_n).
$$
Recall that $
\Pi_{A+}(G_n):=\cup_{\psi\in\Psi^+_{1/2}(G)}\Pi_{\psi}$ (see Definition \ref{A parameters}). On the other hand, to guarantee that the representations in $\Pi_{A+}(G_n)$ are Hermitian, it is necessary to introduce a smaller subset $\Psi^{+}_{\text{unit}}(G_n)$, see \eqref{eq def Psi^+_unit}.
Indeed, for any $\psi \in \Psi^{+}_{1/2}(G_n) \setminus \Psi^+_{\mathrm{unit}}(G_n)$, the representations in $\Pi_\psi$ are \emph{non-unitary}, as they fail to be Hermitian by \cite[Theorem 4.2(ii)]{Tad09a}. While it is conjectured that $\Pi_\psi$ consists of unitary representations when $\psi \in \Psi^+_{\mathrm{unit}}(G_n)$, as in \cite[Conjecture 8.3.1]{Art13}, Part (1) of the following example shows that this \emph{need not hold} in general.

\begin{exmp}\label{ex A+}
Let $\rho \in \mathcal{C}^{sd}$ and $\sigma \in \mathcal{C}_{cl}$. 
Consider the representation
$\Pi_x := \rho\vert\cdot\rvert^x \rtimes \sigma$
of $G_n$ with $x \in \mathbb{R}_{\geq 0}$. Let $\alpha = \alpha_{\rho, \sigma} \geq 0$ be the unique reducibility point. Then
\begin{itemize}
    \item $\Pi_x$ is irreducible if and only if $x \neq \alpha$;
    \item $\Pi_x$ is both unitary and irreducible if and only if $0 \leq x < \alpha$.
\end{itemize}
This follows from \cite[\S 7.1]{Tad23}.

\begin{enumerate}
    \item [(1)] Assume that $\alpha=0$. Then for $x \in (0,\infty)$, $\Pi_x$ is irreducible but not unitary. Now, let $\psi$ be a tempered local Arthur parameter such that $\sigma \in \Pi_{\psi}$. For $x \in (0, \half{1})$, we have $\Pi_x\in \Pi_{\psi_x}$, where
    $\psi_x= (\rho \vert\cdot\rvert^{x} + \rho\vert\cdot\rvert^{-x})+ \psi \in \Psi^{+}_{\text{unit}}(G_n).$
    Consequently, for $x \in (0,\half{1})$, we have 
    $$\Pi_x \in \Pi_{A+}(G_n) \setminus \Pi_{A+,\, u}(G_n),$$
    which provides a counter-example to \cite[Conjecture 8.3.1]{Art13}. When $\sigma$ is generic, this case is already covered in \cite[Theorem 8.1]{Sha90}.

    \item [(2)] Assume that $\alpha\geq  \half{1}$. For any $x \in \R_{\geq 0} \setminus \{\alpha\}$, $\Pi_x$ is irreducible, and 
    $$ \Pi_{x} \in \begin{cases}
            \Pi_{u}(G_n) &\text{ if and only if }x \in [0,\alpha),\\
            \Pi_{A+,\, u}(G_n)& \text{ if and only if }x \in [0,\half{1}) \cup ([0,\alpha) \cap \alpha+\Z) ,\\
            \Pi_{A, \, gp}(G_n)=\Pi_{u, \, gp}(G_n)& \text{ if and only if }x \in [0,\alpha) \cap \alpha+\Z,\\
            \Pi_{u, \, bp}(G_n)& \text{ if and only if }x \in [0,\alpha) \cap \alpha+ \half{1}+\Z,\\
            \Pi_{A, \, bp}(G_n)& \text{ if and only if }x \in  \{0\} \cap \alpha + \half{1}+\Z.
        \end{cases}$$
    Here $\Pi_{u, \, bp}(G_n)$ $($resp. $\Pi_{A, \, bp}(G_n))$ is the set of unitary (resp. Arthur) representations of $G_n$ which are of bad parity. The case of $\Pi_{A, \, bp}(G_n)$ follows from Theorem \ref{thm red from nu to gp}, and the case for $\Pi_{A, \, gp}(G_n)$ follows from Theorem \ref{thm AM}. 
This example implies that the inclusion $\Pi_{A,\, bp}(G_n) \subset \Pi_{u,\, bp}(G_n)$ can also be strict.
\end{enumerate}  
\end{exmp}

Example \ref{ex A+} (2) above implies that one reason for the containment $\Pi_{A+,\, u}(G_n) \subset \Pi_{u}(G_n)$ being strict is that the endpoints of the complementary series contained in $\Pi_{A+}(G_n)$ may still be irreducible, and hence could be extended further (see the following example).
This example also suggests that, besides considering the complementary series, we need to consider the unitary parabolic induction as well.

\begin{exmp}\label{ex A+ unit ind}
Assume that $\alpha_{\rho,\sigma}=\half{3}$. The parabolic induction
$\Pi_{x,y}:= \rho\vert\cdot\rvert^{x} \times \rho\vert\cdot\rvert^{y} \rtimes \sigma$, with $(x,y) \in \R_{\geq 0}^2$, 
is irreducible and unitary in the following regions $($Figure \ref{Figure corank 2}, \cite[\S 7.4]{Tad23}$)$: 

\begin{figure}[H]
\begin{center}
\scalebox{.7}{\small
\begin{tikzpicture}[>=stealth]
\filldraw[fill=gray!30, dashed] (0,0) -- (2,0) -- (0,2) -- (0,0) -- cycle;
\filldraw[fill=gray!30, dashed] (0,2) -- (1,3) -- (0,3) -- (0,2) -- cycle;
\filldraw[fill=gray!30, dashed] (2,0) -- (3,1) -- (3,0) -- (2,0) -- cycle;
\node at (0.7,0.7) {$C_0$};
\node at (0.3,2.7) {$C_1$};
\node at (2.7,0.3) {$C_1'$};
\draw[dashed] (0,2) --  (2,4) node[above] {$ y=x+1 $};
\draw[dashed] (2,0) --  (4,2) node[above] {$ y=x-1 $};
\draw[dashed] (0,3) --  (4,3) node[right] {$ y=\half{3} $};
\draw[dashed] (3,0) --  (3,4) node[right] {$ x=\half{3} $};
\draw[dashed] (-0.5,2.5) --  (2.5,-0.5) node[right] {$ x+y=1$};
\draw[->] (0,-0.5) --  (0,4) node[above] {$ y $};
\draw[->] (-0.5,0) --  (4,0) node[right] {$ x $};

\node at (1,1)[black]{$\oldbullet$} ;
\node at (1.5,1)[black]{$P_1$} ;
\node at (0,2)[black]{$\oldbullet$} ;
\node at (0.5,2)[black]{$P_2$} ;
\node at (1,3)[black]{$\oldbullet$} ;
\node at (1,3.5)[black]{$P_3$} ;

\node at (0,3)[black]{$\oldbullet$} ;
\node at (-0.5,3)[black]{$P_4$} ;

\end{tikzpicture}}
\end{center}
\caption{Unitary subquotients of $\Pi_{x,y}$}\label{Figure corank 2}
\end{figure}
\noindent where $C_0= \{ (x,y) \in \R^2_{\geq 0} \ | \  x+y < 1 \},$ $C_1= \{ (x,y) \in \R^2_{\geq 0} \ | \  x+1 < y<3/2 \},$ and $C_1'= \{ (x,y) \in \R^2_{\geq 0} \ | \  y+1 < x<3/2\}.$
We now claim that $\Pi_{x,y}\in \Pi_{\overline{A}}^{(1)}(G_n)$ for $(x,y)\in C_0$, and $\Pi_{x,y}\in \Pi_{\overline{A}}^{(2)}(G_n)$ for $(x,y)\in C_1\cup C_1'$. The first assertion is clear since the representation corresponding to the point $(0,0)$ is of Arthur type. To see that $\Pi_{x,y}\in \Pi_{\overline{A}}^{(2)}(G_n) $ when $(x,y)\in C_1$, we consider the representation 
$\pi_y:= \rho\vert\cdot\rvert^{y} \rtimes \sigma.$ Since $\pi_y$ is irreducible for $y\in (0,\alpha)$ and $\pi_0$ is of Arthur type,  $\pi_y$ lies in ${\Pi_{\overline{A}}^{(1)}}(G_{n-\dim(\rho)})$ by definition. Then consider the parabolic induction 
$\Pi_x:= \rho\vert\cdot\rvert^{x} \rtimes \pi_y$ with $x \in \R$, 
which is irreducible when $x$ lies in the open set  $ U :=\R \setminus \{ \pm \alpha, \pm (y+1), \pm(y-1)\}$ of $\R$. The unique connected component of $U$ that contains the point $x=0$, which corresponds to a unitary parabolic induction, is $\Sigma=(-y+1, y-1)$. Thus, we realize $\Pi_{x,y}$ as the irreducible representation $\Pi_{x}$, where $x \in \Sigma$. This implies that $ \Pi_{x,y}\in \Pi_{\overline{A}}^{(2)}(G_n)$. The above argument is exactly Tadi{\'c}'s proof of the unitarity of $\Pi_{x,y}$ for $(x,y) \in C_1$.  Similarly, for $(x,y)\in C_1'$, we have $\Pi_{(x,y)}\in \Pi^{(2)}_{\overline{A}}(G_n)$. 

  Following Definition \ref{def closure A+}, we explain that the representations on the boundary of $C_0$, $C_1$ are actually in $\Pi_{\overline{A}}(G_n)$. On the dashed line from $P_1$ to $P_2$, we have $\Pi_{x,y}$ is of length 2 and 
\[\Pi_{x,y}= u_{\rho}(1,2)\lvert \cdot\rvert^{\half{y-x}} \rtimes \sigma + u_{\rho}(2,1)\lvert \cdot\rvert^{\half{y-x}} \rtimes \sigma. \]
In particular, these subquotients lie in $ \Pi_{\overline{A}}^{(1)}(G_n)$. Similarly, on the dashed line from $P_2$ to $P_3$ but not including $P_3$, we have 
\[ \Pi_{x,y}= u_{\rho}(1,2)\lvert \cdot\rvert^{\half{x+y}} \rtimes \sigma + u_{\rho}(2,1)\lvert \cdot\rvert^{\half{x+y}} \rtimes \sigma, \]
and these subquotient representations also lie in $ \Pi_{\overline{A}}^{(1)}(G_n)$. The irreducible representations at the point $P_3$ are all of Arthur type. Finally, the induction $\rho\lvert\cdot\rvert^{\half{3}} \rtimes \sigma= \pi_1 + \pi_2$ where $\pi_1,\pi_2$ are both of Arthur type. For $(x,y)$ on the dashed line from $P_3$ to $P_4$ excluding $P_3$, we have
\[ \Pi_{x,y}= u_{\rho}(1,1)\lvert \cdot\rvert^{x} \rtimes \pi_1 + u_{\rho}(1,1)\lvert \cdot\rvert^{x} \rtimes \pi_2.  \]
Hence, these subquotient representations 
 again lie in $ \Pi_{\overline{A}}^{(1)}(G_n)$.

In conclusion, for an irreducible subquotient $\pi$ of $\Pi_{x,y}$ in this example, we see that $\pi\in \Pi_u(G_n)$ if and only if $\pi\in \Pi_{\overline{A}}(G_n).$
\end{exmp}

Motivated by the above example, we define $\Pi_{ind}(G)$ to be the subset of $\Pi(G)$ consisting of representations that are irreducibly induced from an ``essentially" Arthur type representation of a Levi subgroup. Namely,
\begin{align}\label{eq def of cand}
    \Pi_{ind}(G):= \{ \pi \in \Pi(G)\ | \ \pi= \Ind_{P}^{G} \sigma \otimes \nu,\ \text{ for some }\sigma \in \Pi_A(M) \text{ and }\nu \in X_{un,\R}(M)  \}.
\end{align}
Here we recall that $X_{un,\R}(M)$ is the set of all real unramified characters of $M$. 
It is clear from the definition that $\Pi_{\overline{A}}(G) \subset \Pi_{u}(G) \cap 
\Pi_{ind}(G)$.  
 In \S \ref{sec algo cand}, we will provide an algorithm to determine whether a representation is in $\Pi_{ind}(G)$ for $G=G_n$.
From Example \ref{ex A+ unit ind}, it is attempting to guess that the inclusion $\Pi_u(G) \subseteq \Pi_{ind}(G)$ always holds. However, the following example shows that this fails already for representations of corank 3.

\begin{exmp}\label{ex corank 3 limit}
    Write $\alpha= \alpha_{\rho,\sigma}$ for short.    Consider the one-parameter family of representations 
    \[ \Pi_x= L(\rho\lvert\cdot\rvert^{-x-1},\Delta_{\rho}[-x+1,-x])\rtimes \sigma  \]
    for $x \geq 0$. By \cite[Lemma 8.2(1)]{Tad23}, $\Pi_x$ is irreducible if and only if $x \not\in \{ |\alpha-1|, \alpha, \alpha+1 \}$. Moreover, it is unitary if and only if $0 \leq x \leq \alpha-1$.
If $x- \alpha \in \Z$, then one can check that $\Pi_x$ is of Arthur type (see Proposition \cite[Proposition 13.23]{HJLLZ24}). However, if $x -\alpha \not\in \Z$ and $0 \leq x < \alpha-1$, then $\Pi_x$ does not lie in $\Pi_{ind}(G_n)$. 

On the other hand, consider the two-parameter family
\[ \Pi_{x,y}= \rho\lvert\cdot\rvert^{-y} \times \Delta_{\rho}[-x+1,-x]\rtimes \sigma.  \]
Assume for simplicity that $\alpha=\half{5}$. Then $\Pi_{x,y} $ is irreducible and unitary in the following open regions 
\begin{align*}
    U_1&= \{(x,y)\in \R^2\ | \ \max(-x+2,x+1) <y<\alpha,  \},\\
    U_2&=\{(x,y)\in \R^2\ | \ \max(-x-1,x-2 ) < y < \min(-x+2,x+1)\}.
\end{align*}
\begin{figure}[H]
\begin{center}
\scalebox{.7}{\small
\begin{tikzpicture}[>=stealth]
\filldraw[fill=gray!30, dashed] (-.5,2.5) -- (1.5,2.5) -- (.5,1.5) -- (-.5,2.5) -- cycle;
\filldraw[fill=gray!30, dashed] (.5,1.5) -- (-1,0) -- (.5,-1.5) -- (2,0) -- (.5,1.5) -- cycle;
\node at (0.5,2.1) {$U_1$};
\node at (.5,0.5) {$U_2$};
\draw[dashed] (-2,-1) --  (2,3) node[above] {$ y=x+1 $};
\draw[dashed] (-1,3) --  (3,-1) node[right] {$ y=-x+2 $};
\draw[dashed] (-2,2.5) --  (3,2.5) node[right] {$ y=\half{5} $};
\draw[dashed] (0,-2) --  (2.5,0.5) node[right] {$ y=x-2 $};
\draw[dashed] (-2,1) --  (1,-2) node[right] {$ y=-x-1$};
\draw[->] (0,-2) --  (0,3) node[above] {$ y $};
\draw[->] (-2,0) --  (4,0) node[right] {$ x $};
\end{tikzpicture}}
\end{center}
\caption{Regions where $\Pi_{x,y}$ is irreducible and unitary}
\end{figure}
Thus, the one-parameter family $\{\Pi_x\ | \ 0 \leq x < \alpha-1\}$ is a family of irreducible subquotients of $\Pi_{x,y}$ where $(x,y)$ lies in the closures of $U_1$ and $U_2$ intersecting with $y=x+1$. Note that the representations $\Pi_{x,y} $ for $(x,y)\in U_1\cup U_2$ are in $\Pi_{\overline{A}}(G_n)$ by definition. This shows that the unitary representations $\Pi_x$ lie in $\Pi_{\overline{A}}^{\lim}(G_n)$, when $0\le x<\alpha-1$.
\end{exmp}

\subsection{Main conjecture}

Recall that, by definition, we naturally have an inclusion $\Pi_{\overline{A}}^{\lim}(G) \subseteq \Pi_u(G)$. 
Based on the discussion of those examples in the previous subsection, we conjecture that the inclusion is indeed an equality.

\begin{conj}[{Conjecture \ref{conj A+ intro}}]\label{conj A+}
Let $G$ be a connected reductive group and let $\Pi_{\overline{A}}(G), \Pi_{\overline{A}}^{\lim}(G), \Pi_{\ol{A}}^{\lim'}(G)$, $\Pi_{ind}(G)$ be the sets defined in Definition \ref{def closure A+} and \eqref{eq def of cand}. Then
\begin{enumerate}
    \item [(a)] $ \Pi_{\overline{A}}(G)= \Pi_{u}(G) \cap \Pi_{ind}(G)$.
    \item [(b)] $ \Pi_{\overline{A}}^{\lim}(G)=\Pi_{\ol{A}}^{\lim'}(G)= \Pi_{u}(G)$.
\end{enumerate}
\end{conj} 

In \S \ref{algorithm for A bar} below, for $G=G_n$, we will develop an explicit algorithm to determine the set $\Pi_{\overline{A}}(G) $ (Algorithm \ref{alg A bar}), which in turn generates the candidate set $\Pi_{\overline{A}}^{\lim}(G)$ of the unitary dual. Evidence for the exhaustion of this candidate set, and hence for the above conjecture, includes that it holds for many known cases of unitary duals, as discussed in the introduction. We remark that from the definitions of these sets, this conjecture implies that isolated unitary representations are contained in $\Pi_A(G)$, as conjectured by M. Tadi{\'c} (\cite[Conjecture 1.1]{Tad22}).

\subsection{More evidence for Conjecture \ref{conj A+}} Let $G=G_n$ in this subsection.
As further evidence for Conjecture \ref{conj A+}, we show that $\Pi_{{A+,\,u}}(G_n) \subset \Pi_{\overline{A}}(G_n)$. Recall that $\Pi_{A+, \, u}(G_n)=\Pi_{A+}(G_n)\cap\Pi_u(G_n)$, where $
\Pi_{A+}(G_n)=\cup_{\psi\in\Psi^+_{1/2}(G)}\Pi_{\psi}$ (see Definition \ref{A parameters}).

First, we recall the characterization of 
$\Pi_{A+,\, u}(G_n)$. 
Let $\psi\in \Psi^{+}_{\text{unit}}(G_n)$ with decomposition
\[ \psi= \bigoplus_{i \in I_{>0}} (\rho_i \vert\cdot\rvert^{x_i} \otimes S_{a_i} \otimes S_{b_i}+\rho_i^{\vee} \vert\cdot\rvert^{-x_i} \otimes S_{a_i} \otimes S_{b_i}) + \psi_{u}, \]
where $0< x_i< \half{1} $ for any $i \in I_{>0}$ and  $\psi_{u} \in \Psi(G_m)$ for some $m \leq n$. Every representation $\pi$ in $\Pi_{\psi}$ is of the form
\[ \pi=\bigtimes_{i \in I_{>0}} u_{\rho_i}(a_i, b_i) \vert\cdot\rvert^{x_i} \rtimes \pi_u\]
for some $\pi_u \in \Pi_{\psi_u}$ by Theorem \ref{thm red from nu to gp}. 
In \cite{HJLLZ25}, we proved the following characterization of $\Pi_{A+,\, u}(G_n)$, which is also independently proved in \cite{AM25}.

\begin{thm}[{\cite[Theorem 1.3]{HJLLZ25}}]\label{thm A+,u}
In the above notation, the representation $\pi$ is unitary if and only if for any $i \in I_{>0}$ such that $u_{\rho_i}(a_i,b_i)\rtimes \pi_u$ is reducible, the set
\[ \{j \in I_{>0}\ | \ \rho_j \cong \rho_i, a_j=a_i,\ b_j=b_i\}\]
has even cardinality.
\end{thm}

Note that an algorithm was given in \cite[Theorem 4.4]{Ato22d} to determine the irreducibility of $u_{\rho_i}(a_i,b_i)\rtimes \pi_u$.

\begin{thm}\label{cor closure A+,u=A,u}
    For the groups $G_n$, we have $\Pi_{{A+,\,u}}(G_n) \subset \Pi_{\overline{A}}(G_n)$.
\end{thm}

\begin{proof}

    Let $\pi\in\Pi_{A+,\, u}.$ Then, by Theorem \ref{thm A+,u}, $\pi$ is of the following form
    \begin{align}\label{eq pi closure A= closure A+}
         \pi= \left(\bigtimes_{i \in I_{red}} u_{\rho_i}(a_i,b_i)\vert\cdot\rvert^{x_i} \times u_{\rho_i}(a_i,b_i)\vert\cdot\rvert^{y_i}\right) \times 
    \left(\bigtimes_{i \in I_{irr}}u_{\rho_i}(a_i,b_i)\vert\cdot\rvert^{z_i} \right) \rtimes \pi_{A,\, gp},
    \end{align}
    where
    \begin{itemize}
        \item $\pi_{A,\, gp}$ is an Arthur representation of good parity;
        \item For any $i \in I:= I_{red} \sqcup I_{irr}$, the parabolic induction $u_{\rho_i}(a_i,b_i) \rtimes \pi_{A,\, gp}$ is reducible if and only if $i \in I_{red}$;
        \item For any $i \in  I_{red}$, $0 < |x_i|, |y_i| < \half{1}$. For any $i \in  I_{irr}$, $0 \leq  |z_i| < \half{1}$.
    \end{itemize}
    Let $\underline{x}=(x_i)_{i \in I_{red}}, \underline{y}=(y_i)_{i \in I_{red}}, \underline{z}=(z_i)_{i \in I_{red}}$.

We set up more notation. Consider the family of parabolic induction of $\GL_d(F)$
\[ \Tau_{\underline{\alpha},\underline{\beta}, \underline{\gamma}}:= \left(\bigtimes_{i \in I_{red}} u_{\rho_i}(a_i,b_i)\vert\cdot\rvert^{\alpha_i} \times u_{\rho_i}(a_i,b_i)\vert\cdot\rvert^{\beta_i}\right) \times 
    \left(\bigtimes_{i \in I_{irr}}u_{\rho_i}(a_i,b_i)\vert\cdot\rvert^{\gamma_i} \right),\]
where $\alpha_i, \beta_i ,\gamma_i \in \R$, $\underline{\alpha}= (\alpha_i)_{i \in I_{red}}, \underline{\beta}= (\beta_i)_{i \in I_{red}}$, and $\underline{\gamma}= (\gamma_i)_{i \in I_{irr}}$.
Let  $\Pi_{\underline{\alpha},\underline{\beta}, \underline{\gamma}} := \Tau_{\underline{\alpha},\underline{\beta}, \underline{\gamma}} \rtimes \pi_{A,\, gp}$, representations of $G_n$. 
     For any tuple $\delta= (\delta_i)_{i \in I_{red}}$, we let
    $|\delta|:= (|\delta_i|)_{i \in I_{red}}$ and $-|\delta|:= (-|\delta_i|)_{i \in I_{red}}.$ 
Since $\pi = \Pi_{\underline{x},\underline{y},\underline{z}}= \Pi_{{|\underline{x}|},{-|\underline{y}|}, \underline{z}}.$
    Also, we let $\underline{0}:=(0)_{i \in I_{irr}}$ or $(0)_{i \in I_{red}}$, we may assume $x_i>0$ and $y_i<0$ without loss of generality.
    
First, the induction $\Tau_{{|\underline{x}|},{-|\underline{x}|}, \underline{0}} $ is irreducible if $0 \leq  |x_i| <\half{1}$ and $\Tau_{{\underline{0}},{\underline{0}}, \underline{0}} \otimes \pi_{A,\, gp}$ is an irreducible unitary induction of an Arthur representation from a Levi subgroup 
 of $\GL_d(F)\times G_{n-d}$ {(\cite{Tad86})}. Thus, we see that $\Tau_{{|\underline{x}|},{-|\underline{x}|}, \underline{0}} \otimes \pi_{A,\, gp}  \in  \Pi_{\overline{A}}^{(1)}(\GL_d(F) \times G_{n-d} ).$ Next, note that the connected component of
 \[ U=\{ (\underline{\alpha}, \underline{\beta}, \underline{\gamma})\in \R^{|I|} \ | \ \Tau_{\underline{\alpha},\underline{\beta}, \underline{\gamma} \rtimes \Pi_{A,\, gp}}\text{ is irreducible} \}\]
 that $(\underline{x},\underline{y}, \underline{z})$ lies in must contain the following open set 
 \[ \Omega:=\{  (\underline{\alpha}, \underline{\beta}, \underline{\gamma})\in \R^{|I|}\ | \ 0< \alpha_i <\half{1}, \ 0> \beta_i> -\half{1},\ |\gamma_i|<\half{1} \}. \]
 As $(|\underline{x}|, -|\underline{x}|, \underline{0})$ is also in $\Omega$, we see that $\pi= \Tau_{\underline{x},\underline{y}, \underline{z} \rtimes \Pi_{A,\, gp}}$ lies in $\Pi_{\overline{A}}^{(2)}(G_n)$. This completes the proof of the theorem.
\end{proof}

\subsection{Some interesting consequences}
In this subsection, we prove some implications of
Conjecture \ref{conj A+ intro} for $G_n$.

\begin{cor}\label{cor implication of 1.1}
Assume that Conjecture \ref{conj A+ intro} holds for $G_n$. 
\begin{enumerate}
\item We have an inclusion: $\Pi_{iso}(G_n) \subset \Pi_{A,\, gp}(G_n).$ 
Assume further that both Parts of Question \ref{ques crit} are affirmative. Then
$\Pi_{iso}(G_n) \subset \Pi_{A,\, crit}(G_n),$
             which is a part of Conjecture \cite[Conjecture 1.1]{Tad22}. Recall that $\Pi_{iso}(G_n)$ is the of isolated representations in the unitary dual $\Pi_u(G_n)$. 
\item The unitary dual $\Pi_{u}(G_n)$ is preserved under the Aubert-Zelevinsky involution.
\item Conjectures \ref{conj Tadic preservation}, \ref{conj independence} hold for $G_n$.
\end{enumerate}
\end{cor}
\begin{proof}
Assume that Conjecture \ref{conj A+ intro} holds for $G_n$, 
We just need to verify that these properties hold for $\Pi_{\overline{A}}^{\lim}(G_n)$.

First, we prove Part (1). Recall that  complementary series and unitary parabolic inductions result in non-isolated representations $($\cite[\S 3]{Tad93}$)$. Therefore, $\Pi_{\overline{A}}^{\lim}(G_n) \setminus \Pi_{\overline{A}}^{(0)}(G_n)$ does not contain isolated representations. 
            Moreover, if $\pi \in \Pi_{A}(G_n) \setminus \Pi_{A,\, gp}(G_n)$, then Theorem \ref{thm red from nu to gp} implies that $\pi$ is a unitary parabolic induction; hence, it is also not isolated. This establishes the first inclusion.  Assume further that both Parts of Question \ref{ques crit} are affirmative. Then any representation in $\Pi_{A,\, gp}(G_n) \setminus \Pi_{A,\, crit}(G_n)$ is not isolated. This proves the second inclusion.

For Part (2), it follows from the definition that $\Pi_{\overline{A}}^{\lim}(G_n)$ is preserved under Aubert-Zelevinsky involution if and only if so is $\Pi_{\overline{A}}(G_n)$. Since $\Pi_{u}(\GL_d(F))$ and $\Pi_{A}(G_n)$ are preserved under Aubert-Zelevinsky involution, it follows from the definition that $\Pi_{\overline{A}}^{(k)}(G_n)$ is preserved under Aubert-Zelevinsky involution if and only if the same holds for $\Pi_{\overline{A}}^{(k-1)}(G_n)$. Thus, the problem is reduced to $\Pi_{\overline{A}}^{(0)}(G_n)=\Pi_{A}(G_n)$, which follows from \cite[\S A]{Xu17b} and \cite[Lemma 4.4.4]{AGIKMS24}. 

Part (3) follows immediately from the assumption of Conjecture \ref{conj A+ intro} and Corollary \ref{cor preseveration for Pi A bar} below. This completes the proof of the corollary.
\end{proof}

\begin{remark}\label{rmk-general-classical} { For $\GL_n$, Bernstein conjectured  that $\Pi_u(\GL_n(F))$ is preserved under the Zelevinsky involution in \cite[Conjecture 8.10]{Ber83}. This was proved in \cite{Tad85b}.}
For a general reductive group $G$, we remark that Conjecture \ref{conj A+} implies that $\Pi_u(G)$ is preserved by the Aubert-Zelevinsky dual if we assume that $ \Pi_A(G)$ is stable under Aubert-Zelevinsky dual, see the proof of Corollary \ref{cor implication of 1.1}. In particular, if $G$ is quasi-split $\RO_{2n}$ or $ \mathrm{U}_n$, Conjecture \ref{conj A+} implies that $\Pi_u(G)$ is preserved under Aubert-Zelevinsky involution since the same is true for $\Pi_A$ following from  \cite[\S A]{Xu17b} and \cite[Lemma 4.4.4]{AGIKMS24}.
\end{remark}

\section{New Algorithms for Arthur Representations}\label{sec new algorithm}

In this section, we generalize \cite[Proposition 6.3]{HLLZ25} (restated in Proposition \ref{prop max triangle} below) and introduce new algorithms (Algorithms \ref{algo Arthur type} and \ref{algo Arthur type 2}) to determine whether a representation is of Arthur type, which will be used in \S
\ref{sec inductive approach} and \S \ref{algorithm for A bar}. Unlike the algorithms \cite[Algorithm 3.3]{Ato23} and \cite[Algorithm 7.9]{HLL22}, these new algorithms avoid computing highest derivatives and constructing non-tempered local Arthur packets, making them both theoretically and practically valuable.

\begin{prop}[{\cite[Proposition 6.3]{HLLZ25}}]\label{prop max triangle}
Suppose that 
\[\EE= \cup_{\rho} \{([A_i,B_i]_{\rho},l_i,\eta_i)\}_{i \in (I_{\rho},>)} \in \Rep^{(P')}\]
is absolutely maximal. Fix a $\rho$ such that $b_{\rho} := \max \{A_i-B_i+1\ | \ i \in I_{\rho}\} >1 $. Let $ j = \min\{ i \in I_{\rho}\ | \ A_i-B_i+1=b_{\rho} \}$, and assume that $j= \max\{ i \in I_{\rho} \ | \ B_i=B_j \}$ by applying row exchanges if necessary. Let $\EE^{-}:=add_j^{-1}(\EE)$, which satisfies $(P')$ by the assumptions. Then, $\pi( \EE^{-}) \neq 0$ and 
\[\pi(\EE)= \pi(\EE^-)+\{\Delta_{\rho}[B_j,-A_j]\}.\]
\end{prop}

Proposition \ref{prop max triangle} will be generalized in Proposition \ref{thm pi_rho_minus} and Proposition \ref{prop reduction along min x} below. Based on these generalizations, we give two new algorithms to determine whether a good parity representation $\pi$ is of Arthur type. The first algorithm reduces the problem from $\pi$ to $\pi^{\rho,-}$ (see Algorithm \ref{algo Arthur type} below), while the second reduces from $\pi$ to $\pi_{\rho,-}$ (see Algorithm \ref{algo Arthur type 2} below), where $\pi^{\rho,-}$ and $\pi_{\rho,-}$ are representations of smaller rank defined as follows.

\begin{defn}\label{def pi minus}
Let $\pi\in \Pi(G_n)$ be non-tempered. Write
\[ \pi= L(\Delta_{\rho_1}[x_1,-y_1], \dots , \Delta_{\rho_f}[x_f,-y_f]; \pi(\phi, \varepsilon)).\]
\begin{enumerate}
    \item [(1)] We define $\pi^{\rho,-}$ to be the representation whose $L$-data is obtained by removing all copies of $\Delta_{\rho}[x,-y]$ from $\pi$ such that
  $$ x= \min \{ x_i \ | \ x_i -y_i= \min\{ x_j-y_j \ | \  \rho_j \cong \rho \} \}, \quad y= x- \min\{ x_j-y_j \ | \  \rho_j \cong \rho \}. $$
    We shall write $\pi=\pi^{\rho,-}+\{(\Delta_{\rho}[x,-y])^r\}$, where $r$ indicates the multiplicity. 
    \item [(2)]We define $\pi_{\rho,-}$  to be the representation whose $L$-data is obtained by removing all $\Delta_{\rho_i}[x_i,-y_i]$'s from $\pi$ such that $\rho_i \cong \rho$ and $x_i= \min \{ x_j \ | \ \rho_j \cong \rho \}$.
\end{enumerate}
\end{defn}

\subsection{\texorpdfstring{Reduction along $\pi^{\rho,-}$}{}}  

In this subsection, we give an algorithm to determine whether $\pi$ is of Arthur type based on the information of $\Psi(\pi^{\rho,-})$.

First, applying Proposition \ref{prop max triangle}, we give a new algorithm to compute the $L$-data of $\pi(\EE)$. Currently, there are two algorithms to compute the $L$-data of $\pi(\EE)$, which are given in \cite[\S3.4 and  \S5]{Ato20b}. The first requires the computation of positive highest derivatives, while the second uses deformation on extended multi-segments via operators $ui_{i,j}$, $dual_k^{-}$,  and the computation of positive socles. In contrast, the new algorithm below does not need the computation of highest derivatives or socles, but needs the computation of $\EE^{|max|}$ instead. Recall that $\EE^{|max|}$ is the unique (up to row exchanges) absolutely maximal extended multi-segment such that $\pi(\EE^{|max|})= \pi(\EE)$, see Theorem \ref{thm operator}.

\begin{algo}[Algorithm to compute the $L$-data of $\pi(\EE)$ applying $(\EE^{|max|})^{-}$]\label{algo rep(E)}
Given an extended multi-segment $\EE \in \Rep$, proceed as follows.
\begin{enumerate}
     \item [\textbf{Step 1:}]Compute $\EE^{|max|}= \cup_{\rho} \{([A_i,B_i]_{\rho},l_i,\eta_i) \}_{i \in (I_{\rho},>)}$. If $\EE^{|max|}$ satisfies the condition $(L)$ (see Definition \ref{def (L)}), then the $L$-data of $\pi(\EE)$ is given by Proposition \ref{prop $L$-data L-packet}.
     
     \item [\textbf{Step 2:}]If $\EE^{|max|}$ does not satisfy the condition $($L$)$, then choose a $\rho$ such that $\EE_{\rho}$ does not satisfy the condition $(L)$. Construct $(\EE^{|max|})^{-}= add_j^{-1}(\EE^{|max|})$ as in Proposition \ref{prop max triangle}. Repeat the algorithm for $(\EE^{|max|})^{-}$, which is of smaller rank than $\EE$. Then 
     \[\pi(\EE)=\pi((\EE^{|max|})^{-})+\{ \Delta_{\rho}[B_j,-A_j] \}.\]
\end{enumerate}
\end{algo}

We remark that repeating Step 2 multiple times results in an extended multi-segment $\EE'$ such that $\psi_{\EE'}$ is tempered. Then $\EE'$ must satisfy the condition ($L$), and the procedure terminates. On the other hand, if $\EE^{|max|}$ satisfies the condition $(L)$ but $\psi_{\EE^{|max|}}$ is not tempered, we can still apply Step 2 to $\EE^{|max|}$. The resulting extended multi-segment $(\EE^{|max|})^-$ still satisfies the condition $(L)$; hence it is already absolutely maximal. Repeating this process until $\psi_{\EE^{|max|}}$ becomes tempered, we can also obtain $\pi(\EE^{|max|})$ without applying Proposition \ref{prop $L$-data L-packet}.

Here is an example for the algorithm.

\begin{exmp}\label{ex algorithm 5.1}
 Let $\rho$ be the trivial representation. Consider the following extended multi-segment of $\Sp_{10}$:
\begin{align*}
    \EE= \{([3,-3]_{\rho},3,1),([1,-1]_{\rho},1,-1),([0,0]_{\rho},0,-1)\}
    =& \scalebox{0.8}{\bordermatrix{
    &-3&-2&-1&0&1&2&3 \cr
    &\lhd &\lhd & \lhd & \oplus& \rhd & \rhd &\rhd \cr
    & &&\lhd&\ominus &\rhd && \cr
        & &&&\ominus &&& \cr}}_{\rho}
        =\EE^{|max|}.
\end{align*}
By \cite[Example 10.14.6]{HLL22}, we have that 
$\pi=L(\Delta_{\rho}[-3,-3],\Delta_{\rho}[-1,-2], \Delta_{\rho}[0,-1]; \pi(0^{+})).$
Now we apply Algorithm \ref{algo rep(E)}.
\begin{align*}
    \EE_1&= add_{1}^{-1}(\EE)=\scalebox{0.8}{\bordermatrix{
    &-2&-1&0&1&2 \cr
     &\lhd & \lhd & \oplus& \rhd & \rhd  \cr
     &&\lhd&\ominus &\rhd & \cr
        &&&\ominus && \cr}}_{\rho}, \quad \EE_1^{|max|}=\scalebox{0.8}{\bordermatrix{
    &-1&0&1&2 \cr
      & \lhd & \oplus& \ominus & \rhd  \cr
     &&\ominus & &  \cr}}_{\rho};\\
         \EE_2&= add_{1}^{-1}(\EE_1^{|max|})=\scalebox{0.8}{\bordermatrix{
    &0&1 \cr
    & \oplus& \ominus   \cr
     &\ominus &   \cr}}_{\rho}= \EE_2^{|max|}, \quad \EE_2'=\scalebox{0.8}{\bordermatrix{
    &0&1 \cr
    & \oplus&    \cr
     &\lhd &\rhd   \cr}}_{\rho};\\
     \EE_3&=add_2^{-1}(\EE_2')= \scalebox{0.8}{\bordermatrix{ & 0 \cr & \oplus\cr}}_{\rho}.
\end{align*}
Note that we perform a row exchange from $\EE_2$ to $\EE_2'$ to satisfy the condition in the statement of Proposition \ref{prop max triangle}. Then we have
\begin{align*}
    \pi(\EE_3)&= \pi(0^+),\\
    \pi(\EE_2)&=L(\Delta_{\rho}[0,-1]; \pi(0^+)),\\
    \pi(\EE_1)&= L(\Delta_{\rho}[-1,-2], \Delta_{\rho}[0,-1];\pi(0^{+})),\\
    \pi(\EE)&=L(\Delta_{\rho}[-3,-3],\Delta_{\rho}[-1,-2], \Delta_{\rho}[0,-1]; \pi(0^{+})).
\end{align*}
In this example, we see that $\psi_{\EE_1} \in \Psi(\pi(\EE_1))$ does not satisfy the conclusion of Proposition \ref{thm pi_rho_minus}(c) below since $ \psi_{\EE_1} \neq \psi^{max}(\pi(\EE_1))$. Also, $\pi(add_1^{-1}(\EE_1))$ vanishes, which demonstrates that Proposition \ref{prop max triangle} does not hold if we drop the absolute maximality condition.
\end{exmp}

In the example above, observe that $\pi=\pi(\EE)$ is of Arthur type since $\pi^{\rho,-}= \pi(\EE_1)$ is of Arthur type. Moreover, there exists an $\EE_1$ in the set $\{\EE'\ |\  \pi(\EE')=\pi^{\rho,-} \}$ which contains an extended segment of the form $([2,-2]_{\rho},l,\eta)$, so that we can perform $add^1$ to deform it to $([3,-3]_{\rho},l+1,\eta)$ to recover $\Delta_{\rho}[-3,-3]$ in the $L$-data of $\pi$. That is, we can recover $\EE$ from the pair $(\EE_1,\Delta_{\rho}[-3,-3])$. This motivates the new algorithm (Algorithm \ref{algo Arthur type}) to determine whether $\pi$ is of Arthur type. We first give some useful notation.

\begin{defn}\label{def EE_rho_minus}
Suppose that $\EE= \cup_{\rho} \{([A_i,B_i]_{\rho},l_i,\eta_i)\}_{i \in (I_{\rho},>)} \in \Rep^{(P')}$ is absolutely maximal. If necessary, change the admissible order on each $I_{\rho}$ so that $A_i \geq A_j$ if $i>j$ and $B_i=B_j$. For each $\rho$ such that $I_{\rho}\neq \emptyset$ and $b_{\rho} := \max \{A_i-B_i+1\ | \ i \in I_{\rho}\} >1 $, let $ j_1 = \min\{ i \in I_{\rho}\ | \ A_i-B_i+1=b_{\rho} \}$ and denote 
\[\{j \in I_{\rho} \ |\  [A_j,B_j]_{\rho}= [A_{j_1},B_{j_1}]_{\rho} \}= \{j_1, \dots ,j_s\},\]
with $j_1 < \cdots <j_s$. Then define 
$\EE^{\rho,-}:=(\sum_{k=1}^s add_{j_k}^{-1})(\EE)$, 
which satisfies $(P')$ by the assumptions.
\end{defn}

The following is the first generalization of Proposition \ref{prop max triangle}.

\begin{prop}\label{thm pi_rho_minus}
 Let $\pi \in \Pi_{A, \, gp}(G_n)$ be non-tempered and assume that $\pi=\pi^{\rho,-}+\{(\Delta_{\rho}[x,-y])^r\}$. Let $\EE$ be an absolutely maximal extended multi-segment with $\pi(\EE)=\pi$. Then the following hold.
 \begin{enumerate}
     \item [(a)] The representation corresponding to $\EE^{\rho,-}$ is non-vanishing:  $\pi(\EE^{\rho,-})\neq 0$. 
     \item [(b)] Moreover, $\pi(\EE^{\rho,-})= \pi^{\rho,-}$. In particular, $\pi^{\rho,-}\in \Pi_{\psi_{\EE^{\rho,-}}}$ is of Arthur type, where
     \[ \psi_{\EE^{\rho,-}}= \psi^{max}(\pi)-(\rho\otimes S_{y+x+1} \otimes S_{y-x+1})^{\oplus r}+(\rho\otimes S_{y+x+1} \otimes S_{y-x-1})^{\oplus r}. \]
     \item [(c)] The parameter $\psi^{max}(\pi)$ contains exactly $r$ copies of $ \rho\otimes S_{y+x+1} \otimes S_{y-x+1}$.
 \end{enumerate}
\end{prop}

\begin{proof}
We adopt the notation in Definition \ref{def EE_rho_minus}. For $1 \leq i \leq s$, let $\EE^i:=(\sum_{k=i}^s add_{j_i}^{-1})(\EE)$. Also set $\EE^{s+1}=\EE$. Note that $B_{j_1}=x$, $A_{j_1}=y$ by construction.

Part (a) follows the same proof of  {\cite[Lemma 3.28]{HLLZ25}}, which we omit here. For Part (b),  Proposition \ref{prop max triangle} implies that for $1 \leq i \leq s,$ we have
\[ \pi(\EE^{i+1})=\pi(\EE^i)+\{\Delta_{\rho}[x,-y]\},\quad {\rm and}\quad 
\pi(\EE)=\pi(\EE^{\rho,-})+\{ (\Delta_{\rho}[x,-y])^s\}. \]
 Therefore, to show $\pi(\EE^{\rho,-})=\pi^{\rho,-}$,
it is equivalent to show that $r=s$. Note that this also implies the rest of Parts (b) and (c). Observe that for any extended multi-segment $\EE'$ obtained from $\EE^{\rho,-}$ by a sequence of raising operators or $add^{-1}$, the parameter $\psi_{\EE'}$ does not contain any copy of $\rho \otimes S_{y+x+1}\otimes S_{y-x+1}$ according to the definitions of these operators. Therefore, Algorithm \ref{algo rep(E)} implies that the $L$-data of $\pi(\EE^{\rho,-})$ does not contain $\Delta_{\rho}[x,-y]$, so we have $r=s$. This completes the proof.
\end{proof}

\begin{remark}\ 
    \begin{enumerate}
        \item Part (c) of Proposition \ref{thm pi_rho_minus} may fail for an arbitrary $\psi \in \{\psi \ | \ \pi \in \Pi_{\psi}\} \setminus \{\psi^{max}(\pi)\}$ (see Example \ref{ex algorithm 5.1}). This indicates the importance of the distinguished member $\psi^{max}(\pi)$.
        \item With Proposition \ref{thm pi_rho_minus} at hand, one can rewrite Algorithm \ref{algo rep(E)} by replacing $(\EE^{|max|})^{-}$ with $(\EE^{|max|})^{\rho, -}$, for the chosen $\rho$, to improve its efficiency.
    \end{enumerate}
\end{remark}

Continue with the notation of Proposition \ref{thm pi_rho_minus}. We would like to recover the extended multi-segment $\EE$ from the data $(\EE^{\rho,-}, \Delta_{\rho}[x,-y], r)$. To this end, we introduce the following notation.

\begin{defn}\label{def criterion for Arthur type}
Let $\pi\in\Pi_{gp}(G_n)$, $x,y\in\R$ such that $x-y\in\mathbb{Z}_{< 0}$, $\rho\in\mathcal{C}^{sd}$,  and $r\in\mathbb{Z}_{>0}.$
\begin{enumerate}
    \item [(1)] We denote by $\Psi(\pi; \Delta_{\rho}[x,-y],r)$ the set of local Arthur parameters $\psi$ such that
\begin{enumerate}
    \item [$\oldbullet$] $\pi \in \Pi_{\psi}$;
    \item [$\oldbullet$] if $y-x-1>0$, then $\psi$ contains at least $r$ copies of $\rho\otimes S_{x+y+1}\otimes S_{y-x-1}$; 
    \item [$\oldbullet$] any summand of $\psi$ of the form $\rho \otimes S_{a} \otimes S_b$ satisfies $b \leq y-x+1$, and $a>x+y+1$ if $b=y-x+1$.
\end{enumerate}
For any $\psi \in \Psi(\pi; \Delta_{\rho}[x,-y],r)$, we define
\[\psi^{+}:= \psi- (\rho\otimes S_{x+y+1}\otimes S_{y-x-1})^{\oplus r}+(\rho\otimes S_{x+y+1}\otimes S_{y-x+1})^{\oplus r}.\]

\item [(2)] We denote by $ \mathscr{E}(\pi; \Delta_{\rho}[x,-y],r)$ the set of extended multi-segments $\EE \in \Rep^{(P')}$ such that $\pi(\EE)=\pi$ and $\psi_{\EE}\in \Psi(\pi; \Delta_{\rho}[x,-y],r)$. For each $\EE \in \mathscr{E}(\pi; \Delta_{\rho}[x,-y],r)$,
we define $\EE^{\rho,+}$ as follows.
\begin{enumerate}
    \item [(i)] If $y-x=1$, then we define $\EE^{\rho,+}$ by {inserting $r$ copies of $([x+1,x]_{\rho},1,1)$} in $\EE$ with admissible order $\gg$ on  $I_{\rho} \sqcup \{j_1,\dots ,j_r\}$, where $j_{k}$ corresponds to the $k$-th copy we inserted, as follows
\[\begin{cases}
j_r \gg j_{r-1} \gg \cdots \gg j_1,\\
\alpha \gg \beta  \Longleftrightarrow \alpha > \beta & \text{for }\alpha, \beta \in I_{\rho},\\
\alpha \gg j_k \Longleftrightarrow B_\alpha >x-1 & \text{for }\alpha \in I_{\rho} \ \mathrm{and} \ 1\leq k \leq r.
\end{cases}\]
\item [(ii)] If $y-x>1$, then change the admissible order if necessary so that there exists $j_1,\ldots, j_r \in I_{\rho}$ such that
\begin{enumerate}
    \item [$\oldbullet$] $j_1 < \cdots < j_r$ are adjacent under the admissible order on $I_{\rho}$,
    \item [$\oldbullet$] $[A_{j_1},B_{j_1}]_{\rho}= \cdots = [A_{j_r},B_{j_r}]_{\rho}=[y-1,x+1]_{\rho}$,
    \item [$\oldbullet$]  $j_1 = \min\{i \in I_{\rho} \ | \ B_i=B_{j_1}\}$,
    \item [$\oldbullet$] $ A_{j_1}-B_{j_1}+3 \geq A_i-B_i+1$ for all $i \in I_{\rho}$ and the equality does not hold for $i<j_1$.
\end{enumerate}
Then we define $\EE^{\rho,+}:=(\sum_{k=1}^r add_{j_k}^{1})(\EE) \in \Rep^{(P')}$.
\end{enumerate}
\end{enumerate}
\end{defn}

The following lemma presents the relation between the operations $\EE^{\rho,+}$ and $\pi^{\rho,-}$.

\begin{lemma}\label{lemma pi^rho^plus}
    Let $\pi\in\Pi_{gp}(G_n)$, $x,y\in\R$ such that $x-y\in\mathbb{Z}_{< 0}$, $r\in\mathbb{Z}_{\geq 1},$ and $\EE  \in \mathscr{E}( \pi;\Delta_{\rho}[x,-y],r )$. Then $\pi(\EE^{\rho,+})\neq 0$ and 
    \[ \pi(\EE^{\rho,+})= \pi(\EE^{\rho,+})^{\rho,-}+ \{ (\Delta_{\rho}[x,-y])^r\}. \]
\end{lemma}
\begin{proof}
    For the moment, assume that we have verified $\pi(\EE^{\rho,+})\neq 0$. {Then any raising operator applicable on $\EE^{\rho,+}$ is also applicable on $\EE$ by Lemma \ref{lem add^1 raising}, where the Condition (2) in that lemma follows from Proposition \ref{prop max triangle}}. Hence, we may assume that $\EE^{\rho,+}$ is absolutely maximal. By definition, we have $(\EE^{\rho,+})^{\rho,-}= \EE$ (see Definition \ref{def EE_rho_minus}) up to row exchanges; hence, Proposition \ref{thm pi_rho_minus} implies the rest of the conclusion. Therefore, it remains to verify $\pi(\EE^{\rho,+})\neq 0$.

    First, assume that $y-x=1$. Note that any raising operator applicable on $\EE$ preserves the condition in Definition \ref{def criterion for Arthur type}(1); hence, we can replace $\EE$ by $\EE^{|max|}$ if necessary.  {Note that $\pi(\EE^{\rho,+})\neq 0$ is equivalent to $\pi((\EE^{|max|})^{\rho,+})\neq 0$ since $\EE^{\rho,+}$ can be obtained from $(\EE^{|max|})^{\rho,+}$ by a sequence of inverse of raising operators.}  Since we have $b \leq y-x+1=2$ for any summand $\rho \otimes S_a \otimes S_b$ of $\psi_{\EE}$, the extended multi-segment $\EE^{\rho,+}$ automatically satisfies the condition $(L)$ (see Definition \ref{def (L)}). Consequently, $\pi(\EE^{\rho,+})\neq 0$ by Proposition \ref{prop $L$-data L-packet}.

    Next, assume that $y-x \geq 1$. Note that $\pi(\EE^{\rho,+}) \neq 0$ if and only if $\pi(dual(\EE^{\rho,+}))\neq 0$. By Lemma \ref{lem equalities of operators}(2), we have (using the notation in Definition \ref{def criterion for Arthur type})
 $$dual(\EE^{\rho,+})= \sum_{k=1}^r sh_{j_k}^{1}(dual(\EE)).$$
 Write $dual(\EE)= \cup_{\rho}\{([A_{i},B_i]_{\rho'},l_i,\eta_i)\}_{ i \in (I_{\rho},>')}$. With the definition of $dual$, the assumptions in Definition \ref{def criterion for Arthur type}(2)(ii) translates into the following conditions:
 \begin{enumerate}
    \item [(a)] $j_1 >'\cdots >' j_r$ are adjacent under the admissible order $>'$ on $I_{\rho}$,
    \item [(b)] $[A_{j_1},B_{j_1}]_{\rho}= \cdots = [A_{j_r},B_{j_r}]_{\rho}=[y-1,-x-1]_{\rho}$,
    \item [(c)]  $j_1 = \max\{i \in (I_{\rho},>') \ | \ B_i=B_{j_1}\}$,
    \item [(d)] $ A_{j_1}+B_{j_1}+3 \geq A_i+B_i+1$ for all $i \in I_{\rho}$ and the equality does not hold for $i>'j_1$.
\end{enumerate}
Note that conditions (c) and (d) imply that  $B_{i}>B_{j_1}$ and $A_{j_1}\geq A_{i}$, for any $i>' j_1$. Thus if $r=1$, the desired conclusion $\pi(\EE^{\rho,+}) \neq 0$ is a consequence of \cite[Lemma 4.3(iii)]{HLL22}. We prove the general case by applying induction on $r$. Write $A:=A_{j_1} $ and $B:=B_{j_1}$ in the rest of the proof.

Let $I_{\rho,1}:= \{i \in I_{\rho},\  j_r>' i\}$, $I_{\rho,2}:= \{i \in I_{\rho},\ i>' j_1\}$, and $\widetilde{\EE_{\rho,k}}:= \{([A_{i},B_i]_{\rho},l_i,\eta_i)\}_{ i \in (I_{\rho,k},>')}$ for $k=1,2$. We have 
\[ dual(\EE_{\rho})= \widetilde{\EE_{\rho,1}}+ \{([A,B]_{\rho},\ast, \ast)\}+\{([A,B]_{\rho},\ast, \ast)^{r-1}\}+ \widetilde{\EE_{\rho,2}}, \]
where $r-1$ denotes the multiplicity. The conditions (c) and (d) imply that we may apply row exchanges on $\EE_{\rho}$ to get 
\[dual(\EE_{\rho})'= \widetilde{\EE_{\rho,1}}+\{([A,B]_{\rho},\ast, \ast)\}+\widetilde{\EE_{\rho,2}}'+\{([A,B]_{\rho},\ast, \ast)^{r-1}\}. \]
The case for $r=1$ implies that 
$\widetilde{\EE_{\rho,1}}+\{([A+1,B+1]_{\rho},\ast, \ast)\}+\widetilde{\EE_{\rho,2}}'$
satisfies the non-vanishing criterion in Theorem \ref{thm non-vanishing}. The same is true for
\[ \widetilde{\EE_{\rho,1}}+\{([A,B]_{\rho},\ast, \ast)\}+\ \widetilde{\EE_{\rho,2}}'+\{([A+1,B+1]_{\rho},\ast, \ast)^{r-1}\} \]
by the induction hypothesis. Then, it follows that 
\[dual(\EE_{\rho})^+:= \widetilde{\EE_{\rho,1}}+\{([A+1,B+1]_{\rho},\ast, \ast)\}+\widetilde{\EE_{\rho,2}}'+\{([A+1,B+1]_{\rho},\ast, \ast)^{r-1}\}  \]
also satisfies the non-vanishing criterion. Applying row exchanges to $dual(\EE_{\rho})^+$ recovers $dual(\EE^{\rho,+})$. By Theorem \ref{thm dual}, we conclude that $\pi(\EE^{\rho,+})\neq 0$. This completes the proof of the lemma.
\end{proof}

The following theorem gives a criterion to determine whether a representation $\pi \in \Pi(G_n)$ is of Arthur type using the information of $\Psi(\pi^{\rho,-})$, where $\pi^{\rho,-} \in \Pi(G_m)$ for some $m<n$. 

\begin{thm}\label{thm algo for Arthur type}
Let $\pi\in\Pi_{gp}(G_n)$ and $\pi=\pi^{\rho,-}+\{(\Delta_{\rho}[x,-y])^r\}$. Then $\pi$ is of Arthur type if and only if $\Psi(\pi^{\rho,-}; \Delta_{\rho}[x,-y],r)$ is non-empty. Moreover, for any $\EE \in  \mathscr{E}(\pi^{\rho,-}; \Delta_{\rho}[x,-y],r)$, we have $\pi=\pi(\EE^{\rho,+})$. In other words, for any $ \psi \in \Psi(\pi^{\rho,-}; \Delta_{\rho}[x,-y],r)$, we have $ \psi^{+} \in \Psi(\pi)$.
\end{thm}

\begin{proof}
If $\pi$ is of Arthur type, then $\pi^{\rho,-}\in \Pi_{\psi^{-}}$ by Proposition \ref{thm pi_rho_minus}, where $\psi^{-} \in\Psi(\pi^{\rho,-}; \Delta_{\rho}[x,-y],r)$. In particular, $\Psi(\pi^{\rho,-}; \Delta_{\rho}[x,-y],r)$ is non-empty. 

Conversely, assume that $\Psi(\pi^{\rho,-}; \Delta_{\rho}[x,-y],r)$ is non-empty and take  $\EE\in \mathscr{E}(\pi^{\rho,-}; \Delta_{\rho}[x,-y],r)$.
Then Lemma \ref{lemma pi^rho^plus} implies that $\pi(\EE^{\rho,+})=\pi$; hence, $\pi \in \Pi_{\psi_{\EE^{\rho,+}}}$. Note that $(\psi_{\EE})^{+}= \psi_{\EE^{\rho,+}}$. This completes the proof of the theorem.
\end{proof}

We remark that the set $\{\psi^+ \ | \ \psi \in \Psi(\pi^{\rho,-}; \Delta_{\rho}[x,-y];r)\}$ is often a proper subset of $\Psi(\pi)$.
Based on the above theorem, we give a new algorithm to determine whether a good parity representation $\pi$ is of Arthur type. If $\pi$ is of Arthur type, then it outputs the set
$$\mathscr{E}(\pi):=\{\EE \ | \ \pi(\EE)= \pi\}/ \text{(row exchanges)}.$$

\begin{algo}[Algorithm to determine Arthur type using $\pi^{\rho,-}$]\label{algo Arthur type}
Given a representation $\pi$ of $G_n$ of good parity, proceed as follows:

\begin{enumerate}
    \item [\textbf{Step 1:}] If $\pi$ is tempered, then $\pi$ is of Arthur type. According to Theorem \ref{thm Arthur tempered}, write $\pi=\pi(\phi, \varepsilon)$, where
    \[\phi= \bigoplus_{\rho}\bigoplus_{i\in I_{\rho}} \rho \otimes S_{2a_i+1}.\]
    We equip $I_{\rho}$ with a total order $>$ such that $a_i$ is non-decreasing. Then let 
    \[ \EE:= \cup_{\rho} \{ ([a_i,a_i],0, \varepsilon( \rho \otimes S_{2a_i+1})) \}_{i \in (I_{\rho, >})},\]
    and output $\mathscr{E}(\EE):=\{ \EE' \ | \ \pi(\EE')\cong \pi(\EE) \}/\text{(row exchanges)}$ by \cite[Theorem 7.4]{HLL22}. We regard the trivial representation of $G_0$ as a tempered representation, and output $\{\EE\}$, where $\EE=\emptyset$.
    \item [\textbf{Step 2:}] If $\pi$ is not tempered, say $\pi^{\rho,-}$ is obtained from $\pi$ by removing $r$ copies of $\Delta_{\rho}[x,-y]$ $($see Definition $\ref{def pi minus}(1))$. Then apply the algorithm on $\pi^{\rho,-}$ to construct the set (possibly empty)
\[ \{\EE' \ | \ \pi(\EE') \cong \pi^{\rho,-}\}.\]
The representation $\pi$ is of Arthur type if and only if the set $\mathscr{E}(\pi^{\rho,-}; \Delta_{\rho}[x,-y],r) $ is non-empty. If this set is non-empty, take any $\EE$ in this set and then output $\mathscr{E}(\EE^{\rho,+})$ by \cite[Theorem 7.4]{HLL22}.
\end{enumerate}
\end{algo}

\begin{remark}
Since repeated application of $\pi^{\rho,-}$ eventually yields a tempered representation, the algorithm must terminate. Unlike the algorithms in \cite[Algorithm 3.3]{Ato23} and \cite[Algorithm 7.9]{HLL22}, the new algorithm avoids constructing potentially non-tempered local Arthur packets and does not rely on computing highest derivatives.
\end{remark}

As an example, we now apply the new algorithm to \cite[Example 7.11]{HLL22}. Rather than computing $\Psi(\pi^{\rho,-}; \Delta_{\rho}[x,-y], r)$ in full detail, we simply list an element $\psi$ from this set whenever it is nonempty.  

\begin{exmp}\label{ex arthur type}\ Let $\rho$ be the trivial representation.

1. Let 
    $\pi= L\left(\Delta_{\rho}\left[ \half{-1},\half{-5}\right], \Delta_{\rho}\left[ \half{-1},\half{-1}\right],\Delta_{\rho}\left[ \half{3},\half{-5}\right]; \pi\left(\half{1}^+,\half{3}^+, \half{5}^+\right)\right)$
    be a representation of $SO_{31}$
    and 
    \begin{align*}
        \pi_1&:=L\left( \Delta_{\rho}\left[ \half{-1},\half{-1}\right],\Delta_{\rho}\left[ \half{3},\half{-5}\right]; \pi\left(\half{1}^+,\half{3}^+, \half{5}^+\right)\right),\\
        \pi_2&:=L\left( \Delta_{\rho}\left[ \half{3},\half{-5}\right]; \pi\left(\half{1}^+,\half{3}^+, \half{5}^+\right)\right),\\
        \pi_3&:= \pi\left(\half{1}^+,\half{3}^+, \half{5}^+\right).
    \end{align*}
    First, note that $\pi_3$ is tempered; hence, it is of Arthur type. We have  $\pi_3=\pi(\EE_3)$, $\psi_{\EE_3} \in \Psi(\pi_3; \Delta_{\rho}[\half{3},\half{-5}],1)$, and $\pi_2=\pi(\EE_3^{\rho,+})$, where 
    \[ { \EE_3=  \scalebox{0.8}{\bordermatrix{ &\half{1}&\half{3}&\half{5} \cr 
    &\oplus && \cr 
    &&\oplus& \cr
    &&&\oplus \cr}}_{\rho},\ \EE_3^{\rho,+}= \scalebox{0.8}{\bordermatrix{ &\half{1}&\half{3}&\half{5} \cr
    &\oplus && \cr 
    &&\oplus& \cr
    &&{\color{red}\lhd}&{\color{red}\rhd} \cr
    &&&\oplus \cr}}_{\rho}.}\]
Therefore, $\pi_2$ is of Arthur type. Similarly, we have  $\pi_2= \pi(\EE_2)$, $\psi_{\EE_2} \in \Psi\left(\pi_2; \Delta_{\rho}\left[\half{-1},\half{-1}\right],1\right),$
and $\pi_1=\pi(\EE_2^{\rho,+})$, where 
    \[{ \EE_2= \scalebox{0.8}{\bordermatrix{ &\half{1}&\half{3}&\half{5} \cr
    &\oplus && \cr 
    &&\oplus& \cr
    &&\lhd&\rhd \cr
    &&&\oplus \cr}}_{\rho},\ \EE_2^{\rho,+}= \scalebox{0.8}{\bordermatrix{ &\half{-1}&\half{1}&\half{3}&\half{5} \cr
    &{\color{red}\lhd} & {\color{red}\rhd} & &\cr
    &&\oplus && \cr 
    &&&\oplus& \cr
    &&&\lhd&\rhd \cr
    &&&&\oplus \cr}}_{\rho}.}\]
Thus $\pi_1$ is of Arthur type. Finally, applying \cite[Theorem 7.4]{HLL22}, we obtain that the set $\{\EE \ | \ \pi(\EE)=\pi(\EE_2^{\rho,+})=\pi_1\}/\text{(row exchanges)}$ consists of four extended multi-segments
\[\scalebox{0.8}{\bordermatrix{ &\half{-1}&\half{1}&\half{3}&\half{5} \cr
    &\lhd & \rhd & &\cr
    &&\oplus && \cr 
    &&&\oplus& \cr
    &&&\lhd&\rhd \cr
    &&&&\oplus \cr}}_{\rho},\ \scalebox{0.8}{\bordermatrix{ &\half{-1}&\half{1}&\half{3}&\half{5} \cr
    &\lhd & \rhd & &\cr
    &&\oplus && \cr 
    &&&\oplus& \ominus\cr
    &&&\ominus&\oplus \cr}}_{\rho},\ \scalebox{0.8}{\bordermatrix{ &\half{-1}&\half{1}&\half{3}&\half{5} \cr
    &\oplus & \ominus & &\cr
    &\ominus&\oplus && \cr 
    &&&\oplus& \cr
    &&&\lhd&\rhd \cr
    &&&&\oplus \cr}}_{\rho},\ \scalebox{0.8}{\bordermatrix{ &\half{-1}&\half{1}&\half{3}&\half{5} \cr
    &\oplus & \ominus & &\cr
    &\ominus&\oplus && \cr 
    &&&\oplus& \ominus\cr
    &&&\ominus&\oplus \cr}}_{\rho}.  \]
    Thus, $\Psi(\pi_1; \Delta_{\rho}[\half{-1},\half{-5}],1)$ is empty by definition, and we conclude that $\pi$ is not of Arthur type.

\quad

2. Consider 
    $\pi= L\left(\Delta_{\rho}\left[ \half{-1},\half{-5}\right], \Delta_{\rho}\left[ \half{-1},\half{-1}\right],\Delta_{\rho}\left[ \half{3},\half{-5}\right]; \pi\left(\half{1}^-,\half{3}^+, \half{5}^-\right)\right)$
    and let 
    \begin{align*}
        \pi_1&:=L\left( \Delta_{\rho}\left[ \half{-1},\half{-1}\right],\Delta_{\rho}\left[ \half{3},\half{-5}\right]; \pi\left(\half{1}^-,\half{3}^+, \half{5}^-\right)\right),\\
        \pi_2&:=L\left( \Delta_{\rho}\left[ \half{3},\half{-5}\right]; \pi\left(\half{1}^-,\half{3}^+, \half{5}^-\right)\right),\\
        \pi_3&:= \pi\left(\half{1}^-,\half{3}^+, \half{5}^-\right).
    \end{align*}
    Following the same computation as in the previous example, we see that $\pi_1$ is of Arthur type and $\pi_1=\pi(\EE_2^{\rho,+})$, where 
    \[ \EE_2^{\rho,+}=\scalebox{0.8}{ \bordermatrix{ &\half{-1}&\half{1}&\half{3}&\half{5} \cr
    &\lhd & \rhd & &\cr
    &&\ominus && \cr 
    &&&\oplus& \cr
    &&&\lhd&\rhd \cr
    &&&&\ominus \cr}}_{\rho}. \]
Then we obtain that $\psi_{\EE_1} \in \Psi(\pi_1; \Delta_{\rho}[\half{-1},\half{-5}],1)$ and $\pi=\pi(\EE_1^{\rho,+})$, where
\[ \EE_1= \scalebox{0.8}{\bordermatrix{ &\half{-1}&\half{1}&\half{3}&\half{5} \cr
    &\lhd & \rhd & &\cr
    &&\ominus &\oplus& \cr 
    &&&\lhd&\rhd \cr
    &&&&\ominus \cr}}_{\rho},\ \EE_1^{\rho,+}= \scalebox{0.8}{\bordermatrix{ &\half{-1}&\half{1}&\half{3}&\half{5} \cr
    &\lhd & \rhd & &\cr
    &{\color{red}\lhd}&\ominus &\oplus&{\color{red}\rhd} \cr 
    &&&\lhd&\rhd \cr
    &&&&\ominus \cr}}_{\rho}.  \]
    Therefore, $\pi$ is of Arthur type.
\end{exmp}

\subsection{\texorpdfstring{Reduction along $\pi_{\rho,-}$}{}}

 In this subsection, we give another algorithm to determine whether $\pi$ is of Arthur type, from the information of $\Psi(\pi_{\rho,-})$. This is inspired by a related question from Atobe for which we thank him, and we include it here for future reference.

\begin{defn} \label{def E lower minus}
Assume that $\EE= \cup_{\rho} \{ ([A_i,B_i],l_i, \eta_i)\}_{i \in (I_{\rho},>)}\in \Rep^{(P')}$. If there exists an $i \in I_{\rho}$ such that $A_i \neq B_i$, then after row exchanges if necessary, there exists a unique $B \in \half{1}\Z$ and a decomposition $I_{\rho}= I_{\rho,1} \sqcup I_{\rho,2} \sqcup I_{\rho,3}$ such that for any $i_1 \in I_{\rho,1}, i_2 \in I_{\rho,2}, i_3 \in I_{\rho,3}$, the following hold:
 \begin{enumerate}
     \item [$\oldbullet$] $A_{i_1}=B_{i_1}\leq B$,
     \item [$\oldbullet$] $ B=B_{i_2}< A_{i_2}$,
     \item [$\oldbullet$] $B<B_{i_3}$.
 \end{enumerate}
 If $l_{i_2} \geq 1$ for all $i_2 \in I_{\rho,2}$, then we define $\EE_{\rho,-}:=\sum_{i_2 \in I_{\rho,2}} add_{i_2}^{-1} (\EE).$
\end{defn}

We first show that if $\pi(\EE)$ is non-tempered, then the above conditions must hold (after row exchange if necessary) for $\EE^{|max|}.$ Moreover,  $\pi((\EE^{|max|})_{\rho,-})$ is a nonzero representation.

\begin{lemma}\label{len E max lower minus}
    Suppose that $\EE \in \Rep^{(P')}$ is absolutely maximal and $\pi(\EE)$ is non-tempered. Then the conditions in Definition \ref{def E lower minus} are satisfied and $\pi(\EE_{\rho,-})\neq 0$. 
\end{lemma}

\begin{proof}
Keep the notation in Definition \ref{def E lower minus} and write $(I_{\rho,2},>)=\{1<\cdots<k\}$. Apply row exchanges if necessary such that we have $A_1\leq \cdots\leq  A_k$. Then we must have that $B_i+l_i \geq -\half{1}$  and  $ l_1\leq \cdots \leq l_{k}$ by Part (i)(b) and Part (ii)(3) of Theorem \ref{thm non-vanishing}.
We claim that the absolute maximality of $\EE$ implies that $l_1 \geq 1$. Indeed, suppose the contrary that $l_1=0$. If $B_1=-\half{1}$, then one can check that the raising operator $dual_1^{-}$ is applicable, which contradicts the absolute maximality of $\EE$. If $B_1 \geq \half{1}$, then we can apply $ui^{-1}$ to break $([A_1,B_1]_{\rho},0, \ast)$ into 
\[ \{([B_1,B_1]_{\rho},0, \ast),([A_1,B_1+1]_{\rho},0, \ast) \}. \] Note that the resulting extended multi-segment still gives a nonzero representation by  \cite[Corollary 5.8]{HLL22}.
Again, this contradicts to the absolute maximality. This verifies the claim, which guarantees that the conditions of Definition \ref{def EE_rho_minus} hold.

Next, we check that $\pi( \EE_{\rho,-}) \neq 0$. According to Theorem \ref{thm non-vanishing}(i), we may assume that $I_{\rho,1}$ is empty and $B_i\leq -\half{1}$. Note that $\pi(\EE_{\rho,-})\neq 0$ is equivalent to $\pi(dual(\EE_{\rho,-}))\neq 0$.  Write $dual(\EE_{\rho})=\FF_{3}+\FF_2,$ where
$\FF_3= \{ ([A_i,-B_i]_{\rho},l_i',\eta_i') \}_{i \in (I_{\rho,3},>')}$ and $\FF_2= \{ ([A_i,-B_i]_{\rho},l_i',\eta_i') \}_{i \in (I_{\rho,2},>')}.$ 
 By Lemma \ref{lem equalities of operators}(2), it is enough to show that $\pi(dual(\EE)^{\rho} \cup (\FF_3+ sh^{-1}(\FF_2))) \neq 0$. Indeed, this follows from \cite[Proposition 6.13(i)]{HLL22}, since the absolute maximality of $ \EE$ implies that there is no inverse of a raising operator applicable on $\FF_2+\FF_3$ (\cite[Lemma 10.16]{HLL22}). This completes the proof of the lemma.
 \end{proof}

The following proposition, which is a generalization of Proposition \ref{prop max triangle}, 
relates the $L$-data of $\pi(\EE)$ with $\pi(\EE_{\rho,-}),$ provided that the latter is nonzero. 

 \begin{prop}\label{prop reduction along min x} 
If $\pi(\EE_{\rho,-}) \neq 0$, then 
$\pi(\EE)= \pi(\EE_{\rho,-})+\{\Delta_{\rho}[B_{i_2},-A_{i_2}]\}_{i_2 \in I_{\rho,2}}.$ 
Moreover, $\pi(\EE_{\rho,-})= \pi(\EE)_{\rho,-}$ (see Definition \ref{def pi minus}(2)).
 \end{prop}

\begin{proof}
 Write $\EE_{\rho}=\EE_{\rho,1}+\EE_{\rho,2}+\EE_{\rho,3}$ where
$\EE_{\rho,j}:= \{ ([A_{i},B_{i}],l_{i},\eta_{i}) \}_{i \in (I_{\rho,j},>)},$
for $j=1,2,3$. To describe the relation between $\pi(\EE)$ and $ \pi(\EE_{\rho,-})$, we apply induction on  $r:=|I_{\rho,2}|$. When $r=0$, there is nothing to prove.

When $r \geq 1$, we apply induction on $n$, the rank of the group. {By Lemma \ref{lem add^1 raising}, we have
$\EE_{\rho}^{|max|}= \EE_{\rho,1}+\EE_{\rho,2}+ (\EE_{\rho,3})^{|max|}$}.
Hence, we may assume that $\EE_{\rho,3}=(\EE_{\rho,3})^{|max|}$. Let $j$ be the index chosen in Proposition \ref{prop max triangle}. If $j \in I_{\rho,2}$, then the conclusion follows from Proposition \ref{prop max triangle} and the induction hypothesis for $|I_{\rho,2}|=r-1$. If $j\in I_{\rho,3}$, define 
\begin{align*}
    \EE_1&:=\EE^{\rho} \cup (\EE_{\rho,1}+ \EE_{\rho,2}+ add_{j}^{-1}(\EE_{\rho,3})),\\
    \EE_2&:=\EE^{\rho} \cup (\EE_{\rho,1}+ add^{-1}(\EE_{\rho,2})+ add_{j}^{-1}(\EE_{\rho,3})).
\end{align*}
We have $\pi(\EE_1) \neq 0$ by Proposition \ref{prop max triangle} and  $\pi(\EE_{\rho,-})\neq 0$ by assumption. Also, one can check that $\pi(\EE_2)\neq 0$ since $\EE_2= (add_j^{-1}(\EE))_{\rho,-}$. Now Proposition \ref{prop max triangle} and the induction hypothesis on the rank $n$ imply that
\begin{align*}
    \pi(\EE)&=\pi(\EE_1)+\{ \Delta_{\rho}[B_j,-A_j] \}\\
    &=\pi(\EE_2)+\{\Delta_{\rho}[B_{i_2},-A_{i_2}]\}_{i_2 \in I_{\rho,2}}+\{ \Delta_{\rho}[B_j,-A_j] \}\\
    &=\pi(\EE_{\rho,-})+\{\Delta_{\rho}[B_{i_2},-A_{i_2}]\}_{i_2 \in I_{\rho,2}}.
\end{align*}
Finally, by applying Algorithm \ref{algo rep(E)} on $\EE_{\rho,-}$, one can see that any segment $\Delta_{\rho}[x,-y]$ in the $L$-data of $\pi(\EE_{\rho,-})$ must satisfy that $x > B$. Therefore, $\pi(\EE_{\rho,-})=\pi(\EE)_{\rho,-}$. This completes the proof of the proposition. 
\end{proof}

\begin{remark}\label{rmk socle Atobe}
    When $ I_{\rho,1}$ is empty, \cite[Theorem 5.1(2)]{Ato20b} states that 
\begin{align}\label{eq Ato socle}
    \pi(\EE)= S_{\rho\vert\cdot\rvert^{y_1}} \circ \cdots \circ S_{\rho\vert\cdot\rvert^{y_s}} (\pi(\EE_{\rho,-})),
\end{align}
where
$\bigsqcup_{i \in I_{\rho,2}} [B_i,-A_i]_{\rho}=\{\rho\vert\cdot\rvert^{y_1},\dots , \rho\vert\cdot\rvert^{y_s}\}$ 
and $y_1 \geq \cdots \geq y_s$. Here $S_{\rho\vert\cdot\rvert^{z}}(\pi)$ is the unique irreducible subrepresentation of the parabolic induction 
$ \rho\vert\cdot\rvert^{z} \rtimes \pi. $
Then the key point of Proposition \ref{prop reduction along min x} is that the effect of the composition of socles in \eqref{eq Ato socle} is to insert $\{\Delta_{\rho}[B_{i_2},-A_{i_2}]\}_{i \in I_{\rho,2}}$ in the $L$-data of $\pi(\EE_{\rho,-})$.
\end{remark}

We have the following algorithm to compute $\pi(\EE)$, which is an analogue of Algorithm \ref{algo rep(E)}. Here, we use Proposition \ref{prop reduction along min x} as the analogue for Proposition \ref{prop max triangle}.

\begin{algo}[Algorithm to compute the $L$-data of $\pi(\EE)$ applying $(\EE^{|max|})_{\rho,-}$]\label{algo rep(E) 2}
Given an extended multi-segment $\EE \in \Rep$, proceed as follows.
\begin{enumerate}
    \item [\textbf{Step 1:}]Compute $\EE^{|max|}= \cup_{\rho} \{([A_i,B_i]_{\rho},l_i,\eta_i) \}_{i \in (I_{\rho},>)}$. If $\EE^{|max|}$ satisfies the condition $(L)$, then the $L$-data of $\pi(\EE)$ is given by Proposition \ref{prop $L$-data L-packet}.
    \item [\textbf{Step 2:}] If $\EE^{|max|}$ does not satisfy the condition $(L)$, then choose a $\rho$ such that $\EE_{\rho}$ does not satisfy the condition $(L)$. Construct $(\EE^{|max|})_{\rho,-}$ as in Definition \ref{def E lower minus}. Repeating the algorithm for $(\EE^{|max|})_{\rho,-}$, which is of smaller rank than $\EE$, we obtain the $L$-data of $\pi((\EE^{|max|})_{\rho,-})$. Then, by Proposition \ref{prop reduction along min x}, we have
    \[ \pi(\EE)=\pi((\EE^{|max|})_{\rho,-})+ \{\Delta_{\rho}[B_{i_2},-A_{i_2}]\}_{i \in I_{\rho,2}}.\]
\end{enumerate}
\end{algo}

\begin{remark}
We may replace the $\EE^{|max|}$ in the above algorithm by $\EE^{\ast}$ given in \cite[Algorithm 5.6]{Ato20b}, with a modification that we need to replace the definition of $B^{\min}$ and $I_{\rho}^{\min}$ in Step 1 of \cite[Algorithm 5.6]{Ato20b} by $B^{\min}=\min \{B_j\ | \ j \in I_{\rho}, A_j\neq B_j\}$ and $I_{\rho}^{\min}=\{ j \in I_{\rho} \ | \ B_j=B^{\min}, A_j \neq B_j \}$. Note that if $\{B_j\ | \ j \in I_{\rho}, A_j\neq B_j\}=\emptyset,$ then $X_{\rho}(\pi(\EE))$ is already tempered.
\end{remark}

Algorithms \ref{algo rep(E)}, \ref{algo rep(E) 2} are identical on Example \ref{ex algorithm 5.1} since $(\EE^{|max|})^{\rho,-}=(\EE^{|max|})_{\rho,-}$ for those extended multi-segments. The following example below shows the difference between these two algorithms.
\begin{exmp}
Let 
$$\EE= \{([1,0]_{\rho},1,1),([3,1]_{\rho},1,1)\}= \scalebox{0.8}{\bordermatrix{
    &0&1&2&3 \cr
    &\lhd&\rhd&&\cr
    &&\lhd&\oplus&\rhd\cr}}_{\rho}=\EE^{|max|},$$ 
    and set
    $\EE_1= \scalebox{0.8}{\bordermatrix{
    &1&2&3 \cr
    &\lhd&\oplus&\rhd\cr}}_{\rho}$, $\EE_2= \scalebox{0.8}{\bordermatrix{
    &0&1&2&3 \cr
    &\lhd&\rhd&&\cr
    &&&\oplus&\cr}}_{\rho}$, and $\EE_3= \scalebox{0.8}{\bordermatrix{
    &2 \cr
    &\oplus\cr}}_{\rho}.$
    Then $\EE_{1}=\EE_{\rho,-}$, $\EE_{2}=\EE^{\rho,-}$, and $\EE_3=(\EE_{\rho,-})_{\rho,-}=(\EE^{\rho,-})^{\rho,-}$. Applying Algorithm \ref{algo rep(E)}, we have that 
    \begin{align*}
        \pi(\EE_3)&=\pi(2^+),\\ \pi(\EE_2)&= L(\Delta_{\rho}[0,-1]; \pi(2^+)), \\
    \pi(\EE)&=L(\Delta_{\rho}[1,-3], \Delta_{\rho}[0,-1]; \pi(2^+)).
    \end{align*}
    Applying Algorithm \ref{algo rep(E) 2} we have that 
    \begin{align*}
        \pi(\EE_3)&=\pi(2^+), \\
        \pi(\EE_1)&= L(\Delta_{\rho}[1,-3]; \pi(2^+)),\\
        \pi(\EE)&=L(\Delta_{\rho}[1,-3], \Delta_{\rho}[0,-1]; \pi(2^+)).
    \end{align*}
    Thus, the representations involved in the intermediary steps of Algorithms \ref{algo rep(E)} and \ref{algo rep(E) 2} are different.
\end{exmp}

Next, we would like to consider the analogues of Theorem \ref{thm algo for Arthur type} and Algorithm \ref{algo Arthur type} in the case of $\pi_{\rho,-}$. The following is the analogue of Definition \ref{def criterion for Arthur type}.

\begin{defn}\label{def criterion for Arthur type 2}
Let $\pi\in\Pi_{gp}(G_n)$ be of the form 
$\pi= L(\Delta_{\rho_1}[x_1,-y_1], \dots , \Delta_{\rho_f}[x_f,-y_f]; \pi(\phi, \varepsilon)).$
Define
\begin{align*}
    x_{\rho,-}&:= \min\{ x_i\ | \ \rho_i \cong \rho \},\\
    \Delta_{\rho,-}&:=\{ \Delta_{\rho_i}[x_i,-y_i] \ | \ 1 \leq i \leq f, \rho_i \cong \rho, x_{i}=x_{\rho,-}\},\\
    \psi_{\rho,-}&:=\sum_{i \in \Delta_{\rho,-}} \rho \otimes S_{x_i+y_i+1}\otimes S_{y_i-x_i-1},\\
    \psi_{\rho,+}&:=\sum_{i \in \Delta_{\rho,-}} \rho \otimes S_{x_i+y_i+1}\otimes S_{y_i-x_i+1}.
\end{align*}
Note that $\Delta_{\rho,-}$ is a multi-set but not necessarily a set.
\begin{enumerate}
    \item [(1)] Denote by $\Psi(\pi_{\rho,-}; \Delta_{\rho,-})$ the set of local Arthur parameters $\psi$ such that
\begin{enumerate}
    \item [$\oldbullet$] $\pi_{\rho,-} \in \Pi_{\psi}$,
    \item [$\oldbullet$] $\psi \supseteq \psi_{\rho,-}$,
    \item [$\oldbullet$]  and any summand $\rho \otimes S_a \otimes S_b$ of $\psi$ with $a-b\leq 2x_{\rho,-}$ satisfies that $ b=1$.
\end{enumerate}
For any $\psi \in \Psi(\pi_{\rho,-}; \Delta_{\rho,-})$, we define that $\psi_{+}:= \psi- \psi_{\rho,-}+ \psi_{\rho,+}$.
\item [(2)] Let $\mathscr{E}(\pi_{\rho,-}; \Delta_{\rho,-})$ denote the set of extended multi-segments $\EE=\cup_{\rho} \{ [A_i,B_i]_{\rho}, l_i, \eta_i\}_{i \in (I_\rho,>)}\in \Rep^{(P')}$ such that $\pi(\EE)=\pi_{\rho,-}$ and $\psi_{\EE}\in \Psi(\pi_{\rho,-}; \Delta_{\rho,-})$, there exists a unique decomposition $I_{\rho}= I_{\rho,1} \sqcup I_{\rho,2} \sqcup I_{\rho,3}$ such that for any $i_1 \in I_{\rho,1}, i_2 \in I_{\rho,2}, i_{3} \in I_{\rho,3}$,
\begin{enumerate}
    \item [$\oldbullet$] $A_{i_1}=B_{i_1} \leq x_{\rho,-}$,
    \item [$\oldbullet$] 
    $B_{i_2}=x_{\rho,-}$, 
    \item [$\oldbullet$]  and $B_{i_3} >x_{\rho,-}$.
\end{enumerate}
 Let $m$ be the multiplicity of $\Delta_{\rho}[x_{\rho,-},-x_{\rho,-}-1]$ in the $L$-data of $\pi$. For $k\in \{1,2,3\}$, set $\EE_{\rho,k}:= \{([A_i,B_i]_{\rho},l_i,\eta_i)\}_{i \in (I_{\rho,k},>)}$. We define 
\[ \EE_{\rho,+}:=\EE^{\rho} \cup (\EE_{\rho,1}+ \{([x_{\rho,-}+1, x_{\rho,-}]_{\rho},1,1)^{m}\} + add^1(\EE_{\rho,2})+ \EE_{\rho,3}).\]
\end{enumerate}
\end{defn}

The following theorem gives a criterion for whether $\pi$ is of Arthur type using $\Psi(\pi_{\rho,-})$.

\begin{thm}\label{thm algo for Arthur type 2}
If $\pi$ is a representation of $G_n$ of good parity, then $\pi$ is of Arthur type if and only if $\Psi(\pi_{\rho,-}; \Delta_{\rho,-})$ is non-empty. Moreover, for any $\EE \in \mathscr{E}(\pi_{\rho,-}; \Delta_{\rho,-})$, we have  $\pi=\pi(\EE_{\rho,+})$. In other words, for any $ \psi \in \Psi(\pi_{\rho,-}; \Delta_{\rho,-})$, we have  $\psi_{+} \in \Psi(\pi)$.
\end{thm}

\begin{proof}
Assume that $\pi$ is of Arthur type and $\EE$ is an absolutely maximal extended multi-segment such that $\pi=\pi(\EE)$. Then Lemma \ref{len E max lower minus} and Proposition \ref{prop reduction along min x} implies that $\pi_{\rho,-}= \pi(\EE_{\rho,-})$, which is of Arthur type. Moreover, it follows from the definition that $\psi_{\EE_{\rho,-}}$ is in $\Psi(\pi_{\rho,-}; \Delta_{\rho,-})$. 

Conversely, if $\EE\in \mathscr{E}(\pi_{\rho,-}; \Delta_{\rho,-})$ and $\pi(\EE)= \pi_{\rho,-}$. We first check that $\pi(\EE_{\rho,+}) \neq 0$.  Clearly, Theorem \ref{thm non-vanishing}(i)(b) is satisfied. Thus, by Part (i)(a) of the same theorem, we may replace $\EE_{\rho,+}$ by $sh^d(\EE_{\rho,+})$ for some large $d$ so that $B_i\geq 0$ for any $i \in I_{\rho}$. Now we check Theorem \ref{thm non-vanishing}(ii) for $\EE_{\rho,+}$, where the only non-trivial part is to check it for the subset $ \FF:=add^1(\EE_{\rho,2})+ \EE_{\rho,3}$ of $\EE_{\rho,+}$ (see Definition \ref{def criterion for Arthur type 2}(2) for the notation).

Since $\pi(\EE)\neq 0$,  $\EE_{\rho,2}+ \EE_{\rho,3}$ satisfies Theorem \ref{thm non-vanishing}(ii). Thus, we can take another extended multi-segment $\EE'$ such that $\pi(\EE')\neq 0$ and $(\EE')_{\rho}= \EE_{\rho,2}+ \EE_{\rho,3}$; hence , $\pi(dual(\EE'))\neq 0$. Write $dual(\EE')_{\rho}= \widetilde{\EE_{\rho,3}} + \widetilde{\EE_{\rho,2}}$. Let $(\EE')_{\rho,+}= (\EE')^{\rho} \cup \FF$. Then $\FF$ satisfies Theorem \ref{thm non-vanishing}(ii) if and only if $\pi((\EE')_{\rho,+})\neq 0$, which is equivalent to $\pi(dual((\EE')_{\rho,+}))\neq 0$. Now since 
\[ dual((\EE')_{\rho,+})= dual((\EE')^{\rho}) \cup dual(add^1(\EE_{\rho,2})+ \EE_{\rho,3} )=  dual((\EE')^{\rho}) \cup ( \widetilde{\EE_{\rho,3}} + sh^1(\widetilde{\EE_{\rho,2}}))\]
by Lemma \ref{lem equalities of operators}(2), we have $\pi(dual((\EE')_{\rho,+}))\neq 0$ by \cite[Lemma 4.3(i)]{HLL22}. This completes the verification that $\pi(\EE_{\rho,+})\neq 0$.

Finally, since $(\EE_{\rho,+})_{\rho,-}=\EE$, Proposition \ref{prop reduction along min x} implies that $\pi= \pi(\EE_{\rho,+})$ is of Arthur type. Note that $(\psi_{\EE})_{+}=\psi_{\EE_{\rho,+}}$. This completes the proof of the theorem.
\end{proof}

Applying the theorem above, we obtain an  algorithm analogous to Algorithm \ref{algo Arthur type}. 

\begin{algo}[Algorithm to determine Arthur type using $\pi_{\rho,-}$]\label{algo Arthur type 2}
Given a representation $\pi$ of $G_n$ which is of good parity, proceed as follows:

\begin{enumerate}
    \item [\textbf{Step 1:}] If $\pi$ is tempered, then $\pi$ is of Arthur type. Write $\pi=(\phi, \varepsilon)$, where
    $\phi= \bigoplus_{\rho}\bigoplus_{i\in I_{\rho}} \rho \otimes S_{2a_i+1},$
    where we equip $I_{\rho}$ with a total order $>$ such that $a_i$ is non-decreasing. Then let 
    \[ \EE= \cup_{\rho} \{ ([a_i,a_i],0, \varepsilon( \rho \otimes S_{2a_i+1})) \}_{i \in (I_{\rho, >})},\]
    and output $\mathscr{E}(\EE)$ given by \cite[Theorem 7.4]{HLL22}. 
    We regard the trivial representation of $G_0$ as a tempered representation, and output $\{\EE\}$, where $\EE=\emptyset$.
    \item [\textbf{Step 2:}] If $\pi$ is not tempered, then apply the algorithm on $\pi_{\rho,-}$ (see Definition \ref{def pi minus}) to see whether it is of Arthur type. 

If $\pi_{\rho,-}$ is not of Arthur type, then neither is $\pi$ and the procedure ends.

If $\pi_{\rho,-}$ is of Arthur type, then $\pi$ is of Arthur type if and only if the set $\mathscr{E}(\pi_{\rho,-}; \Delta_{\rho,-})$ is nonempty.  If this set is non-empty, choose any $\EE$ in this set and output $\mathscr{E}(\EE_{\rho,+})$ by \cite[Theorem 7.4]{HLL22}.
\end{enumerate}
\end{algo}

For comparison, we apply Algorithm \ref{algo Arthur type 2} to the same representations in Example \ref{ex arthur type}.
\begin{exmp} In this example, $\rho$ still denotes the trivial representation.

1. Consider 
    $\pi= L\left(\Delta_{\rho}\left[ \half{-1},\half{-5}\right], \Delta_{\rho}\left[ \half{-1},\half{-1}\right],\Delta_{\rho}\left[ \half{3},\half{-5}\right]; \pi\left(\half{1}^+,\half{3}^+, \half{5}^+\right)\right).$ 
    Then
    \begin{align*}
        \pi_{\rho,-}=L\left( \Delta_{\rho}\left[ \half{3},\half{-5}\right]; \pi\left(\half{1}^+,\half{3}^+, \half{5}^+\right)\right)\ {\rm and}\ 
        \Delta_{\rho,-}=\left\{\Delta_{\rho}\left[ \half{-1},\half{-5}\right], \Delta_{\rho}\left[ \half{-1},\half{-1}\right] \right\}.
    \end{align*}
    By the same computation as in Example \ref{ex arthur type}, we see that $\pi_{\rho,-}$ is of Arthur type, and the set $\{\EE \ | \ \pi(\EE)=\pi_{\rho,-}\}/\text{(row exchanges)}$ consists of two extended multi-segments
    \[  \scalebox{0.8}{\bordermatrix{ &\half{1}&\half{3}&\half{5} \cr
    &\oplus && \cr 
    &&\oplus& \cr
    &&\lhd&\rhd \cr
    &&&\oplus \cr}}_{\rho},\ \scalebox{0.8}{\bordermatrix{ &\half{1}&\half{3}&\half{5} \cr
    &\oplus && \cr 
    &&\oplus& \ominus\cr
    &&\ominus&\oplus \cr}}_{\rho}.\]
    Thus, $\Psi(\pi_{\rho,-}; \Delta_{\rho,-})$ is empty by definition and $\pi$ is not of Arthur type.

    \quad
    
    2. Consider 
    $\pi= L\left(\Delta_{\rho}\left[ \half{-1},\half{-5}\right], \Delta_{\rho}\left[ \half{-1},\half{-1}\right],\Delta_{\rho}\left[ \half{3},\half{-5}\right]; \pi\left(\half{1}^-,\half{3}^+, \half{5}^-\right)\right).$
    Then
    \begin{align*}
        \pi_{\rho,-}=L\left( \Delta_{\rho}\left[ \half{3},\half{-5}\right]; \pi\left(\half{1}^-,\half{3}^+, \half{5}^-\right)\right)\ {\rm and}\ 
        \Delta_{\rho,-}=\left\{\Delta_{\rho}\left[ \half{-1},\half{-5}\right], \Delta_{\rho}\left[ \half{-1},\half{-1}\right] \right\}.
    \end{align*}
    Again, by the same computation as in Example \ref{ex arthur type}, we see that $\pi_{\rho,-}$ is of Arthur type and $\pi_{\rho,-}=\pi(\EE)$ where
    \[ \EE=  \scalebox{0.8}{\bordermatrix{ &\half{1}&\half{3}&\half{5} \cr
    &\ominus&\oplus& \cr
    &&\lhd&\rhd \cr
    &&&\ominus \cr}}_{\rho}.\]
    Since $\psi_{\EE}\in \Psi(\pi_{\rho,-}; \Delta_{\rho,-})$, $\pi$ is also of Arthur type and $\pi=\pi(\EE_{\rho,+})$, where
\[ \EE_{\rho,+}=  \scalebox{0.8}{\bordermatrix{ &\half{-1}&\half{1}&\half{3}&\half{5} \cr
    &{\color{red}\lhd} & {\color{red}\rhd} & &\cr
    &{\color{red}\lhd}&\ominus &\oplus&{\color{red}\rhd} \cr 
    &&&\lhd&\rhd \cr
    &&&&\ominus \cr}}_{\rho}.\]
\end{exmp}

\subsection{Tadi{\' c}'s family}\label{sec Tadic family}
 In this subsection, applying Algorithm \ref{algo Arthur type}, we give a new proof that the 2-parameter family of representations $\pi_{m,n}$ defined in \cite{Tad22} (see also \eqref{pimn} below) are of Arthur type. Throughout this subsection, we fix an irreducible supercuspidal representation 
$\pi_{sc}=\pi(\phi_{sc},\varepsilon_{sc}) \in \mathcal{C}_{cl}$, and an irreducible supercuspidal representation $\rho \in \mathcal{C}^{sd}$, 
 such that the reducibility point of the pair $(\rho, \pi_{sc})$ is $\alpha:=\alpha_{\rho,\pi_{sc}} \geq 3/2$. The cases that $ \alpha =0,1/2,1$ can be proved similarly, which we omit here. Take any extended multi-segment $\EE_{sc}$ such that $\pi(\EE_{sc})=\pi_{sc}$.

In terms of Langlands classification, the $L$-data of $\pi_{m,n}$ is given by
\begin{equation}\label{pimn}
    \pi_{m,n}= L(\Delta_{\rho}[-\alpha-m,-\alpha-m],\dots,\Delta_{\rho}[-\alpha,-\alpha], \Delta_{\rho}[-\alpha+1,-\alpha+1];\pi_{temp}),
\end{equation}
where $m,n \in \mathbb{Z}_{\geq 0}$, $\pi_{temp}=\pi(\phi,\varepsilon)$ with
$\phi=\phi_{sc}- \rho\otimes S_{2(\alpha-1)+1}+ {\rho\otimes}S_{2(\alpha+n)+1},$ 
for any $\rho'\in\mathcal{C}^{sd}$ and $y\in\frac{1}{2}\mathbb{Z}_{\geq0}$,
\[
\varepsilon(\rho' \otimes S_{2y+1})=\left\{\begin{array}{ll}
   \varepsilon_{sc}(\rho \otimes S_{2(\alpha-1)+1}),  & \mathrm{if} \ y=\alpha+n \ \mathrm{and}  \ \rho'=\rho,  \\
    \varepsilon_{sc}(\rho' \otimes S_{2y+1}), & \mathrm{otherwise.}
\end{array}\right.
\]

If $\alpha=3/2$, then $\pi_{temp}=\pi(\EE_{temp})$, where
\[\EE_{temp}= \EE_{sc}^{\rho} \cup \scalebox{0.8}{\bordermatrix{ & \alpha+n \cr 
&\ominus \cr }}_{\rho}.\]
Note that the sign of the circle at $\alpha+n$ must be $-$  by Theorem \ref{thm characterizatioin of supercuspidal}.
Then applying Algorithm \ref{algo rep(E)}, we see that $\pi_{m,n}=\pi(\EE_{m,n}),$ where
\[ \EE_{m,n}= \EE_{sc}^{\rho} \cup \scalebox{0.8}{\bordermatrix{
&- \alpha-m & \cdots &-\half{1} & \half{1} & \cdots & \alpha+m & \cdots &\alpha+ n \cr 
&\lhd &\cdots &\lhd&\rhd &\cdots&\rhd& &\cr 
&&&&&&&&\ominus\cr
}}_{\rho}.\]
Therefore, $\pi_{m,n}$ is of Arthur type. Here $\alpha+m$ can exceed $\alpha+n$ in the symbol.

For $\alpha \geq 2$, we claim that there exists an extended multi-segment $\EE$ such that $\pi(\EE)=\pi_{temp}$ and the first extended segment of $\EE_{\rho}$ is $([\alpha-2,-(\alpha-2)]_{\rho},\ast,\ast)$. We give an explicit construction of $\EE$ below. Note that the results in \cite[\S 8]{HLL22} already guarantee the existence of such $\EE$ abstractly.

When $\alpha=2$ (resp. $\alpha=5/2$), we have  $\pi_{temp}=\pi(\EE_{temp})$, where
\[  \EE_{temp}=\EE_{sc}^{\rho}\cup \scalebox{0.8}{\bordermatrix{ & 0& \cdots &\alpha+n \cr 
&\odot&  & \cr 
&&&\odot\cr }} \ \text{(resp. }  \EE_{temp}= \EE_{sc}^{\rho}\cup \scalebox{0.8}{\bordermatrix{ &-\half{1}& \half{1}& \cdots &\alpha+n \cr 
&\oplus&\ominus&  & \cr 
&&&&\oplus\cr }}\text{)}, \]
which confirms the claim in this case. When $\alpha \geq 3 $, recall that $\epsilon_{\rho}\in \{ 0,\half{1} \}$ is the number such that $\alpha+\epsilon_{\rho}\in \Z$. We have $\pi_{temp}=\pi(\EE_{temp})$, where
\[\EE_{temp}=\EE_{sc}^{\rho}\cup\{([ \alpha-2,\epsilon_{\rho}]_{\rho}, 0, \varepsilon(\rho\otimes S_{2\epsilon_{\rho}+1})),([\alpha+n,\alpha+n],0,\varepsilon(\rho\otimes S_{2(\alpha+n)+1}))\}.\]
Then we apply $dual \circ ui^{-1} \circ dual$ on $\EE_{temp}$ as follows. First, applying the formula of $dual$, we obtain that
\[ dual(\EE_{temp})= \{([\alpha+n,-\alpha-n]_{\rho}, \ast,\ast), ([\alpha-2, -\epsilon_{\rho}]_{\rho}, 0,\ast)\}.\]
 Then we apply $ui^{-1}$ of type 3$'$ (see Lemma \ref{cor ui inverse type 3'}) to break $([\alpha-2, -\epsilon_{\rho}]_{\rho}, 0,\ast)$ into
\[ \{ ([\alpha-3, -\epsilon_{\rho}]_{\rho}, 0,\ast),([\alpha-2, \alpha-2]_{\rho}, 0,\ast)\}.\]
Taking dual again, the resulting extended multi-segment is of the form
\[ 
{ \EE=\EE_{sc}^{\rho} \cup \scalebox{0.8}{\bordermatrix{ 
&-\alpha-2 & \cdots &\epsilon_{\rho} & \cdots  &\alpha-3&\alpha-2 &\cdots& \alpha+n \cr 
&\lhd  &\cdots&\cdots&\cdots&\rhd&\rhd && \cr 
&&& \odot &\cdots &\odot &&&\cr 
&&&&&&&&\odot\cr
}}},\]
which verifies the claim.

Now applying Algorithm \ref{algo rep(E)}, we obtain that $\pi_{m,n}=\pi(\EE_{m,n})$ where $\EE_{m,n}= \EE_{sc}^{\rho} \cup (add_{1}^{m+2}(\EE_{\rho}))$; hence, $\pi_{m,n}$ is of Arthur type. If we define $\EE_{m,n}'$ by breaking $([\alpha-3,\epsilon_{\rho}]_{\rho},0,\ast)$ in $\EE_{m,n}$ into $ \{([z,z]_{\rho},0, \ast)\}_{z=\epsilon_{\rho}}^{\alpha-3}$ (i.e., applying $ui^{-1}$), namely, 
\[ 
{ \EE_{m,n}'=\EE_{sc}^{\rho} \cup \scalebox{0.8}{\bordermatrix{ 
&-\alpha-m & \cdots &\epsilon_{\rho} & \cdots  &\alpha-3& \cdots&\alpha+m &\cdots& \alpha+n \cr 
&\lhd  &\cdots&\cdots&\cdots&\rhd& \cdots&\rhd && \cr 
&&& \odot & & &&&\cr 
&&&  &\ddots & &&&\cr 
&&&  & &\odot &&&\cr 
&&&&&&&&&\odot\cr
}}},\]
then $\pi_{m,n} \in \Pi_{\psi_{\EE_{m,n}'}}$, where
\begin{align*}
    \psi_{\EE_{m,n}'}=& (\phi_{sc}- \rho\otimes S_{2(\alpha-1)+1}- \rho\otimes S_{2(\alpha-2)+1} ) \otimes S_1
    + \rho\otimes S_{1} \otimes S_{2(
\alpha+m)+1} + \rho\otimes S_{2(\alpha+n)+1} \otimes S_1,
\end{align*}
which matches $\psi_{2\alpha+1+2m, 2 \alpha+1+2n}$ in \cite[Theorem 4.2(1)]{Tad22}. Finally, it is not hard to check that $dual((\EE_{m,n})_{\rho})=(\EE_{n,m})_{\rho}$ if $\alpha\in \Z$, and $dual((\EE_{m,n})_{\rho})=dual_2^{+}( (\EE_{n,m})_{\rho})$ if $\alpha \in \half{1}+\Z$, which implies that $\widehat{\pi_{m,n}}=\pi_{n,m}$.

\subsection{Preservation and independence for Arthur representations}\label{sec proof of preservation}
With the notation developed in this section, we state and prove the analogues of Theorem \ref{thm weakly real} and Conjectures \ref{conj Tadic preservation}, \ref{conj independence} for Arthur representations.
Recall that we say a representation $\pi \in \Pi(G)$ is of Arthur type if $\pi \in \Pi_A(G)$ (not $\Pi_{A+}(G)$). 

\begin{thm}\label{thm preservation and independence}
    Let $\pi \in \Pi(G_n)$ and recall that $X^{sd}= \sqcup_{\rho \in \mathcal{C}^{sd}} X_{\rho}$. 
    \begin{enumerate}
        \item [(a)] Assume that $\pi= \theta \rtimes X^{sd}(\pi)$ where $\theta$ is Hermitian. Then $\pi$ is of Arthur type if and only if both $\theta$ and $X^{sd}(\pi)$ are of Arthur type.
        \item [(b)] Assume that $\sigma \in \mathcal{C}_{cl}$ and $\pi \in \Irr(X_{\rho_1}\sqcup\cdots \sqcup X_{\rho_r}; \sigma)$ is weakly real. Then $\pi$ is of Arthur type if and only if $X_{\rho_i}(\pi)$ are all of Arthur type. Moreover, suppose further that $\pi$ is of good parity and write $\sigma=\pi(\EE_{sc})$ with  $\EE_{sc}=\cup_{\rho} (\EE_{sc})_{\rho}$. If $\pi$ is of Arthur type with $\pi=\pi(\cup_{\rho} \EE_{\rho})$, then 
        \begin{align}\label{eq preservation A-type 1}
            X_{\rho_i}(\pi)= \pi\left( \EE_{\rho_i} \cup (\EE_{sc})^{\rho_i}  \right).
        \end{align}
        Conversely, if $X_{\rho_i}(\pi)$ is of Arthur type and $ X_{\rho_i}(\pi)= \pi(\cup_{\rho}\EE_{i,\rho})$, then
        \begin{align}\label{eq preservation A-type 2}
            \pi= \pi\left( \bigcup_{i=1}^{r} \EE_{i,\rho_i} \cup \bigcup_{\rho' \not\cong \rho_1,\ldots, \rho_r} (\EE_{sc})_{\rho'}  \right). 
        \end{align}
        \item [(c)] Assume that $\pi_1\in \Irr(\rho_1;\sigma_1)$, $\rho_2 \in \mathcal{C}^{sd}$, and $\sigma_2\in \mathcal{C}_{cl}$, such that $\alpha_{\rho_1,\sigma_1}= \alpha_{\rho_2,\sigma_2}$. Then $\pi_1$ is of Arthur type if and only if $E(\pi_1)$ is of Arthur type. Moreover, suppose further that $\pi_1$ is of good parity and write $\sigma_k=\pi(\EE_{k,sc})$ with  $\EE_{k,sc}=\cup_{\rho} (\EE_{k,sc})_{\rho}$.      If $\pi_1$ is of Arthur type with 
        \[\pi_1=\pi\left(\{([A_i,B_i]_{\rho_1},l_i,\eta_i)\}_{i \in (I_{\rho_1},>)} \cup \bigcup_{\rho \not \cong \rho_1}\EE_{1,\rho}\right), \]
        then (note that we replace $[A_i,B_i]_{\rho_1}$ by $[A_i,B_i]_{\rho_2}$)
        \[ E(\pi_1)= \pi\left( \{([A_i,B_i]_{\rho_2},l_i,\eta_i)\}_{i \in (I_{\rho_1},>)}  \cup   \bigcup_{\rho \not\cong \rho_2} (\EE_{k,sc})_{\rho} \right).\]
    \end{enumerate}
\end{thm}

\begin{proof}
    Part (a) is a consequence of Theorem \ref{thm red from nu to gp}. Indeed, it implies that if $\theta$ and $X^{sd}(\pi)$ are both of Arthur type, then so is $\pi$. Conversely, assume that $\pi$ is of Arthur type. By the unitary parabolic reduction (see \cite[\S 2.15]{Tad23} for example), $\theta$ must be a unitary representation of $\GL_d(F)$ for some $d$; hence, $\theta \in \Pi_{\psi_{\theta}}$ for a unique $\psi_{\theta} \in \Psi^{+}_{1/2}(\GL_d(F))$ by the classification of unitary dual for general linear group. Then Theorem \ref{thm red from nu to gp} implies that
    \begin{align*}
         \{\psi \in \Psi^{+}_{1/2}(G_n) \ | \ \pi \in \Pi_{\psi}\}= \{ \psi_{\theta}+ \psi^{sd}+ \psi_{\theta}^{\vee} \ | \ X^{sd}(\pi)\in \Pi_{\psi^{sd}} \}.
    \end{align*}
    If there exists a $\psi$ in the left hand side above that lies in $ \Psi(G_n)$ and write $\psi=\psi_{\theta}+ \psi^{sd}+ \psi_{\theta}^{\vee}$, then we must have that $\psi_{\theta} \in \Psi(\GL_d(F))$ and $\psi^{sd} \in \Psi(G_{n-d})$. We conclude that $\theta \in \Pi_A(\GL_d(F))$ and $X^{sd}(\pi)\in \Pi_{A}(G_{n-d})$. This completes the verification of Part (a).

    For Part (b), we may assume that $\pi$ is of good parity again by Theorem \ref{thm red from nu to gp}. Thus, it suffices to show Equations \eqref{eq preservation A-type 1} and \eqref{eq preservation A-type 2}. 

   {We first prove \eqref{eq preservation A-type 1} in the special case that $\pi$ is tempered. In this case, by the process in \S \ref{sec classification of tempered} below, there exists $\theta_j$ supported in $X_{\rho_j}$ such that
    \[ \pi \hookrightarrow \left(\bigtimes_{j \neq i} \theta_j \right)\rtimes \pi( (\EE_{sc})^{\rho_i} \cup \EE_{\rho_i}). \]
    Moreover, $\pi( (\EE_{sc})^{\rho_i} \cup \EE_{\rho_i})= X_{\rho_i}(\pi)$. This proves \eqref{eq preservation A-type 1} in this case.

    Now we consider general $\pi$. Applying Algorithm \ref{algo rep(E)} to $\EE$, we obtain an injection
    \begin{align}\label{eq injection independence}
        \pi \hookrightarrow \bigtimes_{1\leq j \leq r} \theta_j \rtimes \pi(\EE_{temp})
    \end{align}
    where
    \begin{itemize}
        \item $\theta_j= \bigtimes_{k \in I_j}\Delta_{\rho_j}[x_k, y_k]$ is supported in $X_{\rho_j}$, and
        \item $\psi_{\EE_{temp}}$ is tempered. 
    \end{itemize}
    Note that each $\Delta_{\rho_j}[x_k, y_k]$ is obtained from Step 2 of Algorithm \ref{algo rep(E)} with $\rho=\rho_j$, and the right hand side of \eqref{eq injection independence} is indeed a standard module, which gives the $L$-data of $\pi$. Applying the same process to the right hand side of \eqref{eq preservation A-type 1}, denoted by $\pi_i$ for short, we obtain that
    \begin{align}\label{eq injection independence 2}
       \pi_i:= \pi\left( \EE_{\rho_i} \cup \bigcup_{\rho \not\cong \rho_i} (\EE_{sc})_{\rho}  \right) \hookrightarrow \theta_i \rtimes \pi( (\EE_{sc})^{\rho_i} \cup (\EE_{temp})_{\rho_i} )= \theta_i \rtimes X_{\rho_i}(\pi(\EE_{temp})),
    \end{align}
    where the equality follows from the previous special case for the tempered representation $\pi(\EE_{temp})$. Again, note that the right hand side of \eqref{eq injection independence 2} is a standard module, which gives the $L$-data of $\pi_i$. Now comparing the $L$-data of $\pi$ and $\pi_i$, we conclude that $X_{\rho_i}(\pi)= \pi_i$ by \cite[Corollary 8.6]{Jan97}. This completes the proof of \eqref{eq preservation A-type 1}.

    Conversely, assume that $X_{\rho_i}(\pi)$ are all of Arthur type with $X_{\rho_i}(\pi)=\pi( \cup_{\rho} \EE_{i,\rho})$. Let
    \[ \EE=  \bigcup_{i=1}^{r} \EE_{i,\rho_i} \cup \bigcup_{\rho \not\cong \rho_1,\ldots, \rho_r} (\EE_{sc})_{\rho}.\]
    Since $\pi( \cup_{\rho} \EE_{i,\rho}) \neq 0$, its $\rho_i$-part $\EE_{i,\rho_i}$ must satisfy  the non-vanishing condition in Theorem \ref{thm non-vanishing}. We see that $\EE$ also satisfies the non-vanishing condition, and we let $\pi':= \pi(\EE)$ be the corresponding (nonzero) irreducible representation. Then for any $1 \leq i \leq r$, we obtain that
    \[X_{\rho_i}(\pi')= \pi( \EE_{i,\rho_i} \cup \bigcup_{\rho \not\cong \rho_i} (\EE_{sc})_{\rho}) = X_{\rho_i}(X_{\rho_i}(\pi))= X_{\rho_i}(\pi), \]
    where the first and second equality follows from \eqref{eq injection independence} for $\pi'$ and $X_{\rho_i}(\pi)$ respectively. We conclude that $\pi=\pi'$ by Theorem \ref{thm Jantzen}(ii). This completes the proof of \eqref{eq preservation A-type 2} and Part (b). Part (c) follows from the same argument of Part (b), which we omit. This completes the proof of the theorem.    
}
\end{proof}

\begin{remark}
Applying Theorem \ref{thm red from nu to gp}, Theorem \ref{thm preservation and independence} is true if we replace $\Pi_{A}$ by $\Pi_{A+}$. Additionally, the good parity condition in the second part of Part (b) can be dropped, but it complicates the notation.
\end{remark}

Theorem \ref{thm preservation and independence} allows us to prove the analogues of Conjectures \ref{conj Tadic preservation} and \ref{conj independence} for $\Pi_{\overline{A}}(G_n)$ and $\Pi_{\overline{A}}^{\lim}(G_n)$.

\begin{cor}\label{cor preseveration for Pi A bar}
The analogues of Conjectures \ref{conj Tadic preservation} and \ref{conj independence} hold for both $\Pi_{\overline{A}}(G_n)$ and $\Pi_{\overline{A}}^{\lim}(G_n)$.
\end{cor}

\begin{proof}
We claim that the following holds.
\begin{itemize}
    \item [(i)]For any $\pi \in \Pi(G_n)$,
    $\pi \in \Pi_{\overline{A}}^{(k)}(G_n)$ if and only if $X_{\rho}(\pi)\in \Pi_{\overline{A}}^{(k)}(G_n)$ for any $\rho$.
    \item [(ii)]For any $\pi_1 \in \Irr(X_{\rho_1}; \sigma_1)$, $\pi_1 \in \Pi_{\overline{A}}^{(k)}(G_n)$ if and only if $E(\pi_1)\in \Pi_{\overline{A}}^{(k)}(G_n)$.
\end{itemize}  
Indeed, the base case that $k=0$ follows from Theorem \ref{thm preservation and independence}(b) and (c). For claim (i), the irreducibility of the parabolic induction $\Ind_{P}^G \tau\otimes\chi$ in Definition \ref{def closure A}(1) is equivalent to the irreducibility of the collection of the parabolic induction $\{ \Ind_{P}^G X_{\rho} (\tau \otimes \chi)\}_{\rho}$ by Theorem \ref{thm Jantzen}(iii). For statement (ii), the parabolic induction 
\[ \bigtimes_{i=1}^s u_{\rho_1}(a_i,b_i)\lvert\cdot\rvert^{x_i} \rtimes \pi_{1,A}\]
is irreducible if and only if the parabolic induction
\[ \bigtimes_{i=1}^s u_{\rho_2}(a_i,b_i)\lvert\cdot\rvert^{x_i} \rtimes E(\pi_{1,A})\]
is irreducible by \cite[Theorem 1.2]{BS25}, \cite[Theorem 1.1]{Tad14} and \cite[Corollary 1.3]{Ato22d} due to the combinatorial nature of the extended multi-segment and derivatives.  This completes the proof of the claim.
  
Now the claim implies the analogues hold for $\Pi_{\overline{A}}(G_n)$. The same holds for  $\Pi_{\overline{A}}^{\lim}(G_n)$ since the process of taking limit commutes with the map $\pi\mapsto X_{\rho_i}(\pi)$ and $\pi_1 \mapsto E(\pi_1)$. This completes the proof of the corollary.
\end{proof}

\section{Inductive Classification of Arthur representations}
\label{sec inductive approach}

In this section, based on the algorithms introduced in \S \ref{sec new algorithm}, we describe an inductive approach to classify all Arthur representations of arbitrary corank for $G_n$, with an explicit algorithm (Algorithm \ref{alg Pi A}).

\subsection{\texorpdfstring{Classification of tempered representations of corank $r$}{}}\label{sec classification of tempered}

Given a tempered representation $\pi_{temp}$ of $G_n$, in this subsection we construct  an embedding
\begin{align}\label{eq injection of tempered}
     \pi_{temp} \hookrightarrow \rho_1\times \cdots \times \rho_k \rtimes \pi_{sc},
\end{align} 
where $\rho_i \in \mathcal{C}$ and $\pi_{sc}\in \mathcal{C}_{cl}$. This determines the corank of the tempered representation $\pi_{temp}$. Then, by reversing this process, we describe how to (non-uniquely) parameterize tempered representations of corank $r$ inductively. Our arguments rely on the local Langlands correspondence established in \cite{Art13} (see Theorem \ref{thm Arthur tempered}) and the computation of derivatives for tempered representations in  
\cite[Theorem 3.6]{Ato22a}.

We begin by fixing some notation and convention. By Theorem \ref{thm Arthur tempered}, a tempered representation $\pi$ of $G_n$ corresponds to a pair $(\phi, \varepsilon)$, where $\phi$ is a tempered local $L$-parameter and $\varepsilon \in \widehat{\mathcal{S}}_{\phi}$. We regard $\phi$ as a tempered local Arthur parameter (specifically identifying it with $\phi\otimes S_1$) and decompose it as
\[\phi= (\phi_{ngp}+\phi_{ngp}^{\vee})+ \phi_{gp},\]
by Equation \eqref{eq decomp red to gp}. Further,  we write
$\phi_{gp}= \bigoplus_{i \in I} \rho_i \otimes S_{2x_i+1}.$ 
The character $\varepsilon$ is identified with a function $ \varepsilon: I \to \{\pm 1\}$ subject to certain constraints. We extend $\varepsilon$ to a function on $\{\rho\otimes S_{a}\}$, the set of all (nonzero) irreducible 
representations of $W_F \times \SL_2(\BC)$ which are bounded on $W_F$, by defining 
\begin{align*}
    \varepsilon(\rho \otimes S_{a}):= \begin{cases}
        \varepsilon(i) & \text{ if }\rho\otimes S_{a} \cong \rho_i \otimes S_{2x_i+1} \text{ for some }i \in I_{gp},\\
       0& \text{ otherwise.}
    \end{cases}
\end{align*}
 Let $m_{\phi}(\rho\otimes S_{a})$ denote the multiplicity of $ \rho\otimes S_{a}$ in $\phi$. By convention, we set $\varepsilon(\rho\otimes S_0):= -1$ and $m_{\phi}(\rho \otimes S_0):= +\infty$, which will be used in Theorem \ref{thm temp algo} below.

The following theorem is a combination of Theorem \ref{thm characterizatioin of supercuspidal} and the computation of derivatives of tempered representations.

\begin{thm}\label{thm temp algo}
Let $\pi=\pi(\phi,\varepsilon)$ be a tempered representation of good parity. Then, $\pi$ is not supercuspidal if and only if there exist $\rho\in \CC^{sd}, x\in \half{1}\Z_{\geq 0}$,  with $m:=m_{\phi}(\rho \otimes S_{2x+1})> 0$, and at least one of the following holds.
\begin{enumerate}
    \item [(I)] $x>0$ and $ \varepsilon(\rho\otimes S_{2x+1})\varepsilon(\rho\otimes S_{2x-1}) \neq -1.$ In this case, there exists a unique tempered representation $\pi_{temp}$ of smaller rank such that
   \[\pi \hookrightarrow \underbrace{\rho\vert\cdot\rvert^{x}\times \cdots \times \rho\vert\cdot\rvert^{x}}_{m \text{ copies}} \rtimes \pi_{temp} , \]
    and we write $\pi= \Temp{I}{x}{m}(\pi_{temp}).$ More precisely, we have $\pi_{temp}=\pi(\phi',\varepsilon')$, where 
    $\phi'= \phi- (\rho \otimes S_{2x+1})^{\oplus m}+(\rho \otimes S_{2x-1})^{\oplus m},$
    and for any $\rho'\otimes S_{a}$,
  \[ \varepsilon'(\rho'\otimes S_{a})=\begin{cases}
    \varepsilon(\rho\otimes S_{2x+1}) & \text{ if }\rho'\otimes S_{a} \cong \rho \otimes S_{2x-1},\\
      0 & \text{ if }\rho'\otimes S_{a} \cong \rho \otimes S_{2x+1},\\
    \varepsilon(\rho'\otimes S_{a}) & \text{ otherwise.}\end{cases}\]

    \item [(II)] $x>0$, $m=m_{\phi}(\rho\otimes S_{2x+1})>1$ is \textbf{odd}, and
    $ \varepsilon(\rho\otimes S_{2x+1})\varepsilon(\rho\otimes S_{2x-1})=-1 .$ In this case, there exists a unique tempered representation $\pi_{temp}$ of smaller rank such that
     \[\pi \hookrightarrow \underbrace{\rho\vert\cdot\rvert^{x}\times \cdots \times \rho\vert\cdot\rvert^{x}}_{m -1\text{ copies}} \rtimes \pi_{temp}, \]
    and we write $\pi= \Temp{II}{x}{m}(\pi_{temp}).$ More precisely, we have $\pi_{temp}=\pi(\phi',\varepsilon')$, where 
    $\phi'= \phi- (\rho \otimes S_{2x+1})^{\oplus (m-1)}+(\rho \otimes S_{2x-1})^{\oplus (m-1)},$
    and for any $\rho'\otimes S_{a}$, $\varepsilon'(\rho'\otimes S_{a})=\varepsilon(\rho'\otimes S_{a})$.
    
    \item[(III)] $x>0$, $m=m_{\phi}(\rho\otimes S_{2x+1})>1$ is \textbf{even}, and $ \varepsilon(\rho\otimes S_{2x+1})\varepsilon(\rho\otimes S_{2x-1})=-1 .$ In this case, there exists a unique tempered representation $\pi_{temp}$ of smaller rank such that
    \begin{align*}
        \pi \hookrightarrow \underbrace{\rho\vert\cdot\rvert^{x}\times \cdots \times \rho\vert\cdot\rvert^{x}}_{m -1\text{ copies}} \times \rho\vert\cdot\rvert^{x-1} \times \rho\vert\cdot\rvert^{x-2} \times \cdots \times \rho\vert\cdot\rvert^{-x} \rtimes \pi_{temp} ,
    \end{align*}
    and we write $\pi= \Temp{III}{x}{m}(\pi_{temp}).$ More precisely, we have  $\pi_{temp}=\pi(\phi',\varepsilon')$, where 
    \[ \phi'= \phi- (\rho \otimes S_{2x+1})^{\oplus m}+(\rho \otimes S_{2x-1})^{\oplus (m-2)},\]
    and for any $\rho'\otimes S_{a}$, $\varepsilon'(\rho'\otimes S_{a})=\varepsilon(\rho'\otimes S_{a})$.
    
    \item [(IV)] $x=0$ and $m=m_{\phi}(\rho\otimes S_1)>1$ is odd. In this case, there exists a unique tempered representation $\pi_{temp}$ of smaller rank such that
    \[\pi \hookrightarrow \underbrace{\rho\times \cdots \times \rho}_{(m-1)/2 \text{ copies}} \rtimes \pi_{temp}, \] 
    and we write $\pi= \Temp{IV}{}{m}(\pi_{temp}).$ More precisely, we have  $\pi_{temp}=\pi(\phi',\varepsilon')$, where 
    $\phi'= \phi- (\rho \otimes S_{1})^{\oplus (m-1)},$ 
    and for any $\rho'\otimes S_{a}$, $\varepsilon'(\rho'\otimes S_{a})=\varepsilon(\rho'\otimes S_{a})$. 
    
    \item [(V)] $x=0$ and $m=m_{\phi}(\rho\otimes S_1)>1$ is even.  In this case, there exists a unique tempered representation $\pi_{temp}$ of smaller rank such that
    \[\pi \hookrightarrow \underbrace{\rho\times \cdots \times \rho}_{m/2 \text{ copies}} \rtimes \pi_{temp}, \] 
    and we write $\pi= \Temp{V}{\varepsilon(\rho\otimes S_{1})}{m}(\pi_{temp}).$ More precisely, we have  $\pi_{temp}=\pi(\phi',\varepsilon')$, where 
    $\phi'= \phi- (\rho \otimes S_{1})^{\oplus m},$ 
    and for any $\rho'\otimes S_{a}$, 
     \[ \varepsilon'(\rho'\otimes S_{a})=\begin{cases}0 & \text{ if }\rho'\otimes S_{a} \cong \rho \otimes S_{1},\\
  \varepsilon(\rho'\otimes S_{a}) & \text{ otherwise.}\end{cases}\] 
\end{enumerate}
\end{thm}
\begin{proof}
The classification of the five cases follows from Theorem \ref{thm characterizatioin of supercuspidal}.
Assume that $\rho\otimes S_{2x+1} \subseteq \phi_{gp}$. Then the injection in cases (I), (II), (IV), (V) directly follows from \cite[Proposition 3.6(1), (2), (3)]{Ato22a} and the Frobenius reciprocity. For case (III), we have 
\begin{align*}
    \pi \hookrightarrow & \underbrace{\rho\vert\cdot\rvert^{x}\times \cdots \times \rho\vert\cdot\rvert^{x}}_{m -1\text{ copies}}
         \rtimes L(\Delta_{\rho}[x-1,-x];\pi( \phi- (\rho\otimes S_{2x+1})^{\oplus m}+\rho\otimes S_{2x-1})^{\oplus (m-2)}, \varepsilon'))\\
         \hookrightarrow  & \underbrace{\rho\vert\cdot\rvert^{x}\times \cdots \times \rho\vert\cdot\rvert^{x}}_{m -1\text{ copies}}
         \times \rho\vert\cdot\rvert^{x-1} \times \rho\vert\cdot\rvert^{x-2} \times \cdots \times \rho\vert\cdot\rvert^{-x}\\
         &\ \ \ \ \ \ \rtimes \pi( \phi- (\rho\otimes S_{2x+1})^{\oplus m}+\rho\otimes S_{2x-1})^{\oplus (m-2)}, \varepsilon'),
\end{align*}
by \cite[Proposition 3.6(2)]{Ato22a}. This completes the proof of the theorem.
\end{proof}

By iteratively applying Theorem \ref{thm temp algo}, we obtain an injection of the form \eqref{eq injection of tempered}. Reversing the process yields the following corollary, which classifies tempered representations of corank $r$ inductively, assuming that the classification of tempered representations of corank $k$ is known for all $k<r$. 
\begin{cor}\label{cor corank of tempered}
Let $r > 0$. Every tempered representation $\pi$ of corank $r$ is one of the following forms: 
\begin{enumerate}
    \item [$\oldbullet$]$\pi=\Temp{I}{x}{m}(\pi_{temp})$, where $x>0$, $m>0$, and $\pi_{temp}$ has corank $r-m$;
    \item [$\oldbullet$]$\pi=\Temp{II}{x}{m}(\pi_{temp})$, where $x>0$, $m>1$ is odd, and $\pi_{temp}$ has corank $r-(m-1)$;
    \item [$\oldbullet$]$\pi=\Temp{III}{x}{m}(\pi_{temp})$, where $x>0$, $m>1$ is even, and $\pi_{temp}$ has corank $r-(m+2x-1)$;
    \item [$\oldbullet$]$\pi=\Temp{IV}{}{m}(\pi_{temp})$, where $m>0$ is odd  and $\pi_{temp}$ has corank $r-(m-1)/2$;
     \item [$\oldbullet$]$\pi=\Temp{V}{\epsilon}{m}(\pi_{temp})$, where $m>0$ is even, $\epsilon\in \{0, \pm 1\}$, and $\pi_{temp}$ has corank $r-m/2$. 
\end{enumerate}
\end{cor}

At the end of this subsection, we summarize the constraints on the triple $(z,x,m)$ and the local $L$-parameter $(\phi',\epsilon')$ of $\pi_{temp}$ required for $T_{z,m}^x(\pi_{temp})$  to be well-defined.

\begin{remark}\label{rmk well-defined for Temp}
    Let $\pi_{temp}= \pi(\phi', \varepsilon')$.
    \begin{enumerate}
        \item [$\oldbullet$] The representation $\Temp{I}{x}{m}(\pi_{temp})$ is well-defined if and only if $x>0$, $m>0$, $m_{\phi'}(\rho\otimes S_{2x+1})=0$, and $m_{\phi'}(\rho\otimes S_{2x-1})\geq m$. 
         \item [$\oldbullet$] The representation $\Temp{II}{x}{m}(\pi_{temp})$ is well-defined if and only if $x>0$, $m>0$ is odd, $m_{\phi'}(\rho\otimes S_{2x+1})=1$, $m_{\phi'}(\rho\otimes S_{2x-1})\geq m$, and $\varepsilon'(\rho\otimes S_{2x+1})\varepsilon'(\rho\otimes S_{2x-1})=-1$.
        \item [$\oldbullet$] The representation $\Temp{III}{x}{m}(\pi_{temp})$ is well-defined if and only if $x>0$, $m>0$ is even, $m_{\phi'}(\rho\otimes S_{2x+1})=0$, and $m_{\phi'}(\rho\otimes S_{2x-1})\geq m-1$.
         \item [$\oldbullet$] The representation $\Temp{IV}{}{m}(\pi_{temp})$ is well-defined if and only if $m>0$ is odd and $m_{\phi'}(\rho\otimes S_{1})=1$.
         \item [$\oldbullet$] The representation $\Temp{V}{\epsilon}{m}(\pi_{temp})$ is well-defined if and only if $m>0$ is even and $m_{\phi'}(\rho\otimes S_{1})=0$.
    \end{enumerate}
\end{remark}

\subsection{\texorpdfstring{Classification of non-tempered Arthur representations of corank $r$}{}}
\label{subsec inductive approach}

Let $\pi \in \Pi(G_n)$ be a non-tempered representation of corank $r$, expressed as $\pi^{\rho,-}+\{(\Delta_{\rho}[x,-y])^{t}\}$ with $t>0$, as in Definition \ref{def pi minus}. Then $\pi^{\rho,-}$ is of corank $r-t(x+y+1)$, which is strictly smaller than $r$. Therefore, according to Theorem \ref{thm algo for Arthur type} and Algorithm \ref{algo Arthur type}, given $r \in \Z_{>0}$, we can classify all good parity non-tempered Arthur representations of corank $r$ assuming the following:
\begin{enumerate}
    \item [(i)] We have classified all good parity Arthur representations of corank $s$, for any $s<r$.
    \item [(ii)] We understand $\Psi(\pi^{\rho,-})$ ``well enough" for those $\pi^{\rho,-}$ of corank $s<r$ based on the classification in (i). 
\end{enumerate}
We then can classify all non-tempered Arthur representations of corank $r$ applying Theorem \ref{thm red from nu to gp}.

We now clarify what is meant by ``well enough" in (ii) above. If we have an exhaustion of  $\Psi(\pi^{\rho,-})$ for all $\pi^{\rho,-}$ of corank $s<r$, then we can classify the corank $r$ representations by applying Theorem \ref{thm algo for Arthur type}. For example, when $r=1$, then $\pi^{\rho,-}$ must be tempered and we have a complete characterization of $\Psi(\pi^{\rho,-})$ in \cite[\S 8]{HLL22}. However, the exhaustion of $\Psi(\pi^{\rho,-})$ is usually complicated and not necessary for our purpose. Indeed, 
as demonstrated in Example \ref{ex arthur type}, in order to show that $\pi$ is of Arthur type, it is sufficient to construct a special element $\psi^{-}\in \Psi(\pi^{\rho,-};\Delta_{\rho}[x,-y],t)$ or an extended multi-segment $\EE^{\rho,-}\in \mathscr{E}(\pi^{\rho,-};\Delta_{\rho}[x,-y],t)$.

\begin{remark}
    We remark that the above classification can alternatively be conducted by applying Theorem~\ref{thm algo for Arthur type 2} and Algorithm~\ref{algo Arthur type 2} instead of Theorem~\ref{thm algo for Arthur type} and Algorithm~\ref{algo Arthur type}. 
However, this alternative approach is generally more intricate because $\Delta_{\rho,-}$ in Definition~\ref{def criterion for Arthur type 2} is more flexible than the segments $\Delta_{\rho}[x,-y]^{t}$ removed in the definition of $\pi^{\rho,-}$. Consequently, enumerating all $\pi$ of corank $r$ from a fixed $\pi_{\rho,-}$ becomes more challenging. However, for $r \leq 3$, the difference is negligible.
\end{remark}

In the following, we make it precise on the above algorithm
to exhaust $\Pi_{A,\,gp}(G_n) \cap \Irr(\rho; \pi_{sc})$ (see Definition \ref{def Jantzen decomposition}), which can be extended to general case by Theorem \ref{thm preservation and independence} and Theorem \ref{thm red from nu to gp}.

\begin{algo}\label{alg Pi A}
Let $\pi_{sc}$ be a supercuspidal representation of $G_n$, $\rho$ be a {selfdual} supercuspidal representation of $\GL_d(F)$, and let $r \in \Z_{\geq 0}$. In this algorithm, we output the set
\[\Pi_{A, \, gp}(\pi_{sc}, \rho, r):= \Pi_{A, \, gp}(G_{n+rd}) \cap \Irr(X_{\rho}; \pi_{sc}).\]

\begin{enumerate}
    \item [Step 1.] If $r=0$, then output $\{\pi_{sc}\}$.
    \item [Step 2.]List the finite set of all tempered $L$-parameters $\phi$ of $G_{n+rd}$ such that
    \[ \lambda_{\phi}= \lambda_{\phi_{\pi_{sc}}} + \sum_{i=1}^{r} \rho\lvert \cdot \rvert^{y_i}+\rho\lvert \cdot \rvert^{-y_i}, \]
    for some $y_i \in \R$.       
    For each tempered representation $\sigma$ in the union of these tempered packets, compute its supercuspidal support by Theorem \ref{thm temp algo}. Let $\Omega_{temp}(\pi_{sc},\rho,r)$ be the subset of these tempered representations which are contained in $\Irr(X_{\rho}; \pi_{sc})$.
    
    \item [Step 3.] Assume that we have already computed $\Pi_{A,\, gp}(\pi_{sc},\rho, s)$ for $0 \leq s<r$. Initiate $\Omega_{non-temp}(\pi_{sc},\rho,r)$ to be the empty set.

    For any $0 \leq s <r$ and any $\pi^{\rho,-} \in \Pi_{A,\, gp}(\pi_{sc}, \rho, s)$, let (see Convention \ref{conv L-data})
    \[\pi_x:= \pi^{\rho,-}+ \{\Delta_{\rho}[ \half{r-s-1}-x, \half{1-r+s}-x ] \},\]
    and let $\epsilon\in \{1, \half{1}\}$ such that $\epsilon+ \half{1-r+s}+ \alpha_{\pi_{sc},\rho} \in \Z$.  For $x \in \{\alpha_{\pi_{sc},\rho}+ r+ \half{1-r+s}, \alpha_{\pi_{sc},\rho}+ r+ \half{1-r+s}-1,\ldots, \epsilon \}$, determine whether $\pi_x$ is of Arthur type by Algorithm \ref{algo Arthur type}. Append $\pi_x$ into $\Omega_{non-temp}(\pi_{sc},\rho,r)$ if it is of Arthur type.    
    \item [Step 4.]  Output $ \Pi_{A, \, gp}(\pi_{sc},\rho,r):= \Omega_{temp}(\pi_{sc},\rho,r) \cup \Omega_{non-temp}(\pi_{sc},\rho,r)$.
\end{enumerate}
\end{algo}

\section{\texorpdfstring{An explicit algorithm to determine $\Pi_{\overline{A}}(G_n)$}{}}\label{algorithm for A bar}

In this section, based on the algorithms given in \S \ref{sec new algorithm} and \S \ref{sec inductive approach}, 
we develop an explicit algorithm (Algorithm \ref{alg A bar}) to compute $\Pi_{\overline{A}}(G_n)$, thus generating the candidate set $\Pi_{\overline{A}}^{\lim}(G_n)$, 
for the split symplectic or odd special orthogonal group $G_n$. First, we introduce an algorithm for the set 
$\Pi_{ind}(G_n)$, which is the subset of $\Pi(G_n)$ consisting of representations that are irreducibly induced from an ``essentially" Arthur type representation of a Levi subgroup (see \eqref{eq def of cand}, also recalled below).

\subsection{\texorpdfstring{An algorithm to determine $\Pi_{ind}(G_n)$}{}}\label{sec algo cand}

Recall that $  \Pi_{ind}(G_n)$ is the subset of $\Pi(G_n)$ consisting of representations $\pi$ that can be realized as an irreducible induction of the form
\begin{align}\label{eq decomp Pi A cand}
   \bigtimes_{i \in I} u_{\rho_i}(a_i, b_i)\lvert\cdot \rvert^{x_i} \rtimes \pi_{A},
\end{align}
where $\pi_A$ is of good parity and Arthur type and $x_i \in \R$. 
Note that we have $\Pi_{\overline{A}}(G_n) \subseteq \Pi_{ind}(G_n)$ by Definition \ref{def closure A}. 
In this subsection, we explain how to determine whether a given representation $\pi$ lies in $\Pi_{ind}(G_n)$ or not, and in the affirmative case, output all expressions in the form of \eqref{eq decomp Pi A cand}.

First, we consider the case that $\pi$ is of good parity. 

\begin{algo}[$\Pi_{ind}(G_n) \cap \Pi_{gp}(G_n)$]\label{algo pi ind gp} Let $\pi$ be an irreducible representation of $G_n$ of good parity.
\begin{itemize}
    \item [Step 1.]we list a finite set $\Sigma(\pi)$ consisting of pairs $(\{u_{\rho_i}(a_i,b_i)\lvert \cdot\rvert^{x_i}\}_{i \in I}, \pi_A)$ such that $\pi_A$ is (of good parity and) of Arthur type and
\[ \phi_{\pi}=  \left(\bigoplus_{i \in I } \bigoplus_{r=0}^{b_i-1} \rho_i \lvert\cdot\rvert^{\half{1-b_i}+r + x_i} \otimes S_{a_j} + (\rho_i \lvert\cdot\rvert^{\half{1-b_i}+r + x_i} \otimes S_{a_i})^{\vee}\right)+ \phi_{\pi_{A}}. \]
Note that if \eqref{eq decomp Pi A cand} holds, then $(\{u_{\rho_i}(a_i,b_i)\lvert \cdot\rvert^{x_i}\}_{i \in I}, \pi_A) \in \Sigma(\pi)$  (see \cite[Lemma 10.3]{HLLS24} for example).
\item [Step 2.]For each pair in $\Sigma(\pi)$, we check whether the corresponding induction \eqref{eq decomp Pi A cand}
is irreducible or not by \cite[Theorem 1.2]{BS25} and \cite[Corollary 1.3]{Ato22d}. If it is irreducible, we can compute its $L$-data by taking consecutive derivatives of the induction in a systematic way, as considered in \cite{Lo}, to check whether the irreducible induction is equal to $\pi$ or not. 
Clearly, $\pi$ lies in $\Pi_{ind}(G_n)$ if and only if there exists a pair in $\Sigma(\pi)$ such that the induction is irreducible and equals $\pi$.
\end{itemize}
\end{algo}

\begin{remark}
Let $(\{u_{\rho_i}(a_i,b_i)\lvert \cdot\rvert^{x_i}\}_{i \in I}, \pi_A) \in \Sigma(\pi)$ and set 
\[ \Pi_{\epsilon}:=   \bigtimes_{i \in I} u_{\rho_i}(a_i, b_i)\lvert\cdot \rvert^{x_i+\epsilon_i} \rtimes \pi_{A}, \]
for $\epsilon= (\epsilon_i) \in \BR^{|I|}$. We explain in this remark how to check whether $\pi= \Pi_{0}$ or not without using \cite{Lo}, under the assumption that $\pi$ is unitary and of good parity. 

Applying \cite[Theorem 1.2]{BS25} and \cite[Corollary 1.3]{Ato22d}, we can check whether $\Pi_{0}$ is irreducible or not. We assume that it is irreducible from now on. Then there exists a small neighborhood $U\subset \BR^n$ of $0$ such that $\Pi_{\epsilon}$ is irreducible for all $\epsilon \in U$. Note that $\{\Pi_{\epsilon}\}_{\epsilon \in U}$ forms a continuous family of Hermitian representations since they are all weakly real. 

We claim that if $\Pi_{0}=\pi$, then the induced representation
\[ \Pi_{I'}:=  \bigtimes_{i \in I'} u_{\rho_i}(a_i, b_i)\lvert\cdot \rvert^{x_i} \rtimes \pi_{A} \]
must be of Arthur type for any $I' \subseteq I$. Indeed, suppose that $\Pi_{0}=\pi$, which is unitary by assumption. Then $\{\Pi_{\epsilon}\}_{\epsilon \in U}$ is a continuous family of unitary representations. For any $I' \subseteq I$, we may take an $\epsilon'=(\epsilon_i')\in U$ such that $\epsilon_i'=0$ for $i \in I'$ and $\epsilon_i' \not\in \Z$ for $i \in I\setminus I'$. The unitarity of $\Pi_{\epsilon'}$ and \cite[Theorem 3.1]{AM25} then imply that $\Pi_{I'}$ is of Arthur type.

Now we state the algorithm to check whether $\Pi_0=\pi$. Take any filtration $\emptyset= I_0 \subset I_1 \subset \cdots \subset I_{|I|}=I$ with $|I_{k+1} \setminus I_{k}|=1$. Suppose that we have computed the $L$-data of $\Pi_{I_k}$ and have verified that $\Pi_{I_{k}}$ is of Arthur type. Then the $L$-data of $\Pi_{I_{k+1}}$ can be computed from $\Pi_{I_k}$ by \cite[Corollary 1.3]{Ato22d}, and we can check whether $\Pi_{I_{k+1}}$ is of Arthur type by Algorithm \ref{algo Arthur type}. If $\Pi_{I_{k}}$ is not of Arthur type for some $k$, then $\Pi_0 \neq \pi$. Otherwise, we have computed the $L$-data of $\Pi_{I_{|I|}}= \Pi_{0}$, and we can check whether $\Pi_0=\pi$ directly.
\end{remark}

Next, we consider general representations $\pi$ that are not necessarily of good parity. The following lemma shows that if $\pi \in  \Pi_{ind}(G_n)$, then the expression of $\pi$ in the form of \eqref{eq decomp Pi A cand} is essentially unique modulo the good parity parts.

{ 
\begin{lemma}\label{lemma component of Pi A bar reduction}
For $i=1,2$, let
\begin{align}\label{eq pi_i algorithm no repetition}
    \pi_i=  \bigtimes_{j \in I_i} u_{\rho_j}(a_j, b_j)\lvert\cdot \rvert^{x_j} \rtimes \pi_{i,0}
\end{align}
be an irreducible induction, where $\pi_{i,0}$ is of good parity (not necessarily of Arthur type). Decompose $I_{i}= I_{i,gp}\sqcup I_{i,ngp}$, where $ I_{i,gp}$ is the maximal subset of $I_i$ such that
\[ \pi_{i,gp}:= \bigtimes_{j \in I_{i,gp}} u_{\rho_j}(a_j, b_j)\lvert\cdot \rvert^{x_j} \rtimes \pi_{i,0} \]
is of good parity. Then $\pi_1\cong \pi_2$ if and only if
\begin{enumerate}
    \item [(a)] $\pi_{1,gp}\cong \pi_{2,gp}$, and,
    \item [(b)] there exists a bijection $\iota: I_{1,ngp} \to I_{2,ngp}$ such that $u_{\rho_{j}}(a_j, b_j)\lvert \cdot\rvert^{x_j} \cong u_{\rho_{\iota(j)}}(a_{\iota(j)}, b_{\iota(j)})\lvert \cdot\rvert^{x_{\iota(j)}}$ or $(u_{\rho_{j}}(a_j, b_j)\lvert \cdot\rvert^{x_j})^{\vee} \cong u_{\rho_{\iota(j)}}(a_{\iota(j)}, b_{\iota(j)})\lvert \cdot\rvert^{x_{\iota(j)}}$.
\end{enumerate}
\end{lemma}
\begin{proof}
    If both Part (a) and Part (b) hold, then $\pi_1=\pi_2$ in the Grothendieck group. Since $\pi_i$ is irreducible, we have $\pi_1 \cong \pi_2$.

    Conversely, suppose that $\pi_1 \cong \pi_2$. Then $\pi_{i,gp}= X_{i,gp}(\pi_i)$ under the Jantzen decomposition (Theorem \ref{thm Jantzen}), where $X_{i,gp}=\{\rho\lvert\cdot\rvert^x \in \mathcal{C} \ | \ \rho\lvert\cdot\rvert^x \rtimes \pi_{i,0} \text{ is of good parity}, x \in \mathbb{R}\}$. Thus, we have $\pi_{1,gp}\cong \pi_{2, gp}$, which proves Part (a). 
    
    Now we prove Part (b) by showing that the set $\{ (u_{\rho_{j}}(a_j, b_j)\lvert \cdot\rvert^{x_j}, (u_{\rho_{j}}(a_j, b_j)\lvert \cdot\rvert^{x_j})^{\vee} )\}_{j \in I_{1,ngp}}$ can be computed from the $L$-parameter of $\pi_1$ (and similar for $\pi_2$). We assume that  $|I_{1,ngp}|>0$ in the rest of the proof.

    We first do some reductions. First, after replacing $\pi_{1,0}$ by $\pi_{1,gp}$ and $I_{1}$ by $I_{1,ngp}$, we may assume that $I_{1}=I_{1,ngp}$. Next, fix a $\rho \in \mathcal{C}^u$ such that $\rho \cong \rho_j$ for some $j \in I_{1, ngp}$ and set $X_{\rho, \rho^{\vee}}:= \{\rho\lvert \cdot \rvert^{x} \ | \ x \in \R\}\cup \{\rho^{\vee}\lvert \cdot \rvert^{x} \ | \ x \in \R\}$, which is self-contragredient. By replacing $\pi_i$ with $X_{\rho, \rho^{\vee}}(\pi_i)$ via the Jantzen decomposition, we may assume that $I_{i}=I_{i,ngp}$ and for any $j \in I_{1,ngp}$, either $ \rho_j \cong \rho$ or $ \rho_j \cong \rho^{\vee}$. Finally, replacing $u_{\rho_{j}}(a_j, b_j)\lvert \cdot\rvert^{x_j}$ with its contragredient if necessary, we may further assume that $x_j \leq 0$ for any $j \in I_{1,ngp}$. 
    
    Now we fix a $k \in I_{1, ngp}$ such that the following conditions hold.
    \begin{enumerate}
        \item [(i)] $\half{-a_k-b_k}+1+x_{k} =\min\{\half{-a_j-b_j}+1+x_{j} \}_{j \in I_{1, ngp}}$.
        \item [(ii)] $\half{a_k+b_k}-1+x_{k} =\max\{\half{a_j+b_j}-1+x_{j}\}_{j \in I_{1,ngp}'}$, where $I_{1,ngp}'=\{j \in I_{1, ngp}\ | \  \half{-a_j-b_j}+1+x_{j}=\half{-a_k-b_k}+1+x_{k}  \}$.
        \item [(iii)] $a_k= \min\{ a_j \}_{j \in I_{1, ngp}''}$, where $I_{1,ngp}''=\{j \in I_{1, ngp}',\ \half{a_k+b_k}-1+x_{k} =\half{a_j+b_j}-1+x_{j}  \}.$
        \item [(iv)] Let $I_{1, ngp}''':= \{ j \in I_{1, ngp}''\ | \ a_j=a_k\} $. If $ \{\rho_j\}_{j\in I_{1,ngp}'''}= \{\rho, \rho^{\vee}\}$ as a set, then we require that $\rho_k= \rho$.
    \end{enumerate}
    Clearly, these four conditions uniquely determine the quadruple $(a_k, b_k, x_k, \rho_k)$.

    Write the $L$-parameter of $\pi_1$ as 
    \begin{align}\label{eq phi_1 algorithm 1}
        \phi_1= \bigoplus_{j \in J } \left(\rho_j\lvert\cdot\rvert^{y_j} \otimes S_{c_j} + (\rho_j\lvert\cdot\rvert^{y_j} \otimes S_{c_j})^{\vee} \right)+ \phi_{\pi_{1,0}}, 
    \end{align}
    where $y_j \leq 0$ for any $j \in J$. Recall that $\rho_j \cong \rho$ or $\rho_j \cong \rho^{\vee}$ by our assumption that $\pi_1= X_{\rho,\rho^{\vee}}(\pi_1)$. For each $j \in J$, we set $\alpha_j:= c_j$, define $\beta_j$ to be the maximal integer such that 
        $\phi_1$ contains 
        \[  \bigoplus_{r=0}^{\beta_j-1} \rho_j \lvert \cdot \rvert^{y_j+ r} \otimes S_{c_j} + (\rho_j \lvert \cdot \rvert^{y_j+ r}\otimes S_{c_j})^{\vee},\]
        and let $\xi_j:= \half{1-c_j}+y_j+ \half{\alpha_j+ \beta_j}-1$. We fix an index $l \in J$ such that 
        \begin{enumerate}
            \item [(i$'$)] $\half{-\alpha_l- \beta_l}+1+ \xi_l= \half{1-c_l}+ y_l= \min\{\half{1-c_j}+ y_j\}_{j \in J}$.
            \item [(ii$'$)] $ \half{\alpha_l+ \beta_l}-1+ \xi_l= \max\{\half{\alpha_j+ \beta_j}-1+ \xi_j \}_{j \in J'}$, where $J'= \{j \in J\ | \ \half{-\alpha_j- \beta_j}+1+ \xi_j= \half{-\alpha_l- \beta_l}+1+ \xi_l \}$.
            \item [(iii$'$)] $ \alpha_l= \min \{\alpha_j\}_{j \in J''}$, where $J''= \{j \in J'\ | \ \half{\alpha_j+ \beta_j}-1+ \xi_j=\half{\alpha_l+ \beta_l}-1+ \xi_l\}$.
            \item [(iv$'$)] Let $J''':= \{j \in J''\ | \ \alpha_j =\alpha_l\}.$ If $\{ \rho_j\}_{j \in J'''}= \{\rho, \rho^{\vee}\}$ as a set, then we require that $\rho_l= \rho$.
        \end{enumerate}
        Again, these four conditions uniquely determine the quadruple $(\alpha_l, \beta_l, \xi_l, \rho_l)$. 
        
        We claim that $(\alpha_l, \beta_l, \xi_l, \rho_l)= (a_k, b_k, x_k, \rho_k)$.
Indeed, from the irreducible induction \eqref{eq pi_i algorithm no repetition} and our assumption that $I_{1,ngp}= I_{1}$, we have
        \begin{align}\label{eq phi_1 algorithm 2}
            \phi_1= \left(\bigoplus_{j \in I_{1,ngp}} \bigoplus_{r=0}^{b_j-1} \rho_j \lvert\cdot\rvert^{\half{1-b_j}+r + x_j} \otimes S_{a_j} + (\rho_j \lvert\cdot\rvert^{\half{1-b_j}+r + x_j} \otimes S_{a_j})^{\vee}\right)+ \phi_{\pi_{1,0}}. 
        \end{align}
        Here, the $j$-th term comes from $u_{\rho_j}(a_j, b_j)\lvert\cdot \rvert^{x_j}.$
        Take any $\ell \in J$ such that $\rho_{\ell}\lvert\cdot\rvert^{y_\ell}\otimes S_{c_\ell}= \rho_k \lvert \cdot\rvert^{\half{1-b_k}+x_k} \otimes S_{a_k}$, { where $k$ is the fixed index above}. We check that $(\alpha_{\ell}, \beta_\ell, \xi_\ell, \rho_\ell)= (a_k, b_k, x_k, \rho_k)$.

        First, it is immediate that $\rho_\ell=\rho_k,$ $c_\ell=\alpha_\ell=a_k,$ and $y_\ell=\half{1-b_k}+x_k.$ Next, we check that $\beta_\ell= b_k$. From \eqref{eq phi_1 algorithm 2} and the definition of $\beta_\ell$, we have $\beta_\ell \geq b_k$. Suppose the contrary that $\beta_\ell> b_k$. Then there exists an index $j \in I_{1,ngp}$ such that $\rho_\ell\lvert\cdot \rvert^{y_\ell+b_k}\otimes S_{c_\ell}+ (\rho_\ell\lvert\cdot \rvert^{y_\ell+b_k}\otimes S_{c_\ell})^{\vee}$ comes from $ u_{\rho_j}(a_j,b_j)\lvert\cdot \rvert^{x_j}$.         
        More explicitly,
        we have $\rho_j \cong \rho_\ell$, $a_j= a_k$, and there exists an $0 \leq r\leq b_j-1 $ such that
        \begin{align}\label{eq phi_1 algorithm 3}
            \half{b_k-1}+x_k+1=y_{\ell}+b_k= \half{1-b_j}+r+x_j.
        \end{align}
        Since $a_j=a_k$, we obtain that
        \begin{align}\label{eq phi_1 algorithm 4}
            \half{1-b_k} +x_k < \half{1-b_j}+x_j \leq \half{b_k-1}+x_k +1 \leq \half{b_j-1}+x_j,
        \end{align}
        where the second and third inequalities follow from \eqref{eq phi_1 algorithm 3}, and the first inequality follows from conditions (i) and (ii) for the index $k$. 
        Now, the sequence of inequalities \eqref{eq phi_1 algorithm 4} and $a_j=a_k$ imply that the parabolic induction
        \[ u_{\rho_k}( a_k, b_k) \lvert\cdot \rvert^{x_k} \times u_{\rho_j}( a_j, b_j) \lvert\cdot \rvert^{x_j} \]
        is reducible by \cite[Theorem 1.1]{Tad14}, contradicting to the irreducibility of \eqref{eq pi_i algorithm no repetition}.   
        Thus, we conclude that $\beta_\ell=b_k$. This also implies that $(\alpha_\ell, \beta_\ell, \xi_\ell, \rho_\ell)= (a_k, b_k, x_k, \rho_k)$. 
        
        Finally, since the index $k \in I_{1,ngp}$ satisfies conditions (i)-(iv), comparing \eqref{eq phi_1 algorithm 1} and \eqref{eq phi_1 algorithm 2}, we see that
        the index $\ell\in J $ satisfies conditions (i$'$)-(iv$'$); hence $(\alpha_l, \beta_\ell, \xi_\ell, \rho_\ell)=(\alpha_l, \beta_l, \xi_l, \rho_l)$. This completes the verification of the claim.

        To summarize, from the $L$-parameter of $\pi_1$, we can specify a pair of generalized Speh representations $(u_{\rho_k}(a_k, b_k)\lvert \cdot \rvert^{x_k}, (u_{\rho_k}(a_k, b_k)\lvert \cdot \rvert^{x_k})^{\vee})$ such that at least one of them is a factor in \eqref{eq pi_i algorithm no repetition}. 
        To finish the proof, we define
        \[ \pi_{1}^{-}:=\bigtimes_{j \in I_{i,ngp}\setminus \{k\}} u_{\rho_j}(a_j, b_j)\lvert\cdot \rvert^{x_j} \rtimes \pi_{1,0}, \]
        and repeatedly apply the process on $\pi_1^{-}$. An inductive argument then implies that we can compute the set $\{ (u_{\rho_{j}}(a_j, b_j)\lvert \cdot\rvert^{x_j}, (u_{\rho_{j}}(a_j, b_j)\lvert \cdot\rvert^{x_j})^{\vee} )\}_{j \in I_{1,ngp}}$ from the $L$-parameter of $\pi_1$. This completes the proof of Part (b) and the lemma.
\end{proof}
}

The above proof provides an algorithm to determine whether a general $\pi$  lies in $  \Pi_{ind}(G_n)$, which we record as follows.

\begin{algo}[$\Pi_{ind}(G_n)$] \label{remark reduction to gp}
Let $\pi$ be an arbitrary irreducible representation of a classical group $G_n$. Suppose that the good parity part $\pi_{gp}= X_{gp}(\pi)$ is of Arthur type. Then we may uniquely write down a decomposition
\[ \phi_{\pi}=  \left(\bigoplus_{i \in I } \bigoplus_{r=0}^{b_i-1} \rho_i \lvert\cdot\rvert^{\half{1-b_i}+r + x_j} \otimes S_{a_i} + (\rho_i \lvert\cdot\rvert^{\half{1-b_i}+r + x_j} \otimes S_{a_i})^{\vee}\right)+ \phi_{\pi_{gp}} \]
by repeatedly looking for the index $l$ satisfying conditions $(\emph{i}')-(\emph{iv}')$ in the proof of Part (b) above. Then $\pi$ is in $\Pi_{ind}(G_n)$ if and only if the following conditions hold:
\begin{itemize}
    \item [(1)] the good parity part $\pi_{gp} $ lies in {$\Pi_{ind}(G_m)$} for some $m \leq n$, and,
    \item [(2)] the induction
\[ \bigtimes_{i \in I} u_{\rho_i}(a_i, b_i) \lvert\cdot \rvert^{x_i} \rtimes \pi_{gp} \]
is irreducible.
\end{itemize}
 Condition \emph{(1)} can be checked by Algorithm \ref{algo pi ind gp}, while Condition \emph{(2)} can be checked by \cite[Theorem 1.2]{BS25}, \cite[Theorem 1.1]{Tad14}, and \cite{Ato22d} by writing $\pi_{gp}$ in the form \eqref{eq decomp Pi A cand}.
\end{algo}

\subsection{\texorpdfstring{An algorithm to determine $\Pi_{\overline{A}}(G_n)$}{}}

In this subsection, given any $\rho \in \mathcal{C}^{sd}$ and $\pi_{sc} \in \mathcal{C}^{cl}$, we develop an algorithm to exhaust the subset of $\Pi_{\overline{A}}(G_n)$ contained in $\Irr(X_\rho; \pi_{sc})$ (see Definition \ref{def Jantzen decomposition}). This can be extended to the general case and exhaust the whole set $\Pi_{\overline{A}}(G_n)$ by Corollary \ref{cor preseveration for Pi A bar}, thus generating the candidate set $\Pi_{\overline{A}}^{\lim}(G_n)$ (abstractly). The algorithm makes use of the unitary dual of $\GL_d(F)$ (\cite{Tad86}), which is a factor of proper Levi subgroups of $G_n$. Although this does not extend to all exceptional groups, we expect that the same algorithm should work for other classical groups.
In the following, by abusing notation, we regard $\R^0=\{0\}$.

\begin{algo}\label{alg A bar}
   Let $\pi_{sc}$ be an irreducible supercuspidal representation of $G_n$, $\rho$ be an irreducible selfdual supercuspidal representation of $\GL_d(F)$, and let $r \in \Z_{\geq 0}$. In this algorithm, we output the set
\[ \Pi_{\overline{A}}(\pi_{sc}, \rho, r):= \Pi_{\overline{A}}(G_{n+rd}) \cap \Irr(X_{\rho}; \pi_{sc}).\]
Write $\alpha=\alpha_{\pi_{sc},\rho}$ for short.
\begin{enumerate}
    \item [Step 1.] Compute $\Pi_{A, \, gp}(\pi_{sc},\rho,s)$ for $0 \leq s \leq r$ by Algorithm \ref{alg Pi A}. Let $\Omega_{A,gp}:= \cup_{0 \leq s \leq r} \Pi_{A, \, gp}(\pi_{sc},\rho,s).$ 
    \item [Step 2.] For $ 0 \leq s \leq r' \leq r$, initiate $R(\pi_{sc},\rho,r',s)$ to be the empty set. Repeat the following steps for each $\pi_A \in \Pi_{A, \, gp}(\pi_{sc},\rho,s)$: 
    \begin{enumerate}
        \item [(2-1)] Let $\Psi(r'-s)$ be the collection of ordered tuples of pairs $\psi=((a_1,b_1),\ldots, (a_{l(\psi)}, b_{l(\psi)}))\in(\mathbb{Z}_{\geq 1}^2)^{l(\psi)}$, ${l(\psi)} \in \mathbb{Z}_{>0}$, such that $\sum_{i=1}^{l(\psi)} a_ib_i=r'-s$. By convention, let $\Psi(0):=\{\emptyset\}$. 
        
        \item [(2-2)] Suppose that $r'-s>0$. For each $\psi \in \Psi(r'-s)$, let $\mathcal{H}(\psi, \pi_A)$ denote the set consisting of the following reducibility hyperplanes in $\R^{{l(\psi)}}$: 
        \begin{itemize}
           
            \item $\{x_i= t \ | \ u_{\rho}(a_i,b_i)\lvert\cdot\rvert^{t} \rtimes \pi_A \text{ is reducible}\}$,
            \item $\{x_i\pm x_j = t\ | \ u_{\rho}(a_i,b_i)\lvert\cdot\rvert^{t} \times u_{\rho}(a_j,b_j) \text{ is reducible}, 1 \leq i < j \leq {l(\psi)}\}$.
        \end{itemize}
       
        \item [(2-3)] If $r'-s>0$, let $R(\psi,\pi_A)$ denote the (finite) set of connected components of $\R^{l(\psi)} \setminus (\cup_{H \in \mathcal{H}(\psi, \pi_A)} H)$.        
        For each $C \in R(\psi,\pi_A)$, append the triple $(C, \psi, \pi_A)$ into $R(\pi_{sc},\rho,r',s)$. If $r'-s=0$, then append $\{ (\R^0, \emptyset , \pi_A) \}$ into $R(\pi_{sc},\rho,r',s)$. 
    \end{enumerate}
    \item [Step 3.] Let $R:= \cup_{0 \leq s \leq r' \leq r} R(\pi_{sc},\rho,r',s) \sqcup \{-1\}$. Define an equivalence relation $\sim$ on $R$ as follows. Let $(C,\psi, \pi_A) \in R(\pi_{sc},\rho,r',s)$ where $C \subseteq \R^{l(\psi)}$.
    \begin{itemize}
    \item [(3-1)] Suppose that $C$ is unbounded. Then we define $ (C,\psi, \pi_A) \sim -1$.
    \item [(3-2)]Suppose that $C \cap \{x_i=0\} \neq \emptyset$. Define $\psi^{-}$ by removing $(a_i,b_i)$ from $\psi$. Take any point $\underline{y} \in C \cap \{x_i=0\}$ and define $\underline{y}^{-}\in \R^{l(\psi)-1}$ by removing the $i$-th coordinate $($if ${l(\psi)}=1$, then set $\underline{y}^-=0$$)$. Let $(C^-, \psi^{-}, \pi_A)$ be the unique element in $R$ such that $\underline{y}^- \in C^- $. Then we define $(C,\psi,\pi_A) \sim (C^-, \psi^{-}, \pi_A)$.
    \item [(3-3)] Suppose that $C \cap \{x_i=t\} \neq \emptyset$ for some $t\in (\alpha+ \half{a_i+b_i})+\Z$. Let $\pi_{A}^+:= u_{\rho}(a_i,b_i)\lvert\cdot\rvert^{t} \rtimes \pi_A$, which is irreducible and of good parity. If $\pi_A^+ $ is not in $\Omega_{A,gp}$, then define $(C,\psi, \pi_A) \sim -1$. If $\pi_A^+ \in \Omega_{A,gp}$, then  define $\psi^{-}$ and take a point $\underline{y}^- \in \R^{{l(\psi)}-1}$ as in the previous case. Let $(C^-, \psi^{-}, \pi_A^+)$ be the unique element in $R$ such that $\underline{y}^- \in C^- $. Then we define $(C,\psi,\pi_A) \sim (C^-, \psi^{-}, \pi_A^+)$.
    \item [(3-4)] Suppose that $(a_i,b_i)=(a_j,b_j)$ with $i \neq j$ and $C \cap \{x_i={ \pm}x_j\}\neq \emptyset$. Define $\psi^-$ by removing both $(a_i,b_i)$ and $(a_j,b_j)$ from $\psi$. Take any point $\underline{y}=(y_1,\ldots, y_{l}) \in C  \cap \{x_i={ \pm}x_j\}$. If $|y_i|>\half{1}$, then define $(C,\psi, \pi_A) \sim -1$. If $|y_i|< 1/2$, then define $\underline{y}^{-} \in \R^{{l(\psi)}-2}$ by removing the $i$-th and $j$-th coordinates $($if ${l(\psi)}=2$, then set $\underline{y}^-=0$$)$. Let $(C^-, \psi^{-}, \pi_A)$ be the unique element in $R$ such that $\underline{y}^- \in C^- $. Then we define $(C,\psi,\pi_A) \sim (C^-, \psi^{-}, \pi_A)$.
    \end{itemize}
     Let $R_{\overline{A}}$ be the collection of $(C,\psi, \pi_A) \in R$ such that $(C, \psi, \pi_A) \sim ( \R^0, \emptyset, \pi_A')$ for some $\pi_A' \in \Omega_{A,gp}$. 
    \item [Step 4.] For each $(C,\psi=((a_1,b_1),\ldots,(a_{l(\psi)},b_{l(\psi)})),\pi_A) \in R$, let 
   {\[ \Pi ({C}, \psi, \pi_A):= \left\{ \bigtimes_{i=1}^{l(\psi)} u_{\rho}(a_i,b_i)\lvert\cdot\rvert^{y_i} \rtimes \pi_A \ \middle| \ (y_1,\ldots, y_{l(\psi)})\in {C}\right\}, \]} 
    if $l(\psi)>0,$ and $\Pi(\BR^0, \emptyset , \pi_A):= \{\pi_A\}$. Then 
    \[ \Pi_{\overline{A}}(\pi_{sc}, \rho, r)= \left(\bigcup_{(C, \psi, \pi_A) \in R_{\overline{A}}} {\Pi ({C}, \psi, \pi_A)}\right) \cap \Pi(G_{n+rd}).\]
\end{enumerate}
\end{algo}

The algorithm has been implemented using Sagemath. We can describe the set $\Pi_{\overline{A}}^{\lim}(\pi_{sc}, \rho, r):= \Pi_{\overline{A}}^{\lim}(G_{n+rd})\cap \Irr(X_{\rho};\pi_{sc})$ abstractly from $\Pi_{\overline{A}}(\pi_{sc}, \rho, r)$ as follows. For each $(C,\psi,\pi_A)\in R_{\overline{A}}$, define
\[ \Pi^{\lim} ({C}, \psi, \pi_A):= \left\{ \pi \ \middle| \ \pi \leq \bigtimes_{i=1}^{l(\psi)} u_{\rho}(a_i,b_i)\lvert\cdot\rvert^{y_i} \rtimes \pi_A \text{ for some } (y_1,\ldots, y_{l(\psi)})\in \overline{C}\right\}. \]
Then by definition,
\[  \Pi_{\overline{A}}^{\lim}(\pi_{sc}, \rho, r)= \left(\bigcup_{(C, \psi, \pi_A) \in R_{\overline{A}}} {\Pi^{\lim} ({C}, \psi, \pi_A)}\right) \cap \Pi(G_{n+rd}).\]
However, for $ (y_1,\ldots, y_{l(\psi)})\in \overline{C} \setminus C $, the induction $\bigtimes_{i=1}^{l(\psi)} u_{\rho}(a_i,b_i)\lvert\cdot\rvert^{y_i} \rtimes \pi_A$ is reducible. There is no explicit algorithm on computing the irreducible subquotients at this moment, except for some small rank examples via certain case-by-case argument.
More remarks on the algorithm are as follows.

\begin{remark}\label{rmk A+ algo}
Let $(C,\psi,\pi_A) \in R$.
    \begin{enumerate}
    \item  In Step 2, we may replace $R(\psi,\pi_A)$ by the orbit under an appropriate Weyl action to increase the efficiency.  
    \item In Step 3, if $C$ is bounded, then $C \cap H$, where $H$ is a hyperplane, is the interior of a convex hull generated by a set of vertices. One can take the point $\underline{y}$ to be the center of these vertices. 
        \item It follows from \cite[Theorem 1.2]{BS25} that Step \emph{(2-2)}  gives all reducibility hyperplane in $\R^{l(\psi)}$. As a consequence, any representation in $\Pi(C,\psi,\pi_A)$ is irreducible. Also, the algorithms to list these reducibility hyperplane are given in \cite{Ato22d} and \cite[Theorem 1.1]{Tad14}.

            \item Suppose that $ (C,\psi,\pi_A) \in R$ and is equivalent to $-1$. Then any representation in $\Pi(C,\psi,\pi_A) $ is not unitary. Indeed, Case \emph{(3-1)} follows from \cite[Theorem 2.5]{Tad88}, Cases \emph{(3-3)}, and \emph{(3-4)} follow from unitary parabolic reduction \emph{(}see \cite[\S 2.15]{Tad23}\emph{)}, the unitary dual of $\GL_n(F)$ in \cite{Tad86} and a recent result of Atobe and M{\'i}nguez \emph{(}\cite[Theorem 3.1]{AM25}\emph{)}.
        
        \item {Conjecture \ref{conj A+ intro} predicts that every representation outside $\Pi_{\overline{A}}^{\lim}(\pi_{sc}, \rho, r)$ is not unitary}. It is natural to ask whether every $(C, \psi, \pi_A) \in R \setminus R_{\overline{A}}$  is equivalent to $-1$. However, this is not always the case, as shown in Example \ref{ex A+ unit ind}. For example, the non-unitary triangles in Figure \ref{Figure components supported in Pi A bar corank 2} below (a zoomed-out version of Figure \ref{Figure corank 2}) are not equivalent to $-1$.
        
        If we add an extra criterion to check whether all irreducible representations in the socle and cosocle at the vertices of this kind of regions are of Arthur type \emph{(}computable by \cite{Ato22d}\emph{)}, then we can exclude these regions in the examples of corank 2 and 3 with small $\alpha$. 
    \end{enumerate}
\end{remark}

\subsection{Applying Algorithm \ref{alg A bar} to Example \ref{ex A+ unit ind}}

In this subsection, we illustrate the above algorithm on Example \ref{ex A+ unit ind}. Note that we do not perform the entire algorithm here, but enough cases to demonstrate the main ideas of the algorithm.
First, we recall the setting. We assume that $\alpha_{\rho,\sigma}=\half{3}$ and consider the parabolic induction
$\Pi_{x,y}:= \rho\vert\cdot\rvert^{x} \times \rho\vert\cdot\rvert^{y} \rtimes \sigma$, with $(x,y) \in \R^2.$ Note that determining the unitary subquotients for $x,y\geq 0$ is sufficient to determine the unitary subquotients for all $(x,y)\in\mathbb{R}^2$ by Remark \ref{rmk A+ algo}(1). There is also symmetry about the line $y=x$ that we could also exploit. According to Algorithm \ref{alg A bar}, we have $r=2$. 

Following Step 1 of the algorithm, we have
$\Pi_{A,\, gp}(\sigma,\rho,0)=\{\sigma\}.$
Step 2 of Algorithm \ref{alg A bar} proceeds by iterating over $r'=0,1,2$. From Step (2-1), we have that $\Psi(0)=\emptyset$ when $r'=0$, $\Psi(1)=\{((1,1))\}$ when $r'=1$, and $\Psi(2)=\{((1,1),(1,1)),((2,1)),((1,2))\}$ when $r'=2.$ 

For $r'=0,$ we add $(\R^0,\emptyset,\sigma)$ to $R(\sigma,\rho,0,0)$ by Step (2-3) of the algorithm. 
For $r'=1,$ we have $\psi_1=((1,1))$ and exactly two reducibility hyperplanes in $\mathcal{H}(\psi_1,\sigma)$ given by the point $x=\pm\frac{3}{2}.$ This is represented in the following picture
\begin{figure}[H]
\begin{center}
\scalebox{1}{\small
\begin{tikzpicture}[>=stealth]
\draw[solid, {Circle[open]}-{Circle[open]}] (-3/2,0) --  (3/2,0);
\draw[solid, ->] (3/2,0) --  (3,0);
\draw[solid, <-] (-3,0) --  (-3/2,0);
\node at (-3.25,0)[black]{$B_1$} ;
\node at (0,.5)[black]{$B_2$} ;
\node at (3.25,0)[black]{$B_3$} ;
\node at (1.5,.5)[black]{$\half{3}$} ;
\node at (-1.5,.5)[black]{$-\half{3}$} ;
\end{tikzpicture}}
\end{center}
\caption{Reducibility hyperplanes of $\rho\vert\cdot\vert^x\rtimes\sigma$}\label{Figure components corank 1}
\end{figure}
The reducibility hyperplanes carve $\mathbb{R}$ into 3 connected components $B_1, B_2, B_3$.  
Thus we add $(B_1,\psi_1,\sigma),$ $(B_2,\psi_1,\sigma),$ and $(B_3,\psi_1,\sigma)$ to $R(\sigma,\rho,1,0)$ by Step (2-3) of the algorithm. Note that only one component is bounded. By Step (3-1), we have that $(B_1,\psi_1,\sigma)\sim (B_3,\psi_1,\sigma)\sim -1.$ On the other hand, $B_2\cap\{x=0\}\neq\emptyset$. Thus, by Step (3-2), we have that $(B_2,\psi_1,\sigma)\sim (\R^0,\emptyset,\sigma).$ 

When $r'=2$ and $\psi_2=((1,1),(1,1))$ the reducibility hyperplanes in $\mathcal{H}(\psi_2,\sigma)$ are denoted by the dashed lines in the following figure (which is a zoomed-out version of Figure \ref{Figure corank 2}).
\begin{figure}[H]
\begin{center}
\scalebox{0.8}{\small
\begin{tikzpicture}[>=stealth]
\filldraw[fill=gray!30, dashed] (0,1) -- (1,0) -- (0,-1) -- (-1,0) -- cycle;
\filldraw[fill=gray!30, dashed] (0,1) -- (.5,1.5) -- (-.5,1.5) -- (0,1) -- cycle;
\filldraw[fill=gray!30, dashed] (1,0) -- (1.5,.5) -- (1.5,-.5) -- (1,0) -- cycle;
\filldraw[fill=gray!30, dashed] (-1,0) -- (-1.5,.5) -- (-1.5,-.5) -- (-1,0) -- cycle;
\filldraw[fill=gray!30, dashed] (0,-1) -- (-.5,-1.5) -- (.5,-1.5) -- (0,-1) -- cycle;
\node at (0.25,0.25) {$C_0$};
\node at (0.25,1.75) {$C_1$};
\node at (1.75,0.25) {$C_1'$};
\node at (-1.75,0.25) {$C_1'''$};
\node at (0.25,-1.75) {$C_1''$};
\draw[dashed] (-4,-3) --  (3,4) node[above] {$ y=x+1 $};
\draw[dashed] (-3,-4) --  (4,3) node[above] {$ y=x-1 $};
\draw[dashed] (-4,1.5) --  (4,1.5) node[right] {$ y=\half{3} $};
\draw[dashed] (1.5,-4) --  (1.5,4) node[left] {$ x=\half{3} $};
\draw[dashed] (-4,-1.5) --  (4,-1.5) node[right] {$ y=-\half{3} $};
\draw[dashed] (-1.5,-4) --  (-1.5,4) node[left] {$ x=-\half{3} $};
\draw[dashed] (-3,4) --  (4,-3) node[right] {$ x+y=1$};
\draw[dashed] (3,-4) --  (-4,3) node[left] {$ -x-y=1$};
\draw[->] (0,-4) --  (0,4) node[above] {$ y $};
\draw[->] (-4,0) --  (4,0) node[right] {$ x $};
\end{tikzpicture}}
\end{center}
\caption{2-dimensional regions of unitarity for $\Pi_{x,y}$}\label{Figure components supported in Pi A bar corank 2}
\end{figure}

It follows that $\mathbb{R}^2\setminus \mathcal{H}(\psi_2,\sigma)$ has 33 distinct connected components which contribute 33 elements of the form $(C,\psi_2,\sigma)$ to $R(\sigma,\rho,2,0).$ Note that only 17 components are bounded. Of the 17 components, only the bounded components intersecting the axes are equivalent to $(\R^0,\emptyset,\sigma).$ These components are shown above. { Indeed, by Step (3-2), we see that $C_0 \sim (B_2 , ((1,1)), \sigma ) \sim C_1'$, or $C_0 \sim (\R^0, \emptyset, \sigma)$ by Step (3-4). Other regions follow similarly or by symmetry.}

We explain {why the other bounded regions are not in $R_{\overline{A}}$.}
Consider the region containing the point $(1,1).$ This region is bounded and intersects the line $y=x$. By Step (3-4), we have that this region is equivalent to $-1$ since $1>\half{1}$. The other regions containing the points $(\pm1,\pm1)$ are handled similarly (or by symmetry).
Consider the region containing the point $(1,1.75).$ This region is bounded, but does not intersect the lines $y=0,$ $x=0,$ $y=x$, $y=t$ or $x=t$ for $t\in \half{1}+\Z.$ Thus, from Step 3, we see that this region is only equivalent to itself; hence it is not in  $R_{\overline{A}}$. 
At this point, Algorithm \ref{alg A bar} shows that \[\Pi_{x,y}\in\Pi_{\overline{A}}(\sigma,\rho,2),\] for any $(x,y)\in C_0\cup C_1\cup C_1' \cup C_1'' \cup C_1'''.$

Next, we consider $\psi_3=((2,1)).$ Per Step (2-2), we consider the reducibility of $u_\rho(2,1)\lvert\cdot\rvert^t\rtimes\sigma.$ The reducibility happens precisely when $t=\pm1,\pm2$ and so we have the following reducibility hyperplanes in $\mathcal{H}(\psi_3,\sigma)$.
\begin{figure}[H]
\begin{center}
\scalebox{1}{\small
\begin{tikzpicture}[>=stealth]
\draw[solid, {Circle[open]}-{Circle[open]}] (-2,0) --  (-1,0);
\draw[solid, -{Circle[open]}] (-1,0) --  (1,0);
\draw[solid, -{Circle[open]}] (1,0) --  (2,0);
\draw[solid, ->] (2,0) --  (4,0);
\draw[solid, <-] (-4,0) --  (-2,0);
\node at (-4.25,0)[black]{$D_1$} ;
\node at (-1.5,-.5)[black]{$D_2$} ;
\node at (0,-.5)[black]{$D_3$} ;
\node at (1.5,-.5)[black]{$D_4$} ;
\node at (4.25,0)[black]{$D_5$} ;
\node at (-2,.5)[black]{$-2$} ;
\node at (-1,.5)[black]{$-1$} ;
\node at (1,.5)[black]{$1$} ;
\node at (2,.5)[black]{$2$} ;
\end{tikzpicture}}
\end{center}
\caption{Reducibility hyperplanes of $u_\rho(2,1)\vert\cdot\vert^t\rtimes\sigma$}\label{Figure components corank 2 (2,1) case}
\end{figure}
Consequently, we add $(D_i,\psi_3,\sigma)$ to $R(\sigma,\rho,2,0)$ for $i=1,\dots,5.$ Note that only three components are bounded. The component $D_3$ intersects the point $x=0.$ From Step (3-2), we have that $D_3\sim(\R^0,\emptyset,\sigma).$ However, the components $D_2, D_4$ are only equivalent to themselves (note that $\half{3}+\half{1+2}\in\mathbb{Z}$). Thus we find that 
\[
u_\rho(2,1)\lvert\cdot\rvert^x\rtimes\sigma \in\Pi_{\overline{A}}(\sigma,\rho,2),
\]
for $x\in D_3.$ The case for $\psi_4=((1,2))$ is similar, which we omit here.

\section{Verification of Conjecture~\ref{conj A+ intro} for \texorpdfstring{$G_2$}{}}\label{sec:G2}

In this section, we verify that Conjecture~\ref{conj A+ intro} holds for \( G_2 = \RG_2(F) \), the group of \( F \)-points of the exceptional group of type \( \RG_2 \). This verification is based on Mui{\'c}'s classification of the unitary dual of \( G_2 \) in~\cite{Mu97}. The local Langlands correspondence for \( G_2 \) was recently established in~\cite{AX22} and~\cite{GS23b}. Although the construction of the Arthur packets \( \Pi_A(G_2) \) via the trace formula and endoscopic character identities remains incomplete, many examples of Arthur packets for \( G_2 \) have been constructed by other methods; see~\cite{GGJ02, GG05, GG06, GS23a, AHRR24} for several such cases.
For our purposes, to construct \( \Pi_u(G_2) \), we do not need the full union of Arthur packets \( \Pi_A(G_2) \), but only
a specific subset, as described below. Furthermore, in the case of $G_2$, we will show that $\Pi_{\ol{A}}^{\lim}(G_2)=\Pi_{\ol{A}}(G_2)=\Pi_u(G_2),$ i.e., we do not need to consider the limit.

To proceed, we recall some notation related to the group $G_2$ from \cite{Mu97}. We fix a Borel subgroup $B=TU$ of $G_2$. The two simple roots of $G_2$ relative to $B$ are denoted by $\alpha,\beta$, with $\alpha$ being short and $\beta$ long. For $\gamma\in \{\alpha,\beta\}$, let $P_\gamma=M_\gamma N_\gamma$ be the parabolic subgroup such that the root space of $\gamma$ is in the Levi $M_\gamma$. We identify $T\cong F^\times \times F^\times$ by 
$t\mapsto ((2\alpha+\beta)(t),(\alpha+\beta)(t)).$ For two unitary characters $\chi_1,\chi_2$ of $F^\times$, and $s_1, s_2 \in \mathbb{R}$, let $I(s_1,s_2,\chi_1,\chi_2):=\mathrm{Ind}_{TU}^{G_2}(\chi_1\lvert\cdot \rvert^{s_1}\otimes \chi_2\lvert\cdot \rvert^{s_2})$ be the induced representation. If $s_1>s_2>0$, let $J(s_1,s_2,\chi_1,\chi_2)$ be the Langlands quotient of $I(s_1,s_2,\chi_1,\chi_2)$. Similarly, for $\gamma\in\{\alpha,\beta\}$, let $I_\gamma(s,\tau)=\mathrm{Ind}_{P_\gamma}^{G_2}(\tau \lvert\det(\cdot)\rvert^s)$ for a representation $\tau$ of $\GL_2(F)$. If $\tau$ is tempered and $s>0$, let $J_\gamma(s,\tau)$ be the Langlands quotient of $I_\gamma(s,\tau)$. For a character $\chi$ of $F^\times$, let $\delta(\chi)=(\chi \circ\det)\otimes \St_{\GL_2(F)}$. For characters $\chi_1, \chi_2$ of $F^{\times}$, let $\pi(\chi_1,\chi_2)$ be the induced representation $\mathrm{Ind}_{B_{\GL_2(F)}}^{\GL_2(F)}(\chi_1\otimes \chi_2)$. Applying the Weyl elements' actions (see \cite[page 467]{Mu97}), in the Grothendieck group of $G_2$, for characters $\chi_1, \chi_2$ of $F^{\times}$, we have that 
$\mathrm{Ind}_{TU}^{G_2}(\chi_1\otimes\chi_2)=\mathrm{Ind}_{TU}^{G_2}(\chi_2\otimes\chi_1)=\mathrm{Ind}_{TU}^{G_2}(\chi_1\chi_2\otimes\chi_2^{-1}).$
Moreover, applying the isomorphisms of $\GL_2 $ with the Levi subgroups $M_\gamma$, one can obtain that 
$$I_\alpha(s,\pi(\chi_1\lvert\cdot\rvert^{s_1},\chi_2\lvert\cdot\rvert^{s_2}))=I(s+s_1,s+s_2,\chi_1,\chi_2), ~ I_\beta(s,\pi(\chi_1\lvert\cdot \rvert^{s_1},\chi_2\lvert\cdot \rvert^{s_2}))=I(s+s_2,s_1-s_2,\chi_2,\chi_1\chi_2^{-1}).$$
Thus, in the Grothendieck group, we have that
$I(s,s-1,\chi,\chi)=I_\alpha(s-1/2,\delta(\chi))+I_\alpha(s-1/2,\chi\circ\det),$ and $I(s-1,1,\chi,1_{F^\times})=I_\beta(s-1/2,\delta(\chi))+I_\beta(s-1/2,\chi\circ\det).$

We recall the following result of Mui\'c (\cite{Mu97}) on the classification of non-tempered unitary representations of $G_2$.  

\begin{thm}[Theorems 5.1, 5.2, 5.3 of \cite{Mu97}]\label{thm-Muic} 
Non-tempered unitary representations of $G_2$ are given in the following list.

\begin{enumerate}
\item[(a)] Suppose that $\chi_1,\chi_2$ are unitary characters and $s_1>s_2>0$ are real numbers. Then $J(s_1,s_2,\chi_1,\chi_2)$ is unitary if and only if one of the following conditions holds:
\begin{enumerate}
\item[(a1)] $\chi_1=\chi_2=1_{F^\times}$, $(s_1,s_2) \in \{ 2s_1+s_2\le 1\}\cup \{ s_1+2s_2\ge 1, s_1+s_2\le 1\}\cup\{s_1=2,s_2=1\};$
\item[(a2)] $\chi_1=1_{F^\times}, \chi_2$ has order $2$, $s_1+2s_2\le 1;$\footnote{In this case, \cite[Theorem 5.2]{Mu97} states that the representation $J(s_1,s_2,1_{F^\times},\chi_2)$ is also unitary when $0<s_2<s_1=1$. This seems to be a typo because $J(1,s_2,1_{F^\times},\chi_2)=I_\beta(s_2+1/2, \chi\circ \det) $ which is not unitary for $0<s_2<1$ by \cite[Lemma 5.2]{Mu97}.}
\item[(a3)] $\chi_1$ has order $2$, $\chi_2=1_{F^\times}$, and $2s_1+s_2\le 1;$
\item[(a4)] $\chi_1=\chi_2$ has order $2$, and $s_1+s_2\le 1$.
\end{enumerate}
\item[(b)] Suppose that $\chi$ is a unitary character. Let $\gamma\in \{\alpha,\beta\}$. Let $\pi$ be a tempered non-supercuspidal representation of $\GL_2(F)$. Then for $s>0$, $J_\gamma(s,\pi)$ is unitary iff one of the following conditions holds:
\begin{enumerate}
\item[(b1)] $\pi=\delta(\chi)$ with $\chi^2=1_{F^\times}$, and $s\le 1/2$;
\item[(b2)] $\gamma=\alpha, \pi=\pi(\chi,\chi^{-1})$, and $s\le 1/2;$
\item[(b3)] $\gamma=\alpha, \pi=\pi(1_{F^\times},\chi),$ $\chi$ has order $2$, and $s\le 1/3$;
\item[(b4)] $\gamma=\beta$, $\pi=\pi(\chi,\chi^{-1})$, and $s\le 1/2;$
\item[(b5)] $\gamma=\beta,\pi=\pi(\chi,\chi^{-1}),\chi^3=1_{F^\times}$, and $s=1$;
\item[(b6)] $\gamma=\beta,\pi=\pi(1_{F^\times},\chi)$, $\chi$ has order $2$, and $s\le 1$.
\end{enumerate}
\item[(c)] Suppose that $\rho$ is an irreducible supercuspidal representation of $\GL_2(F)$, $s>0$. For $\gamma\in\{\alpha,\beta\}$, the Langlands quotient $J_\gamma(s,\rho)$ is unitary iff $\rho\cong \rho^\vee$ and one of the following conditions holds:
\begin{enumerate}
\item[(c1)] The central character $\omega_\rho$ of $\rho$ is trivial and $s\le 1/2;$
\item[(c2)] $\gamma=\beta, \mathrm{Im}(\phi_\rho)\cong \FS_3$ (the symmetric group on 3 letters) and $s\le 1$. Here $\phi_\rho: W_F\to \GL_2(\BC)$ is the Langlands parameter of $\rho$.
\end{enumerate}
\end{enumerate}
\end{thm}

Let $\Pi_{temp}(G_2)$ be the set of irreducible tempered representations of $G_2$ and let 
$$\widehat \Pi_{temp}(G_2)=\{\widehat\pi\mid \pi\in \Pi_{temp}(G_2)\}.$$
Moreover, let $S_1$ be the subset of $\Pi(G_2)$ consisting of representations of the form
$J_\gamma(1/2,\delta(\chi)), $ for $\gamma\in \{\alpha,\beta\}$, where $\chi$ is any unitary character of $F^\times$ with $\chi^2=1_{F^\times}$.
Note that these representations appear in \cite[Propositions 4.1, 4.3]{Mu97}.
Finally, we set $$S=\Pi_{temp}(G_2)\cup \widehat \Pi_{temp}(G_2)\cup S_1.$$
The representations in $S$ are listed in \cite[Table 1]{CK25}, from which we can record the following.

\begin{cor}\label{cor-G2}
The set of representations $S$ has the following properties:
\begin{enumerate}
\item $S\subset \Pi_u(G_2);$
\item The set $S$ is stable under the Aubert-Zelevinsky involution;
\item For each $\pi\in S$, its $L$-parameter $\phi_\pi$ is of Arthur type. 
\end{enumerate}
\end{cor}
\begin{proof} Part (1) and (2) can be checked directly from \cite[Table 1]{CK25} using the results of \cite{Mu97}. We remark that the following representations in \cite[Table 1]{CK25} are neither tempered nor anti-tempered and are also not in $S$: $ J_Q(5/2,\delta(1_{F^\times}))$ and $J_P(3/2,\delta(1_{F^\times}))$ in rows 3 and 4; and $J_P(1/2,\delta(\chi^{\pm1}))$ for a unitary character $\chi$ of order 3 in rows 16 and 17. 
Note that $J_P, J_Q$ are $J_{\alpha}, J_{\beta}$ here, respectively. 

Part (3) can also be checked directly. In fact, if $\pi$ is a subquotient of $ I(1,0,1_{F^\times},1_{F^\times})$, then $\phi_\pi$ is of Arthur type by \cite[Lemma 2.10]{CFZ21}. If $\pi$ is a subquotient of $I(1,0,\chi,\chi)$ with $\chi^2=1_{F^\times}$, then $\phi_\pi$ is of Arthur type by the calculation in \cite[\S2.6, Case (4)]{CFZ21}. See also \cite[Table 4.1.1]{CFZ21}. For the rest cases, $\pi$ is either in $\Pi_{temp}$ or in $\widehat{\Pi}_{temp}$. Then one can check case by case that the $L$-parameters of these anti-tempered representations are all of Arthur type by \cite[Table 1]{CK25} and the list of unipotent $L$-parameters of Arthur type in \cite{CFZ21}.
\end{proof}

 Following Definition \ref{def closure A}, starting from the set $S$, we can similarly define a set $\Pi_{\overline{S}}(G_2)=\bigcup_{k\ge 0}\Pi_S^{(k)}(G_2)$, where the set $\Pi^{(0)}_{S}(G_2)=S$.
Assuming that there is a theory of local Arthur packets as in \cite[Conjecture 6.1]{Art89}, Corollary \ref{cor-G2} shows that we should have $\Pi_S^{(0)}(G_2)\subset \Pi_A(G_2)$. Hence, $\Pi_{\overline{S}}(G_2)\subset \Pi_{\overline A}(G_2)$. Similarly, we define $\Pi_{\overline{S}}^{\lim}(G_2)$. 

\begin{prop}\label{proposition:G2}
We have that $\Pi_{\overline{S}}(G_2)=\Pi_{\overline{S}}^{\lim}(G_2)=\Pi_u(G_2)$. In particular,
 Conjecture \ref{conj A+ intro} holds for the exceptional group $G_2$ if the local Arthur conjecture as in \cite[Conjecture 6.1]{Art89} holds.
 \end{prop}

 \begin{proof}
 By Definition, $\Pi_{\overline{S}}(G_2)\subset \Pi_{\overline{S}}^{\lim}(G_2)\subset \Pi_u(G_2)$. For the reverse inclusion $\Pi_u(G_2) \subset \Pi_{\overline{S}}(G_2)$, we check it case by case, following Mui\'c's classification of $G_2$ unitary dual as in Theorem \ref{thm-Muic}. 

 First, suppose that $\pi$ is a unitary subquotient of $I(s_1,s_2,1_{F^\times},1_{F^\times})$, we will show that $\pi\in \Pi_{\overline{S}}(G_2)$. This will cover the unitary representations in Theorem \ref{thm-Muic} case (a1), and case $\chi=1_{F^\times}$ of (b1), (b2), (b4), and (b5).   The unitary subquotients of $I(s_1,s_2,1_{F^\times},1_{F^\times})$ are described in Figure \ref{figure 1 of G2} below:
 
 \begin{figure}[H]
 \begin{center}
\scalebox{0.7}{\small
\begin{tikzpicture}
\filldraw[fill=gray!30, dashed] (0,2.598) -- (1.5,2.598) -- (0,1.732) -- (0,2.598) -- cycle;
\filldraw[fill=gray!30, dashed] (0,0) -- (0.75,1.299) -- (0,1.732) -- (0,0) -- cycle;
\node at (0.3,1) {$C_1$};
\node at (0.4,2.3) {$C_2$};
\draw[-](0,0)--(3,0)node[above]{$\alpha$};
\draw[-](0,0)--(-4.5,2.598)node[above]{$\beta$};
\draw[-](0,0)--(-1.5,2.598);
\draw (-2,2.8) node{$\alpha+\beta$};
\draw[-](0,0)--(1.5,2.598)node[right]{};
\draw[red, very thick, -](0,0)--(0.75,1.299)node[right]{};
\draw[-](0,2.598)--(1.5,2.598)node[right]{};
\draw[-](0,0)--(4.5,2.598)node[above]{$3\alpha+\beta$};
\draw[-](0,1.732)--(0.75,1.299)node[right]{};
\draw[dashed](-1.5,2.598)-- (3.45,-0.2598) node[right]{$2s_1+s_2=1$};
\draw[dashed](5,2.598)--(-5.5,2.598)node[left]{$s_1+s_2=1$};
\draw[-](0, 2.598)--(0,3.464);
\draw (0,3.564) node{$s_1=s_2$};
\draw[red, very thick, -] (0,0)--(0,2.598);
\draw[blue, very thick, -] (0,2.598)--(1.5,2.598);
\draw[blue, very thick, -](0,1.732)--(1.5,2.598);
\draw[blue, very thick, -](0,1.732)--(0.75,1.299);
\draw[-](0,1.732)--(1.5,2.598)node[above]{};
\draw[dashed]((1.5,2.598)--(2.625, 3.2475);
\draw (3, 3.4) node {$s_1+2s_2=1$};
\draw[dashed](1.5,2.598)--(1.5,7.794)--(1.5,8)node[above]{};
\draw[dashed](-1.5,2.598)--(1.5,7.794)node[above]{};
\node at (1.5,7.794)[black]{\textbullet};
\end{tikzpicture}
}
\end{center}
\caption{Unitary subquotients of $I(s_1,s_2,1_{F^\times},1_{F^\times})$}\label{figure 1 of G2}
\end{figure}

\noindent If $\pi$ is in the open region $C_1$, then $\pi=I(s_1,s_2,1_{F^\times},1_{F^\times})$ is irreducible. Since $I(0,0,1_{F^\times},1_{F^\times})\in S$, we have that $\pi\in \Pi_{\overline{S}}^{(1)}(G_2). $ If $\pi=J(s_1,s_2,1_{F^\times},1_{F^\times})$ is in the open region $C_2$, as in the proof of \cite[Theorem 5.2]{Mu97}, we fix an $s$ with $1/3<s<1/2$ and $t$ with $0<t<s-1/3$, and consider the irreducible induced representation 
    $$ X_t=I_\beta(t, \pi_{s,1_{F^\times}})\cong I(t-s,2s,1_{F^\times},1_{F^\times})=I(s+t,s-t,1_{F^\times},1_{F^\times}),$$
where $\pi_{s,1_{F^\times}}=\pi(\lvert\cdot\rvert^s,\lvert\cdot\rvert^{-s})$ is the complementary series representation of $\GL_2(F).$\footnote{The unitary induction $I_\beta(0,\pi_{s,1_{F^\times}}), 0\le s<1/2$ is the middle red line in Figure \ref{figure 1 of G2}. The other red line segment is the unitary induction $I_\alpha(0,\pi_{s,1_{F^\times}}), 0\le s<1/2.$ The 3 blue line segments denote the reducibility lines.} Note that $\pi_{0,1_{F^\times}}$ is a unitary induction from an Arthur representation of $\GL_1(F)\times \GL_1(F).$ Thus $\pi_{0,1_{F^\times}}, \pi_{s,1_{F^\times}}\in \Pi_{\overline{A}}^{(1)}(\GL_2(F))$. Thus, $I_\beta(0, \pi_{s,1_{F^\times}})$ and $X_t$ are in $ \Pi_{\overline{S}}^{(2)}(G_2)$ for $t$ near zero. Since for $t$ near zero, $X_t$ is in the above connected region $C_2$, we conclude that the representation $J(s_1,s_2,1_{F^\times},1_{F^\times})$ when $s_1, s_2$ in the open region $C_2$ is indeed in $\Pi_{\overline{S}}(G_2).$

Now, we assume that $\pi$ is in one side of the regions $C_1$ or $C_2$. Since the proof is similar, we only check the case where $\pi$ is in the line segment $2s_1+s_2=1, 1/3\le s_1\le 1/2$. Then $\pi$ is a subquotient of  $I(s_1,1-2s_1,1_{F^\times},1_{F^\times})= I(s_1,{s_1-1},1_{F^\times},1_{F^\times})=I_\alpha(s_1-1/2,\delta(1))+I_\alpha(s_1-1/2,1_{\GL_2(F)})$. Thus $\pi=I_\alpha(s_1-1/2,1_{\GL_2(F)})$ or $\pi=I_\alpha(s_1-1/2,\delta(1))$. Since $1_{\GL_2(F)}$ and $\delta(1)$ are Arthur representations of $\GL_2(F)$, the representations $I_\alpha(0,1_{\GL_2(F)})$ and $I_\alpha(0,\delta(1))$ are in $ \Pi_{\overline{S}}(G_2)$. Then we obtain $\pi\in \Pi_{\overline{S}}(G_2)$ by definition. Finally, at the point $s_1=2, s_2=1$, the representation $\pi\in S$ is either the Steinberg representation or the trivial representation. 
  
  Next, we consider the case where a unitary representation $\pi$ is a subquotient of $I(s_1,s_2,1_{F^\times},\chi)$, where $\chi$ is a character of order $2$. Note that up to a Weyl element action, in the Grothendieck group we have 
  $$\mathrm{Ind}_{TU}^{G_2}(\chi\lvert\cdot\rvert^{s_1}\otimes \lvert\cdot\rvert^{s_2})= \mathrm{Ind}_{TU}^{G_2}(\lvert\cdot\rvert^{s_2}\otimes \chi \lvert\cdot\rvert^{s_1}), \text{ and }\mathrm{Ind}_{TU}^{G_2}(\chi\lvert\cdot\rvert^{x}\otimes \chi\lvert\cdot\rvert^{y})= \mathrm{Ind}_{TU}^{G_2}(\lvert\cdot\rvert^{x+y}\otimes \chi\lvert\cdot\rvert^{-y}) .$$
  Thus this will cover unitary representations in Theorem \ref{thm-Muic} cases (a2), (a3), (a4), and those representations in cases (b1) (b2) (b3) (b4) (b6) when  $\chi$ has order $2$. The unitary subquotients of this family are illustrated in Figure \ref{figure 2 of G2} below:

\begin{figure}[H]
\begin{center}
\scalebox{0.7}{\small
\begin{tikzpicture}
\filldraw[fill=gray!30, dashed] (0,0) -- (1.5,2.598) -- (0,1.732) -- (0,0) -- cycle;
\filldraw[fill=gray!30, dashed] (0,0) -- (0,1.732) -- (-0.75,1.299) -- (0,0) -- cycle;
\filldraw[fill=gray!30, dashed] (0,0) -- (2.25,1.299) -- (1.5,2.598) -- (0,0) -- cycle;
\node at (0.5,1.4) {$C_1$};
\node at (-0.3,1.0) {$C_2$};
\node at (1.5,1.3) {$C_3$};
\draw[-](0,0)--(3,0)node[right]{$\alpha$};
\draw[-](0,0)--(-4.5,2.598)node[left]{$\beta$};
\draw[-](0,0)--(-1.5,2.598)node[left]{$\alpha+\beta$};
\draw[-](0,0)--(1.5,2.598)node[right]{};
\draw[-](0,0)--(4.5,2.598)node[right]{$3\alpha+\beta$};
\draw[blue, very thick, -](1.5,2.598)--(-0.75,1.299);
\draw[dashed](1.5,2.598)--(2.625, 3.2475)node[above]{$s_1+2s_2=1$};
\draw[-](0,0)--(0,3.464)node[left]{$s_1=s_2$};
\draw[blue, very thick, -](1.5,2.598)--(2.25,1.299)node[right]{};
\draw[dashed](2.25,1.299)--(2.625,0.6495)node[right]{$s_1=1$};
\end{tikzpicture}
}
\end{center}
\caption{Unitary subquotients of $I(s_1,s_2,1_{F^\times},\chi)$ with $\chi$ order $2$}\label{figure 2 of G2}
  \end{figure}

In the above figure, the region $C_1$ (open regions $C_2$ and $C_3$, respectively) represents the unitary representations described in Theorem \ref{thm-Muic} case (a2) (cases (a3) and (a4), respectively). According to \cite[Proposition 3.1]{Mu97}, these three regions lie in the same connected component as in Definition \ref{def closure A}(2), which is bounded by the lines $s_1+2s_2= \pm 1,\ s_1= \pm 1$ (we only specify the first quadrant in Figure \ref{figure 2 of G2}).
They are separated in Mui{\'c}'s statement because of the requirement of the Langlands classification. Note that $(s_1,s_2)=(0,0)$ is in this connected component, which corresponds to an irreducible unitary induction.

The relevant unitary representations described in Theorem \ref{thm-Muic} case (b) are boundaries of the above connected component. More precisely, 
\begin{itemize}
\item $J_\alpha(s,\delta(\chi))$, $0\le s\le 1/2$ (in case (b1) when $\gamma=\alpha$ and $\chi$ has order 2) is the line segment $s_1+2s_2=1$ with $0\le s_1\le 1$. In fact, $J_\alpha(s,\delta(\chi))$ is a subquotient of $I_\alpha(s,\delta(\chi))$ and thus a subquotient of $I(s+1/2,s-1/2,\chi,\chi)=I(2s,1/2-s,1_{F^\times},\chi)$;
\item $J_\beta(s,\delta(\chi))$, $0\le s\le 1/2$ (in case (b1) when $\gamma=\beta$ and $\chi$ has order 2) is the line segment $s_1=1 , -1/2\le s_2\le 0$, since $ J_\beta(s,\delta(\chi))$ is a subquotient of $I(1,s-1/2,1_{F^\times},\chi)$; (thus $s_2=s-1/2$);
\item $J_\alpha(s,\pi(\chi,\chi^{-1}))$ (in case (b2) when $\chi$ has order 2) for $0\le s\le 1/2$ is the line segment $ s_1+2s_2=0, 0\le s_1\le 1$ since $ J_\alpha(s,\pi(\chi,\chi^{-1}))$ is a subquotient of $ I(2s,-s,1_{F^\times},\chi)$;
\item $J_\alpha(s,\pi(1,\chi))$ for $0\le s\le 1/3$ (case (b3)) is the line segment $ s_1=s_2,0\le s_1\le 1/3$ since $J_\alpha(s,\pi(1,\chi)) $ is a subquotient of $I_\alpha(s,\pi(1,\chi))=I(s,s,1_{F^\times},\chi)$;
\item $J_\beta(s,\pi(\chi,\chi^{-1}))$ (in case (b4) when $\chi$ has order 2) $0\le s\le 1/2$ is the line segment $s_1=0, 0\le s_2\le 1/2,$ since $J_\beta(s,\pi(\chi,\chi^{-1})) $ is a subquotient of $I_\beta(s,\pi(\chi,\chi^{-1}))=I(0,s,1_{F^\times},\chi)$;
\item $J_\beta(s,\pi(1,\chi))$ for $0\le s\le 1$ (in case (b6)) is the line segment $s_2=0, 0\le s_1\le 1$, since $J_\beta(s,\pi(1,\chi)) $ is a subquotient of $I_\beta(s,\pi(1,\chi))=I(s,0,1_{F^\times},\chi)$.
\end{itemize}
Moreover, the representations corresponding to the 4 vertices are all in $S$. More precisely, we have the following: 
\begin{itemize}
\item When $s_1=s_2=0$, it represents the irreducible unitary induction $ I(0,0,1_{F^\times},\chi)$;
\item When $s_1=0, s_2=1/2$, it represents the two irreducible subquotients $I_\alpha(0,\delta(\chi))$ and $I_\alpha(0,\chi\circ\det)$;
\item When $s_1=1,s_2=-1/2$, it represents the two irreducible subquotients $I_\beta(0,\delta(\chi))$ and $I_\beta(0,\chi\circ\det);$
\item When $ s_1=1,s_2=0$, it represents the 4 irreducible subquotients $\pi(\chi), J_\alpha(1/2,\delta(\chi)), J_\beta(1/2,\delta(\chi))$ and $J_\beta(1,\pi(1,\chi))= \widehat{\pi(\chi)}$ of $I(1,0,1_{F^\times},\chi) $, where $\pi(\chi)$ is the unique tempered subquotient of $I(1,0,1_{F^\times},\chi)$, see \cite[Proposition 4.1]{Mu97} for details.
\end{itemize}

Now we can check if a unitary representation $\pi$ is a subquotient of $I(s_1,s_2,1_{F^\times},\chi)$ with $\chi$ order $2$, then $\pi\in\Pi_{\overline{S}}(G_2)$. If $\pi$ is in the open region $ C_i$ for $i=1,2,3$, then $\pi=I(s_1,s_2,1_{F^\times},\chi) $ is irreducible. Since $I(0,0,1_{F^\times},\chi)\in S$, we have $\pi\in \Pi_{\overline{S}}^{(1)}(G_2). $ If $\pi$ is in the line segment $s_1+2s_2=1, 0\le s_1< 1 $, we have $\pi=J_\alpha(s,\delta(\chi))=I_\alpha(s,\delta(\chi))$ for $0\le s=s_1/2<1/2$. Note that the representation $ I_\alpha(s,\delta(\chi)$ is irreducible by \cite[Theorem 3.1]{Mu97}. Since $ I_\alpha(0,\delta(\chi))\in \Pi_{\overline{S}}^{(1)}(G_2)$, we get $ \pi=J_\alpha(s,\delta(\chi))\in \Pi_{\overline{S}}^{(1)}(G_2)$. The rest cases can be checked similarly.

If $\pi$ is a unitary representation of $G_2$ which is not in the above two families, it is easier to check that $\pi\in \Pi_{\overline{S}}(G_2)$. We only give details for one example and omit those for the rest. We consider the case where $\pi$ is a subquotient of $I_\beta(s,\pi(\chi,\chi^{-1}))$ for a character $\chi$ of order $3$, then by Theorem \ref{thm-Muic}, $\pi$ is either tempered, or $\pi=J_\beta(s,\pi(\chi,\chi^{-1}))$ with $0<s\le 1/2$ or $s=1$. Note that $J_\beta(1,\pi(\chi,\chi^{-1})\in \widehat\Pi_{temp}\subset S$ by \cite[Table 2]{CK25}, we only need to consider the case where $0<s\le 1/2$. Notice that $J_\beta(1/2,\pi(\chi,\chi^{-1})=I_\alpha(0,\chi\circ\det)\in \Pi_{\overline{S}}^{(1)}(G_2)$ since $\chi\circ\det\in \Pi_A(\GL_2(F))$. Thus $J_\beta(s,\pi(\chi,\chi^{-1}))\in \Pi_{\overline{S}}(G_2)$ by definition.

The proof of the rest cases are similar and we omit the details here. This completes the proof of the proposition.
\end{proof}

Since the set $S$ is preserved by the
Aubert-Zelevinsky involution, as in the cases of classical groups, we also have a direct corollary as follows. We omit the details here.  

\begin{cor}\label{cor-dual-G2}
The unitary dual $\Pi_u(G_2)$ of the exceptional group $G_2$ is preserved under Aubert-Zelevinsky involution.
\end{cor}

{It was checked in \cite[Section 4]{Qin25} that many cases in $\Pi_u(G_2)$ are preserved by Aubert-Zelevinsky involution, using different methods.}

\end{document}